\documentclass{article}
\usepackage{graphicx} 
\usepackage{fullpage}
\usepackage[utf8]{inputenc}
\usepackage{amsmath,amsfonts,amsthm,amssymb}
\usepackage{url}
\usepackage{hyperref}
\usepackage{color}
\usepackage{enumitem}
\usepackage{lineno}
\usepackage{subfiles}
\usepackage{mathrsfs}
\usepackage{caption}
\usepackage{algorithm}
\usepackage{algpseudocode}
\usepackage{varwidth}
\usepackage{mathabx}
\usepackage{svg}
\usepackage{booktabs}
\usepackage{multirow}
\usepackage{multicol}
\usepackage{subcaption}
\usepackage{here}

\usepackage{tikz}
\usepackage{tkz-graph}
\usetikzlibrary{shapes.geometric}
\usetikzlibrary{positioning}
\usetikzlibrary{decorations.markings}
\usetikzlibrary{arrows.meta}
\usetikzlibrary{calc}

\newtheorem{thm}{Theorem}[section]
\newtheorem{lem}[thm]{Lemma}
\newtheorem{proposition}[thm]{Proposition}
\newtheorem{claim}[thm]{Claim}
\newtheorem{coro}[thm]{Corollary}
\newtheorem{dfn}[thm]{Definition}

\newtheorem{fact}[thm]{Fact}

\newtheorem{conjecture}{Conjecture}

\def\showfiglabel#1{}
\def\showlabel#1{}

\algnewcommand{\Break}{\State \textbf{break} }
\algnewcommand{\Continue}{\State \textbf{continue} }
\algnewcommand{\True}{\textbf{true}}
\algnewcommand{\False}{\textbf{false}}

\title{Three-edge-coloring projective planar cubic graphs:\\ 
A generalization of the Four Color Theorem}

\begin{document}
\author{
Yuta Inoue\thanks{The University of Tokyo, Tokyo, Japan, \texttt{yutainoue@is.s.u-tokyo.ac.jp}.  Supported by JSPS Kakenhi 22H05001, by JST ASPIRE JPMJAP2302 and by Hirose Foundation Scholarship.}
\and
Ken-ichi Kawarabayashi\thanks{National Institute of Informatics \& The University of Tokyo, Tokyo, Japan, \texttt{k\_keniti@nii.ac.jp}.  Supported by JSPS Kakenhi 22H05001, by JSPS Kakenhi JP20A402 and by JST ASPIRE JPMJAP2302.}
\and
Atsuyuki Miyashita\thanks{The University of Tokyo, Tokyo, Japan, \texttt{miyashita-atsuyuki869@is.s.u-tokyo.ac.jp}.  Supported by JSPS Kakenhi 22H05001 and by JST ASPIRE JPMJAP2302.}
\and
Bojan Mohar\thanks{Simon Fraser University, Burnaby, BC, Canada, \texttt{mohar@sfu.ca}. Supported in part by NSERC Doscovery Grant R832714 (Canada),
by the ERC Synergy grant (European Union, ERC, KARST, project number 101071836),
and by the Research Project N1-0218 of ARIS (Slovenia).}
\and
Tomohiro Sonobe\thanks{National Institute of Informatics \& The University of Tokyo, Tokyo, Japan, \texttt{tomohiro\_sonobe@nii.ac.jp}.  Supported by JSPS Kakenhi 22H05001 and by JST ASPIRE JPMJAP2302.}
 }
\date{\today}
\maketitle

\begin{abstract}
We prove that every cyclically 4-edge-connected cubic graph that can be embedded in the projective plane, with the single exception of the Petersen graph, is 3-edge-colorable. 
In other words, the only (non-trivial) snark that can be embedded in the projective plane is the Petersen graph. 

This implies that a 2-connected cubic (multi)graph that can be embedded in the projective plane is not 3-edge-colorable if and only if it can be obtained from the Petersen graph by replacing each vertex 
by a 2-edge-connected planar cubic (multi)graph. 
Here, a replacement of a vertex $v$ in a cubic graph $G$ is the operation that takes a 2-connected planar (cubic) multigraph $H$ containing some vertex $u$ of degree 3, unifying $G-v$ and $H-u$, and connecting the vertices in $N_G[v]$ in $G-v$ with the three neighbors of $u$ in $H-u$ with 3 edges. Any graph obtained in such a way is said to be \emph{Petersen-like}.
   
This result is a nontrivial generalization of the Four Color Theorem, and its proof requires a combination of extensive computer verification and computer-free extension of existing proofs on colorability. 

Using this result, we obtain the following algorithmic consequence. 

\begin{quote}
{\bf Input}: A cubic graph $G$.\\
{\bf Output}: Either a 3-edge-coloring of $G$, an obstruction showing that $G$ is not 3-edge-colorable, or the conclusion that $G$ cannot be embedded in the projective plane (certified by exposing a forbidden minor for the projective plane contained in $G$).\\
{\bf Time complexity}: $O(n^2)$, where $n=|V(G)|$.
\end{quote}
An unexpected consequence of this result is a coloring-flow duality statement for the projective plane: A cubic graph embedded in the projective plane is 3-edge-colorable if and only if its dual multigraph is 5-vertex-colorable. Moreover, we show that a 2-edge connected graph embedded in the projective plane admits a nowhere-zero 4-flow unless it is Peteren-like (in which case it does not admit nowhere-zero 4-flows). This proves a strengthening of the Tutte 4-flow conjecture for graphs on the projective plane.

Some of our proofs require extensive computer verification. The necessary source codes, together with the input and output files and the complete set of more than 6000 reducible configurations are 
available on Github\footnote{\url{https://github.com/edge-coloring}. Refer to the ``README.md'' file in each directory for instructions on how to run each program. Moreover, we provide pseudocodes for all our computer verifications, see Section \ref{sect:code}}, which can be considered as an Addendum to this paper.
\end{abstract}

\section{Introduction}

\subsection{The four color theorem, snarks, and our main result}

The four-color theorem (4CT) states that every loopless planar graph is 4-colorable. 
It was first shown by Appel and Haken \cite{4ct1,4ct2} in 1977, and a simplified version of the proof was given in 1997 by Robertson et al.~\cite{RSST}. The latter proof was used by the same authors \cite{RSST-STOC} to obtain a quadratic-time algorithm for 4-coloring planar graphs. 

Looking outside the set of planar graphs, it seems extremely difficult to extend the 4CT to larger families of graphs.
Indeed, it has been known that for every surface with positive Euler genus there are infinitely many \emph{$5$-color-critical graphs}\footnote{A graph $G$ is said to be \emph{$k$-color-critical} if $G$ is not $(k-1)$-colorable, but every proper subgraph of $G$ is.} that can be embedded in the surface \cite{MT}. 
This means that while there are no 5-color-critical planar graphs, there are infinitely many 5-color-critical graphs on the projective plane. When regarding vertex-coloring, there is a major gap between planar graphs and graphs of higher genus.

On the other hand, the 4CT has another equivalent formulation by taking planar duality:

\begin{thm}\label{thm:3ECplanar}
   Every $2$-connected cubic planar graph is $3$-edge-colorable. 
\end{thm}

While the Four Color Theorem has no extension to other surfaces, the version of Theorem \ref{thm:3ECplanar} has some chance to hold under some additional assumptions. This led Tutte in 1966 to propose his famous \emph{$4$-Flow Conjecture}.

\begin{conjecture}[Tutte, \cite{tutte}]\label{conj:Tutte4Flow}
   Every $2$-connected graph with no Petersen graph minor admits a nowhere-zero $4$-flow. 
\end{conjecture}

While this conjecture is still open, its special case restricted to cubic graphs was proved in 1997 by Robertson, Seymour, and Thomas \cite{3edgecoloring}. The proof of this major achievement is based on two nontrivial results, one about doublecross graphs \cite{doublecross}, and the other one about apex cubic graphs. (The full proof of the latter case was not published.)

For cubic graphs, having a nowhere-zero $4$-flow is equivalent to having a 3-edge-coloring. 
Motivated by questions about colorings and flows for graphs on surfaces, there is an ongoing effort to formulate which cubic graphs that can be embedded in a fixed surface are 3-edge-colorable. 
This question can be reduced to cyclically 4-edge-connected graphs.\footnote{A graph is \emph{cyclically $k$-edge-connected} if after deleting fewer than $k$ edges from the graph, there is at most one component that contains a cycle. See also Section \ref{sect:minimal}.}
Cyclically 4-edge-connected cubic graphs with girth of at least 5 that are not 3-edge-colorable are called \emph{snarks}. 
The most famous snark is the Petersen graph, which was the first snark discovered \cite{petersen-1898}. 
Snarks are an interesting and important class of graphs because they often appear as counterexamples to various intriguing conjectures.

A well-known conjecture by Gr\"{u}nbaum from 1968 \cite{grunbaum} asserts that every cubic graph with a polyhedral embedding (i.e., graphs whose dual triangulation is a simple graph) in an orientable surface is 3-edge-colorable. 
This conjecture was disproved by Kochol \cite{kochol}, who found a counterexample of genus 5. 
However, whether the conjecture holds for surfaces of smaller Euler genus is still unresolved.
Vodopivec \cite{Vodopivec} proved that there are infinitely many snarks that can be embedded in the torus. 
They are certain dot products of copies of the Petersen graph; each of them contains two edges whose removal yields a planar graph. 
It is possible that the provided examples are the only ones, and the conjecture of Gr\"unbaum \cite{grunbaum} claims precisely that.

\begin{conjecture}[Gr\"unbaum \cite{grunbaum}, 1968]\label{conj:grunbaum torus}
    If $G$ is a $2$-connected cubic graph embedded in the torus and $G$ is not $3$-edge-colorable, then $G$ contains two edges whose removal gives a planar graph. 
\end{conjecture}

Returning our attention to the projective plane, Mohar conjectured in 2004 (see \cite{Mohar-snarksPP}) that the only snark that is embeddable on the projective plane is the Petersen graph (Indeed, in the mid-1990, Neil Robertson already considered this question, and he made a conjecture which is similar to Conjecture \ref{conj:grunbaum torus} for any surface, see Problem 5.5.19 in \cite{MT}).  The main result of this paper verifies this conjecture which is tightly related to Gr\"unbaum's conjecture for the projective plane.

\begin{thm}\label{mainth}\showlabel{mainth}
    The only snark embeddable in the projective plane is the Petersen graph. 
    In other words, a $2$-connected cubic graph $G$ that can be embedded in the projective plane is not $3$-edge-colorable if and only if $G$ can be obtained from the Petersen graph by replacing each vertex by a $2$-connected planar cubic graph.
\end{thm}    

By \emph{replacing a vertex $v$} in a cubic graph $G$ by a planar cubic graph $H$, we mean selecting a vertex $u\in V(H)$, then taking the disjoint union of $G-v$ and $H-u$ and add three edges joining vertices of degree 2 in $G-v$ with those in $H-u$. 
This operation can be conducted repeatedly and can also involve a similar operation of replacing an edge of $G$ with $H$ (where we join $G-r$ and $H-f$ for some edge $f\in E(H)$). It is easy to see that making successive replacements can be done with a single step, so there is no need to mention that.
    

Theorem \ref{mainth} implies a strengthening of the Tutte 4-Flow Conjecture (Conjecture \ref{conj:Tutte4Flow}).

\begin{coro}
    Every cyclically-$4$-edge-connected projective planar graph different from the Petersen graph admits a nowhere-zero $4$-flow\footnote{For the precise definition of nowhere-zero $4$-flow, we refer the reader to the textbook by Diestel \cite{Diestel_book}.}.
\end{coro}

Using Theorem \ref{mainth}, we obtain the following algorithmic outcome.

\begin{thm}\label{algo}\showlabel{algo}
The following algorithmic question can be solved in quadratic time, $O(n^2)$, where $n=|V(G)|$.\\
{\bf Input}: A cubic graph $G$.\\
{\bf Output}: Either a 3-edge-coloring of $G$ or an obstruction showing that $G$ is not 3-edge colorable, or the conclusion that $G$ cannot be embedded in the projective plane (certified by exposing a forbidden minor for embedding in the projective plane that is contained in $G$).
\end{thm}

The coloring-flow duality that holds for planar graphs no longer holds for graphs on any other surface. However, an unexpected corollary of our Theorem \ref{mainth} is that there is a similar duality on the projective plane. Namely, the following result is equivalent to Theorem \ref{mainth}.

\begin{thm}\label{thm:duality on pp}
    A cubic graph embedded in the projective plane is $3$-edge-colorable if and only if its dual is vertex $5$-colorable.
\end{thm}

\begin{proof}
    Let $G^*$ be the dual graph with respect to an embedding of $G$ in the projective plane. 
    If $G$ has a bridge, $G^*$ has a loop, hence neither $G^*$ is 5-colorable nor $G$ is 3-edge-colorable. Hence we may assume that $G$ is 2-connected. 
    It follows by Theorem \ref{mainth} that $G$ is not 3-edge-colorable if and only if $G^*$ contains $K_6$ as a subgraph. 
    (This fact, which is not hard to show and whose verification is left to the reader, is based on the fact that the dual of the Petersen graph is $K_6$.) 
    On the other hand, it is known (see \cite{MT}) that the only 6-color-critical graph on the projective plane is $K_6$.  
\end{proof}

The above proof shows that Theorem \ref{mainth} implies Theorem \ref{thm:duality on pp}.
The reverse implication is even easier, and the proof is omitted.

Theorem \ref{thm:duality on pp} is indeed a surprising outcome since it was believed that there is no coloring-flow duality for graphs on nonorientable surfaces (see \cite{MT} or \cite{DGMVZ}).

Our Theorem \ref{mainth} yields another result about 4-flows in more general (not necessarily cubic) graphs. It can be considered as a very strong version of the Tutte 4-flow Conjecture. For that effect, we extend the definition of \emph{Petersen-like} graphs to allow vertices of higher degrees. We start with the Petersen graph and then we replace some of its vertices by planar graphs $H$ that need not be 3-regular; we only request that the vertex $u$ of $H$ that is split to connect with 3 edges to the Petersen graph is of degree 3.

\begin{thm}\label{thm:nz4flow}
    A 2-connected graph embedded in the projective plane admits a nowhere-zero $4$-flow if and only if it is not Petersen-like.
\end{thm}

The proof of Theorem \ref{thm:nz4flow} is given in the appendix.



\subsection{Sketch of the proof: Technical difficulty over the proof of 4CT}

The basic outline of our proof of Theorem \ref{mainth} is similar to the proof of 4CT \cite{RSST} and its extension to the doublecross case in \cite{doublecross}. A description of these proofs is given in \cite{doublecross}, and we omit it here. 

On the other hand, there are essential differences between the projective plane case and the planar case. 
We list below some of the main issues pertaining to the projective plane. 

\begin{itemize}
    \item[1.] The projective plane has smaller Euler characteristic, so the total charge used in the discharging part of the proof is smaller. 
    \item[2.] It is much harder to obtain a set of valid discharging rules and reducible configurations.
    \item[3.] The Petersen graph, a snark, is embeddable in the projective plane. By removing a configuration that is otherwise reducible, one can obtain the Petersen graph or a simple extension of it that is not 3-edge-colorable.
    \item[4.] Two vertices in a reducible configuration may ``wrap around'' the crosscap, admitting small representativity.
    \item[5.] Reducibility of configurations on the projective plane is much harder to achieve. Most of the smaller configurations that are reducible when considering 4-coloring of planar graphs are not reducible for the projective plane.  
    \item[6.] One needs reducible configurations whose diameter is at most five.  
\end{itemize}

The first issue is regarding the \emph{discharging method}, which we discuss in Section \ref{sect:discharging}.
With a cubic graph $G$ polyhedrally embedded to some surface, one can define the dual graph $G'$ of $G$, which is a triangulation\footnote{A \emph{triangulation} is a simple graph on a surface, in which all faces are triangles.}.
The average degree of $G'$ plays an important role in the discharging method. 
Since the Euler characteristic of the projective plane is positive, one can show that the average degree of $G'$ is strictly smaller than $6$, and the discharging method can still be conducted like in the planar case.

The second issue is far more troublesome than the first one.
In Section \ref{sect:discharging}, we discuss the reducibility of configurations.
Roughly speaking, configurations are subgraphs in cubic graphs, and a configuration in $G$ is \emph{reducible} if it can be replaced by a smaller one (by deleting edges and suppressing the endpoints of one or more edges) without changing the non-3-edge-colorability of $G$. 
There are two important differences from the planar case that make the discovery and treatment of reducible configurations much more difficult.
Firstly, it is much harder to prove that the smaller graph is not 3-edge-colorable (see the discussion below), and secondly, there is a possibility that we obtain a Petersen-like graph (see the definition below) and may thus come from a 3-edge-colorable graph to one that is not 3-edge-colorable.

When checking for reducibility of a configuration, we utilize \emph{Kempe chains}, which are sets of disjoint paths connecting the outskirts of the configuration, and all edges on each such path are colored with two colors only. 
Exchanging the two colors on such a path gives a new 3-edge-coloring.
Kempe chains are much more versatile on the projective plane than on the plane, and as a result, many configurations that are reducible in the planar case are non-reducible on the projective plane.
Therefore, reducible configurations are much harder to find on the projective plane than on the plane.
For example, consider the first 7 configurations in the 4CT proof.
Planar graphs and doublecross graphs both use these configurations, and this is the key that the same rules can apply to both cases.
For projective planar graphs, however, only the first one out of the seven is reducible and the rest are not. 
Even the first configuration is harder to reduce for the projective planar case (see Figure \ref{fig:55551}). 
For the planar case, no edge is necessary to delete for reduction, whereas for the projective plane case, we need to delete at least six edges, as shown in Figure \ref{fig:55551} (the figure shows the dual graph, which is a triangulation).
Deleting a lot of edges in a configuration is dangerous in general because the graph after the reduction may contain bridges, which do not satisfy the induction hypothesis.
However, we believe that it is not possible to find a proof without using this configuration, so we very carefully construct a workaround for this particular configuration, as shown in Section \ref{sect:6,7-cut} (see Lemma \ref{lem:5555}). 
One interesting fact is that we need reducible configurations shown in Figure \ref{fig:5555conf_pre} that contain the first graph as a subgraph (they consist of five vertices that are of degree at most six in the dual graph). 
Usually, it does not make sense to add such configurations for reducibility, but in our case, we use these reducible configurations because they are D-reducible (i.e., they do not require contracting any edge). 
If they appear, then the reduction would not end up with $K_6$. 
If they do not appear (but the first graph in Figure \ref{fig:55551} appears), we analyze the structure of the ring (and conclude that it is not possible to end up with $K_6$). See Lemma \ref{lem:5555} for more details. 

\begin{figure}
    \centering
    \includesvg[height=3.6cm]{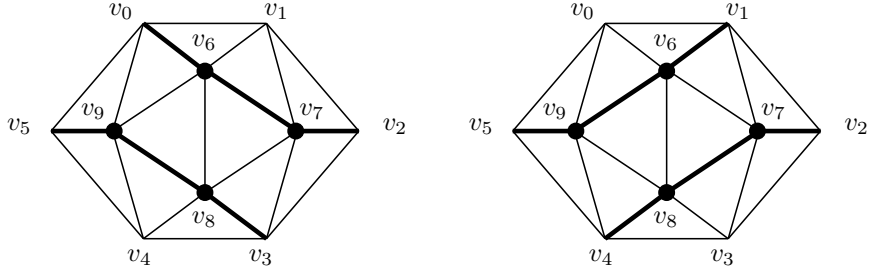}
    \caption{The first reducible configuration that consists of four vertices of degree five in the dual graph. The bold lines represent the contraction (which corresponds to edge-deletion in the cubic graph) used for C-reducibility.}
    \label{fig:55551}
\end{figure}

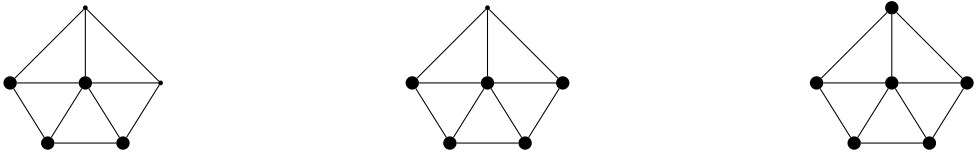
\begin{figure}[htbp]
  \tikzset{deg5/.style={thick, circle, draw, fill=black, inner sep=1.5pt,}}
  \tikzset{deg6/.style={thick, circle, draw, fill=black, inner sep=0pt,}}

\begin{tabular}{ccc}
\begin{minipage}[t]{0.3\hsize}
\centering
\begin{tikzpicture} [baseline=1.5cm]
    \node [deg5] at (0.5, 0) (a) {};
    \node [deg5] at (1.5, 0) (b) {};
    \node [deg5] at (0, 0.8) (c) {};
    \node [deg5] at (1.0, 0.8) (d) {};
    \node [deg6] at (2.0, 0.8) (e) {};
    \node [deg6] at (1.0, 1.8) (f) {};
    \foreach \u / \v in {a/b, a/c, a/d, b/d, b/e, c/d, c/f, d/e, d/f, e/f}
        \draw (\u) -- (\v);
\end{tikzpicture} 
\end{minipage} &
\begin{minipage}[t]{0.3\hsize}
\centering
\begin{tikzpicture} [baseline=1.5cm]
    \node [deg5] at (0.5, 0) (a) {};
    \node [deg5] at (1.5, 0) (b) {};
    \node [deg5] at (0, 0.8) (c) {};
    \node [deg5] at (1.0, 0.8) (d) {};
    \node [deg5] at (2.0, 0.8) (e) {};
    \node [deg6] at (1.0, 1.8) (f) {};
    \foreach \u / \v in {a/b, a/c, a/d, b/d, b/e, c/d, c/f, d/e, d/f, e/f}
        \draw (\u) -- (\v);
\end{tikzpicture}
\end{minipage} &
\begin{minipage}[t]{0.3\hsize}
\centering
\begin{tikzpicture} [baseline=1.5cm]
    \node [deg5] at (0.5, 0) (a) {};
    \node [deg5] at (1.5, 0) (b) {};
    \node [deg5] at (0, 0.8) (c) {};
    \node [deg5] at (1.0, 0.8) (d) {};
    \node [deg5] at (2.0, 0.8) (e) {};
    \node [deg5] at (1.0, 1.8) (f) {};
    \foreach \u / \v in {a/b, a/c, a/d, b/d, b/e, c/d, c/f, d/e, d/f, e/f}
        \draw (\u) -- (\v);
\end{tikzpicture}
\end{minipage} \\
\end{tabular}

  \caption{The reducible configurations used in the proof of Lemma \ref{lem:5555}.}
\label{fig:5555conf_pre}
\end{figure}

The difficulty of finding reducible configurations for the projective planar case has an impact on the overall size of the set of reducible configurations that we construct. 
As in the fifth item listed above, more than 80\% of reducible configurations for the doublecross case (hence the planar case) are not reducible for the projective plane, so we need to search for a completely different set of unavoidable reducible configurations.
Our full list of reducible configurations consists of approximately 6000 configurations\footnote{Unfortunately, we are not able to provide pictures of all reducible configurations. They are available on the afore-mentioned Google drive.}, which is much more than that for 4CT (633 configurations) or for the doublecross case (765 configurations).  

The discharging rules require a major revision as well.
Ultimately, we have constructed a list of 169 discharging rules (vs. the 33 discharging rules for planar/doublecross cases). All of the rules are described in Section \ref{sect:rule} in the appendix. 
More details are provided in Section \ref{sect:reducible configurations} and Section \ref{sect:discharging}. 

The third issue arises when trying to reduce a configuration whose reducibility is based on having three or more edges deleted.
In previous proofs, it is easy to see that the reduction of configurations is closed under the graph family of bridgeless cubic graphs.
However, in our proof, the graph family we are interested in is the family of bridgeless cubic projective planar graphs that are not Petersen-like.\footnote{A graph is \emph{Petersen-like} if it can be obtained from a Petersen graph with the operations defined in Theorem \ref{mainth}.}
It is not obvious at all whether the reduction of a reducible configuration is closed under such a graph family when three or more edge-deletions are used, and this is a problem that must be addressed in order to provide a valid induction hypothesis.

We take care of this problem by involving the following two ideas. When we delete 3 or 4 edges, we prove that any type of a reduction is safe by enumerating all possible ways that may give rise to a Petersen-like graph. (See Section \ref{sect:smallcont}.)
The original cubic graph that we want to reduce has relatively large cyclic edge connectivity, so there are only a few ways in which the three or four edges in a non-Petersen-like graph can be deleted such that the resulting graph will be Petersen-like.
We enumerate all possible cases of such reductions and prove that in each case, there is another reducible configuration that either (1) has at most two edge-deletions or (2) results in representativity at most 2. 
In both cases, the resulting graph is not Petersen-like.

When we need to delete five or more edges, we prove that the types of reductions of configurations used specifically in this paper are safe.
There are over 250 configurations that require five or more edge deletions, and we use various conditions (such as graph size and lengths of faces of $G$) to check that none of them yields a Petersen-like graph after reduction.

To give more detail, for each configuration $I \subseteq G$ (and its dual $W$ contained in the dual triangulation $G'$), we use the following process. Let $X \subseteq E(I)$, $|X|\ge 5$, be the set of edges that are deleted in the reduction.

\begin{enumerate}
    \item Enumerate all possible $(\leq 3)$-vertex-separations in $G'$, after the edge contractions (that correspond to the edge deletions in $G$). 
    
    We utilize a lemma proved in Section \ref{sect:6,7-cut}: if $C$ is a contractible cycle of length 6 or 7 in $G'$ and the number of vertices strictly contained in the disk $D$ bounded by $C$ is at least 5, we can find another reducible configuration that is contained in $D\setminus C$. 
    \item Using the separations obtained above, delete all those vertices that could be potentially removed by a 2 or 3-vertex separation reduction (that corresponds to a 2 or 3-cut reduction in $G$). 
    \item 
    Delete all vertices that are still contained in the ``small'' side of a 2 or 3-vertex cut that contains at most one vertex of the ring or exactly two vertices that are adjacent in the ring.
    \item Check if the resulting graph is of order at least 7, or has a vertex of degree 4 or of degree $\ge6$.

    If either the graph size check or the vertex degree check is satisfied, we can conclude that $G'$ would not become the complete graph $K_6$, and by considering the dual, we can confirm that $G$ would not become the Petersen graph $P_{10}$.
\end{enumerate}

We implement the above algorithm and run computer check for each reducible configuration, and almost all configurations pass the criteria. 
For a few cases that do not pass, we provide proof for each individual configuration using additional criteria. 
The full proof is shown in Section \ref{sect:largecont}.

Both, the fourth and the sixth issue, are a challenge when trying to reduce a configuration.
When dealing with reducibility, we must confirm that the outskirt of the configuration is induced, meaning that there should be no edges connecting two vertices of the configuration.
In the planar case, this can be avoided fairly easily: by the induction hypothesis, there is no separating cycle of length at most five in the triangulation (which are cuts in the dual cubic graph), except for one single vertex on one side of the cycle, and every reducible configuration is of diameter 4 or less, meaning there is no such an edge.
In our proof, it is possible that there is a cycle of short length that may be \emph{non-contractible}, and some configurations are of diameter 5 (but not 6).
Having such a short non-contractible cycle is troublesome, and we must address this issue separately.

The minimum length of such a cycle is called the \emph{representativity} (also called the \emph{face-width}) of the embedded graph, and we address graphs of representativity up to 5. The proof uses a combination of discharging methods and reducibility checking, which also involve some ideas that were developed in \cite{doublecross}.
Configurations of diameter 5 are also a problem, and for every pair of vertices $u,v$ at distance 5 in a configuration $W$, we find one of the following in $W + uv$:

\begin{itemize}
    \item A contractible cycle of length at most five (which we can exclude by applying Lemma \ref{minc}).
    \item A contractible cycle of length exactly six, with both components of $G'-C$ having at least 4 vertices (which we can exclude by using Lemma \ref{lem:6,7-cut}).
    \item A non-contractible cycle of length at most five, which we address in Section \ref{sect:rep}.
\end{itemize}

The full case analysis is presented in Section \ref{sect:dist5}.

\subsection{Some technical results and organization of the paper}\label{org}\showlabel{org}

In addition to the discharging method and the reducibility check, which are detailed in Sections \ref{sect:reducible configurations} and \ref{sect:discharging}, respectively, we utilize the following technical results.

\begin{itemize}
    \item[(a)]
    Let $C$ be a contractible cycle of length $l$ in $G'$, where $l\in\{6, 7\}$ and the disk $D$ bounded by $C$ contains at least $l-2$ vertices. Then there is a reducible configuration that is contained strictly in the interior of $D$. 

   This is shown in Sections \ref{sect:smaller} and \ref{sect:6,7-cut}.
   \item[(b)]
   Let $J$ be a projective-planar subcubic graph with all of its degree-two vertices on a face $F$ of $J$, and assume that $G = J \cup Q$, where $Q$ is a 2-connected subcubic planar graph and is embedded in $F$. Suppose, moreover, that $J$ fits one of the following cases:
    \begin{enumerate}
        \item
        $J$ is obtained from the Petersen graph by deleting a vertex and subdividing some edges on the resulting face $F$ (Section \ref{sect:smallcont}) by adding $t$ subdivision vertices, where $3\le t\le 6$, or 
       \item
        $J$ is obtained from the Petersen graph by deleting a vertex and subdividing some edges on the resulting face $F$ (Section \ref{sect:smallcont}) by adding $t$ subdivision vertices, where $3\le t\le 4$ and then also adding a new edge into one of the faces distinct from $F$ in such a way that the girth of the resulting subcubic graph $J$ is still at least 5, or 
        \item 
        $J$ is either a subdivision of $V_6$, $V_8$, or $V_{10}$, where $V_{2k}$ is obtained from the $2k$-gonal face $F$ by adding $k$ main diagonals in its outside on the projective plane (Section \ref{sect:rep}).
    \end{enumerate}
    Then $J$ is reducible if the size of the face $F$ is not too large. (The actual value of what it means being too large will be discussed later.) 
\end{itemize}

The first result (a) helps a lot for our safety check of reducibility. Indeed, in this case we have another reducible configuration $J'$ in the disk $D$ bounded by $C$ in $G$, and we manage to show that by making a reduction using $J'$, the resulting graph obtained from $G$ is not Petersen-like. More details will be given in Sections \ref{sect:smaller} and \ref{sect:6,7-cut}.  This means that if $C$ is a separating cycle of length exactly $l$ for $l=6, 7$, we may assume that the component of $G'-C$ strictly inside the disk bounded by $C$ has at most $l-3$ vertices.

In a small number of special cases we have to consider contractible cycles of length $l=8$.

These results imply that as long as we contract at most four edges for C-reducible configurations in $G$, they are safely contracted (i.e., after contractions, $G'$ never ends up with $K_6$. In other words, $G$ never ends up being the Petersen graph). 
More details are given in Section \ref{sect:smallcont}. 

We still have 201 reducible configurations that need contraction of five edges or more. 
For each of these configurations, we provide a case-by-case analysis using theorems in Sections \ref{sect:6,7-cut} and \ref{sect:rep} to prove that $G$ does not become a Petersen-like graph after contractions.
While the analysis could theoretically be carried out by hand, it is very repetitive and time-consuming to cover all different cases.
So, for the sake of simplicity, we used computer programs to cover most of the cases.
The algorithm used for the analysis is provided in Appendix \ref{sect:code}.
There are six configurations where the program did not produce valid output (that means that we could not confirm that $G$ does not become a Petersen-like graph after contractions).
In these cases, we provide a hand written proof, which we present in Lemma \ref{lem:5555} and Appendix \ref{subsec:remaining1}.


In Section \ref{sect:minimal} we shall introduce the main notation. 
Discharging method and the reducibility check are detailed in Sections \ref{sect:reducible configurations} and \ref{sect:discharging}, respectively.
Finally, in Section \ref{sect:mainproof}, we give the whole proof of Theorem \ref{mainth} and in Section \ref{sect:algorithm} we outline our quadratic-time algorithm for Theorem \ref{algo}. The algorithm presented here is actually following the same approach as developed in \cite{RSST-STOC}. 

Some of details that mostly devote to case analysis (including the arguments for Sections \ref{sect:smallcont} and \ref{sect:largecont}) or that devote to some arguments that are already known (including the argument for Section \ref{sect:algorithm}), are deferred to the appendix. 
For programs used in our computer verification, their pseudo codes are included in Section \ref{sect:code}.

\section{Minimal counterexamples}
\label{sect:minimal}
\showlabel{sect:minimal}

\begin{dfn}
    Let $G$ be a cubic graph. A \emph{cyclic cut} is a set $F \subset E(G)$ of edges disconnecting $G$, such that at least two components of $G - F$ contain cycles. The graph 
    is \emph{cyclically $k$-edge-connected} if no cyclic cut of size less than $k$ exists.
\end{dfn}

Let $\mathcal{P}_1$ be the set of all simple cubic bridgeless graphs embeddable in the projective plane.
For $G \in \mathcal{P}_1$, the replacement of any cyclic 3-edge-cut by a vertex is called a \emph{3-edge-cut reduction}.
Similarly, the replacement of any cyclic 2-edge-cut by an edge is called a \emph{2-edge-cut reduction}.
These two kinds of reductions are also referred to as \emph{low edge-cut reductions}.
Since there are exactly two connected components separated by a 2-edge or 3-edge-cut, there are two possible graphs that are obtained by making a low edge-cut reduction. 
The following lemma is easy to prove. 

\begin{lem}
    A \emph{low edge-cut reduction} for any $G \in \mathcal{P}_1$ yields one planar graph and one graph in $\mathcal{P}_1$.
    Furthermore, the following statements are equivalent for any $G \in \mathcal{P}_1$:
    \begin{itemize}
        \item $G$ is $3$-edge-colorable.
        \item For any $G_1$ obtained from $G$ by zero or more low edge-cut reductions, $G_1$ is 3-edge-colorable.
    \end{itemize}
    \label{low-red}
\end{lem}

Let $\mathcal{P}_0$ be a set of those graphs $G \in \mathcal{P}_1$ that can be reduced to a graph isomorphic to $P_{10}$ after a sequence of low edge-cut reductions. Graphs in $\mathcal P_0$ are said to be \emph{Petersen-like graphs}.

Because $P_{10}$ is not 3-edge-colorable, Lemma \ref{low-red} implies that graphs in $\mathcal{P}_0$ are not 3-edge-colorable. 
These are the obvious non-3-edge-colorable projective planar cubic graphs, and no other examples are known. 
In this paper, we prove that these are the only non-colorable graphs. In other words, any graph in $\mathcal{P}_1 \setminus \mathcal{P}_0$ can be 3-edge-colored. While trying to prove this, we consider a hypothetical counterexample and discuss how it should look like until finally reaching a contradiction.

\begin{dfn}
    A 2-connected cubic graph $G \in \mathcal{P}_1\setminus \mathcal{P}_0$ is called a \emph{counterexample} if the graph does not admit a 3-edge-coloring. A counterexample is \emph{minimal} if there are no counterexamples of smaller order in $\mathcal{P}_1 \setminus \mathcal{P}_0$.
\end{dfn}

Our main result, Theorem \ref{mainth}, is equivalent to saying that there are no counterexamples in $\mathcal{P}_1 \setminus \mathcal{P}_0$.

The following lemma, whose proof is given in Section \ref{sect:below5cuts} (in the appendix), holds.

\begin{lem}\label{minc}\showlabel{minc}
    A minimal counterexample $G \in \mathcal{P}_1 \setminus \mathcal{P}_0$ is cyclically 5-edge-connected. Furthermore, for any cyclic cut $F$ of size $5$, one of the connected components of $G - F$ must be a 5-cycle.
\end{lem}

Let $G$ be a minimal counterexample. 
Throughout the paper, we will interchangeably consider the cubic graph $G \in \mathcal{P}_1 \setminus \mathcal{P}_0$ and its dual graph $G'$ in our statements and proofs. 
Since $G$ is cubic and nonplanar, $G'$ is a triangulation on the projective plane. The graph $G'$ is loopless since otherwise $G$ would either have a cutedge or have representativity 1, in which case it would be planar. $G'$ is simple and 3-connected if $G$ is polyhedrally embedded in the projective plane. 

For other definitions and notation used in the paper, we refer to Diestel's book \cite{Diestel_book} for basic graph theory and to the monograph \cite{MT} for notions in topological graph theory. But there is one thing to mention. 
A \emph{cycle} in a graph $G$ is a set of vertices $U \subseteq V(G)$ where the induced graph $G[V]$ is a connected subgraph whose all vertices have a degree of 2.
A \emph{circuit} in a graph $G$ is a set of vertices that yield a closed walk. Thus a cycle is a circuit, but the converse is not true.

\section{Reducible configurations}
\label{sect:reducible configurations}
\showlabel{sect:reducible configurations}

In this section, we start by considering the dual graph $G'$, which is a triangulation of the projective plane. 
Let us first recall basic definitions of configurations and reducibility used in \cite{RSST}.
A \emph{near-triangulation} is a plane graph whose faces are triangles with the exception of the infinite face. 

\begin{dfn}\label{dfn:conf}\showlabel{dfn:conf}
    A \emph{configuration} $K$ consists of a near-triangulation $G(K) = (V(K), E(K))$, and a map $\gamma_K \colon V(K) \rightarrow \mathbb{Z}_{+}$ satisfying the following conditions.
    \begin{enumerate}
        \renewcommand{\labelenumi}{(\roman{enumi})}
        \item For every vertex $v \in V(K)$, $G(K) - v$ has at most two components, and if there are two, then $\gamma_K(v) = \deg_{G(K)}(v) + 2$.
        \item For every vertex $v \in V(K)$, if $v$ is not incident with the infinite face, then $\gamma_K(v) = \deg_{G(K)}(v)$, and otherwise $\gamma_{K}(v) > \deg_{G(K)}(v)$. In either case, $\gamma_K(v) \geq 5$.
        \item $K$ has ring-size $\geq 2$ where \emph{ring-size} is $\sum_{v}(\gamma_K(v) - \deg_{G(K)}(v) - 1)$, and the summation runs over all vertices $v \in V(K)$ that are on the infinite face and for which $G(K)-v$ is connected.
    \end{enumerate}
\end{dfn}

In what follows, we will have many figures showing configurations. In all of them, we will use the Heesch notation \cite{Heesch69} designating vertex degrees. To help the reader we collect the vertex-shapes in Figure \ref{fig:shapes}.

\begin{figure}[htb]
\centering
\includesvg[width=0.9\columnwidth]{degreetype.drawio.svg}
\caption{Shapes used to designate vertices of specific degrees.}
  \label{fig:shapes}
\end{figure}

\begin{dfn}\label{dfn:appear}\showlabel{dfn:apear}
    Let $G'$ be a triangulation of some surface $\Sigma$. A configuration $K$ \emph{appears} in $G'$ if 
    \begin{enumerate}
        \renewcommand{\labelenumi}{(\roman{enumi})}
        \item $G(K)$ is an induced subgraph of $G'$.
        \item Every finite face of $G(K)$ is also a face of $G'$.
        \item Every vertex $v \in V(K)$ satisfies $\gamma_K(v) = \deg_{G'}(v)$.
        \item There is a disk $D\subset \Sigma$ containing $G(K)$ in its interior, and the orientation of each oriented finite face of $G(K)$ is the same in $D$ (either all are clockwise or all are counterclockwise). 
    \end{enumerate}
\end{dfn}

\begin{dfn}\label{dfn:ring}\showlabel{dfn:ring}
    Let $K$ be a configuration, and $R = (V(R), E(R))$ be a cycle whose length is equal to the ring-size of $K$ and is disjoint from $G(K)$. A near-triangulation $S$ is a \emph{free completion} of K with \emph{ring} $R$ if
    \begin{enumerate}
        \renewcommand{\labelenumi}{(\roman{enumi})}
        \item $R$ is a cycle in $S$ that bounds the infinite face of $S$.
        \item $G(K)$ is a subgraph of $S$, induced on $V(S)\setminus V(R)$, i.e. $G(K) = S - V(R)$, every finite region of $G(K)$ is a finite region of $S$, and the infinite region of $G(K)$ includes $R$ and the infinite region of $S$.
        \item Every vertex $v$ of $S$ not in $V(R)$ has degree $\gamma_K(v)$ in S.
    \end{enumerate}
\end{dfn}

\begin{dfn}\label{dfn:island}\showlabel{dfn:island}
  An \emph{island} $I$ is a plane graph with the following properties.
  \begin{enumerate}
    \renewcommand{\labelenumi}{(\roman{enumi})}
    \item $I$ is 2-connected.
    \item Every vertex of $I$ has degree two or three.
    \item Every vertex of degree two is incident with the infinite region.
  \end{enumerate}
  An island \emph{appears} in a cubic graph $G \in \mathcal{P}_1$ if $I$ is a subgraph of $G$, and every finite face of $I$ is also a face of $G$.
\end{dfn}

For any configuration $K$, the inner dual\footnote{The \emph{inner dual} is obtained from the dual graph by deleting the vertex that corresponds to the infinite face.} of the free completion of $K$ is an island. This island is called the \emph{island of $K$} and will be denoted by $I(K)$. 

\begin{dfn}\label{dfn:match}\showlabel{dfn:match}
  Let $k$ be a positive integer. 
  A \emph{match} is an unordered pair $\{a, b\}$ of distinct integers in $\{1,\dots,k\}$.
  Two matches $\{a,b\}$ and $\{c,d\}$ with $1\le a<c<b<d\le k$ are said to \emph{overlap}.
  A \emph{signed match} is a pair $(m, \mu)$ where $m$ is a match and $\mu = \pm 1$.
\end{dfn}

\begin{dfn}\label{dfn:matching}\showlabel{dfn:matching}
  A \emph{signed matching} is a set $M$ of signed matches with the following properties.
  \begin{enumerate}
      \renewcommand{\labelenumi}{(\roman{enumi})}
      \item Any two distinct signed matches $(\{a_1, b_1\}, \mu_1), (\{a_2, b_2\}, \mu_2) \in M$ satisfy $\{a_1, b_1\} \cap \{a_2, b_2\} = \varnothing$.
      \item No two matches in $M$ overlap.
  \end{enumerate}
\end{dfn}

\begin{dfn}\label{dfn:projmatching}\showlabel{dfn:projmatching}
  A \emph{projective signed matching} is a set $M$ of signed matches with the following properties.
  \begin{enumerate}
      \renewcommand{\labelenumi}{(\roman{enumi})}
      \item Any two distinct signed matches $(\{a_1, b_1\}, \mu_1), (\{a_2, b_2\}, \mu_2) \in M$ satisfy $\{a_1, b_1\} \cap \{a_2, b_2\} = \varnothing$.
      \item The set of matches $A = \{m\mid (m,\mu)\in M \hbox{ for some } \mu\in \pm1\}$ can be partitioned into $A=A_0\cup A_1$ such that the matches in $A_0$ do not overlap with any other match in $A$, and any two matches in $A_1$ overlap with each other.
  \end{enumerate}
  If $M$ is a (projective) signed matching, we let $E(M)$ denote $\{e \mid 1 \leq e \leq k, e \in m$ for some $(m, \mu) \in M \}$.
\end{dfn}

Let us consider a map $\kappa: \{1,\dots,k\} \rightarrow \{0,1,2\}$. Since such maps will be related to 3-edge-colorings, we will say that $\kappa$ is a \emph{coloring}. The coloring $\kappa$ satisfies \emph{the parity condition} if 
$|\kappa^{-1}(0)| \equiv |\kappa^{-1}(1)| \equiv |\kappa^{-1}(2)| \pmod 2$. We will only consider colorings that satisfy the parity condition since every edge-coloring of edges in a cut satisfies the parity condition.

\begin{dfn}\label{dfn:fit}\showlabel{dfn:fit}
  Let $\kappa: \{1,\dots,k\} \rightarrow \{0, 1, 2\}$ be a coloring that satisfies parity. For $\theta \in \{0, 1, 2\}$, $\kappa$ is said to \emph{$\theta$-fit} a projective signed matching $M$ if
  \begin{enumerate}
    \renewcommand{\labelenumi}{(\roman{enumi})}
    \item $E(M) = \{e \mid 1 \leq e \leq k, \kappa(e) \neq \theta\}$
    \item for each $(\{e, f\}, \mu) \in M$, $\kappa(e) = \kappa(f)$ if and only if $\mu = 1$.
  \end{enumerate}
\end{dfn}

\begin{dfn}\label{dfn:consist}\showlabel{dfn:consist}
  A set $\mathscr{C}$ of colorings $\{1,\dots,k\} \rightarrow \{0,1,2\}$ that satisfy parity is \emph{consistent} if for every $\kappa \in \mathscr{C}$ and every $\theta \in \{0,1,2\}$ there is a projective signed matching $M$ such that $\kappa$ $\theta$-fits $M$, and $\mathscr{C}$ contains every coloring that $\theta$-fits $M$.
\end{dfn}

\begin{dfn}\label{dfn:Dred}\showlabel{dfn:Dred}
  Let $I$ be an island. We add a new incident edge to each vertex of degree 2 in $I$ and denote the set of all added edges by $R$. 
  Let $\mathscr{C}^{\ast}$ be the set of all colorings $\{1,\dots,|R|\} \rightarrow \{0,1,2\}$ that satisfy parity.
  Let $\mathscr{C_0}$ be the set of all $3$-edge-colorings of $R$ that are restrictions of $3$-edge-colorings of $I \cup R$. 
  Let $\mathscr{C}'$ be the maximal consistent subset of $\mathscr{C}^{\ast} - \mathscr{C_0}$. 
  The island $I$ is said to be \emph{D-reducible} if $\mathscr{C}' = \emptyset$. We call a configuration $K$ \emph{D-reducible} if the corresponding island $I(K)$ is D-reducible.
\end{dfn}

\begin{dfn}
    Let $G$ be a subcubic graph, and $F$ be a set of edges in $G$ satisfying the following conditions:
    \begin{itemize}
        \item The endpoints of $f \in F$ have degree 3. 
        \item For each $v \in V(G)$, the number of edges adjacent to $v$ in $F$ is not $2$.
    \end{itemize}
    Let $G \dotdiv F$ be the subcubic graph obtained from $G$ by deleting the edges in $F$ and suppressing its endpoints if possible.
\end{dfn}


Note that if $G$ is a cubic graph, $G\dotdiv F$ is also a cubic graph.

\begin{dfn}\label{dfn:Cred}\showlabel{dfn:Cred}
  Let $I$ be an island and let $X \subseteq E(I)$ be such that none of the vertices in $I$ is incident with exactly two edges of $X$. 
  Let $I' = I \dotdiv X$ .
  Let $\mathscr{D}$ be all restrictions to $R$ of $3$-edge-colorings of $I' \cup R$.
  $I$ is \emph{C-reducible} if
    $\mathscr{D} \cap \mathscr{C}' = \emptyset$.
  The edges in $X$ are called \emph{contraction edges}\footnote{Although the edges in $X$ are not contracted, but rather deleted in $I$, we use the term ``contraction" following previous work. 
  This is because the term originated from definitions using the near-triangulation $K$ rather than the cubic dual $I$.} of $I$, and $X$ is denoted by $c(I)$. 
  The size of $X$ is called the \emph{contraction size} of $I$.
  
  We extend these terms to configurations $K$, whose island $I(K)$ is reducible. We call a configuration $K$ \emph{C-reducible} if $I(K)=I$ is C-reducible, we call the edges of $K$ corresponding to the edges in $c(I)$ the \emph{contraction edges}, we denote them by $c(K)$, and we refer to their number $|c(K)|$ as the \emph{contraction size} of $K$.
\end{dfn}

Let $F$ be a set of edges in $G'$. We denote by $G' / F$ the graph obtained from $G'$ by contracting $F$. Suppose that $G$ is embedded in some surface. A separating cycle $C$ of length 2, 3 in $G' / F$ separates $G' / F$ into a planar graph $H'$ and a non-planar one if the representativity of the embedding of $G' / F$ is at least 3. We consider the graph and its embedding obtained from $G' /F $ by deleting $H'$ from $G' / F$. We call this operation a \emph{low-(vertex)-cut reduction}. This operation corresponds to a low-edge-cut reduction in the original graph $G$. The symbol $\langle G' / F \rangle$ denotes the graph obtained from $G' / F$ by repeatedly applying low-vertex cut reductions to $G' / F$ until no 2,3-vertex cut exists.


The definition of reducibility here is harder to satisfy than the definitions of reducibility in the original 4CT paper due to the large amount of projective signed matches exclusively in the projective plane. 
We refer to the reducibility defined by using only planar signed matches (i.e. the reducibility used in the original 4CT proof) as the \emph{reducibility on the plane}.

Let $I$ be a reducible island appearing in some $G$.
By definition of reducibility, if we can 3-edge color $G \dotdiv c(I)$, we can extend it to a 3-edge-coloring of $G$. 
Thus, contracting the edges in $c(I)$ will not affect the non-3-edge colorability of a graph $G$ in which $I$ appears.
This implies the following lemma.

\begin{lem}\label{reduc1}\showlabel{reduc1}
   Let $G$ be a minimal counterexample that is embedded in the projective plane $\Sigma$. 
Suppose that $C$ is a circuit of $G$ bounding a closed disk $\Delta$ such that the subgraph of $G$ formed by the vertices and edges drawn in $\Delta$ is a reducible island $I$. Then $G \dotdiv c(I)$ either has a bridge or belongs to  $\mathcal{P}_0$.
\end{lem}

Now, we introduce a set of reducible configurations $\mathcal{K}$ which we have constructed by hand and verified each of its reducibility by computer programs.
There are approximately 5000 reducible configurations in $\mathcal{K}$ and the reader can find all reducible configurations (including contraction edges for $C$-reducible configurations) on the aforementioned Google Drive. 
Together with Lemma \ref{reduc1}, let us state a lemma here.
\begin{lem}\label{reduc}\showlabel{reduc}
Let $G$ be a minimal counterexample that is embedded in the projective plane $\Sigma$. 
Suppose that $C$ is a circuit of $G$ bounding a closed disk $\Delta$ such that the subgraph of $G$ formed by the vertices and edges drawn in $\Delta$ is an island $I$. Then either 
\begin{enumerate}
    \item $I$ is not in  $\mathcal{K}$, or  
    \item  $I$ is in $\mathcal{K}$, but $G \dotdiv c(I)$ either has a bridge or belongs to  $\mathcal{P}_0$. 
\end{enumerate}
\end{lem}

In fact, the set $\mathcal{K}$ is an unavoidable set of reducible configurations on the projective plane. We prove unavoidability by the discharging method later in this paper, see Lemma~\ref{T(v)>0}.

We have to check that for every single $C$-reducible configuration in $\mathcal{K}$, it is safe to make a reduction, i.e., we have to make sure that the resulting graph is not in $\mathcal{P}_0$ and that it is bridgeless. 
As pointed out above, this is almost trivial for the planar case and the doublecross case (because at most four edges are necessary to contract for every single $C$-reducible configuration) but in our case, we need a lot of nontrivial work. 
The problem is that the following may happen, which would not happen in the proof of the 4CT. 

\begin{itemize}
    \item There are 201 configurations that need to have at least five edges contracted. 
    Indeed, there are $118, 71, 5, 7$ configurations where the contraction edge size is 5, 6, 7, and 8, respectively. See Table \ref{tab1}. 
    \item There are around 1000 configurations that contain pairs of vertices at a distance of exactly five (although none of our configurations contain pairs of vertices at distance six or more). 
\end{itemize}

\begin{table}[htbp]
    \centering
    \caption{The number of configurations that need the contraction size $k$ $(0 \leq k \leq 8)$ in $\mathcal{K}$.}\label{tab1}
    \begin{tabular}{ccccccccccccc}
        \toprule
         contraction size & 0 & 1 & 2 & 3 & 4 & 5 & 6 & 7 & 8 & 9$_{\geq}$ & total   \\
        \midrule
         \# of configurations & 1083 & 3104 & 284 & 732 & 302 & 118 & 71 & 5 & 7 & 0 & 5706 \\
        \bottomrule
    \end{tabular}
\end{table}

Addressing these two issues makes our proof much harder than that of the 4CT and that of the doublecross case.
We will deal with these difficulties in the following sections.

\subsection{Distance five in reducible configurations}
\label{sect:dist5}\showlabel{sect:dist5}

In the proof of the 4CT in \cite{RSST}, there are 633 reducible configurations, and all of them are of diameter at most four. 
On the other hand, in our set of reducible configurations $\mathcal{K}$, there are many configurations whose diameter is five. 

If a reducible configuration $K$ in a planar graph is of diameter at most four, $K$ is an induced subgraph (otherwise, there would be a separating cycle of order at most five with both sides having at least two vertices, which would not exist in a minimal counterexample). 
This is also the case for doublecross graphs. 
However, in the projective planar case, we have to deal with the case when a reducible configuration is not induced. 

As before, let $G'$ be a triangulation that is embedded in the projective plane, and $K$ be a reducible configuration in $G'$. 
Let us suppose that there are two vertices $u, v \in V(K)$ that are at distance five in $K$, and there is an edge $uv \in E(G')$. 
Then we obtain a cycle $C$ of length exactly six, and have one of the following possibilities:

\begin{enumerate}
    \item[(C1)] $C$ is a contractible cycle in $G'$. 
    In the original cubic graph $G$, $C$ corresponds to a 6-cut, which yields a contractible curve of order exactly six that crosses all edges in the 6-cut.
    \item[(C2)] $C$ is a noncontracible cycle in $G'$.
    In $G$, $C$ yields a noncontractible curve crossing exactly six edges dual to $E(C)$. 
\end{enumerate}

Using Lemma \ref{minc} and Lemma \ref{lem:6,7-cut}, which will be proved in Section \ref{sect:6,7-cut}, we get the following fact.

\begin{fact}\label{fact:cont-cycle}\showlabel{fact:cont-cycle}
    If $G'$ is a minimal counterexample, then $G'$ contains no cycles with any of the following properties:
    \begin{itemize}
        \item A contractible cycle of length at most four, that divides $G'$ into two nonempty components.
        \item A contractible cycle of length five, that divides $G'$ into two components, each of order at least two.
        \item A contractible cycle of length six, that divides $G'$ into two components, each of order at least four.
    \end{itemize}
\end{fact}

To handle the two above cases (C1) and (C2), we show the following.

\begin{lem}\label{pairs6}\showlabel{pairs6}
    Let $C$ be a cycle in $G'$ of length $6$ as in (C1) and (C2).
\begin{enumerate}
    \item If $C$ is contractible, it divides $G'$ into two components, each of order at least four, and thus $G'$ is not a minimal counterexample.
    \item  If $C$ is noncontractible, there is a noncontractible cycle of length at most five in $G'$. That cycle corresponds to a non-contractible curve hitting at most five edges of the cubic dual graph $G$. 
\end{enumerate}
\end{lem}

\begin{proof}
    This proof depends on computer-check. In Section \ref{subsect:code-dist5}, we disclose the pseudo-code used to prove the lemma.
    By computer-check, there are around 1000 pairs that are at distance exactly five in reducible configurations in $\mathcal{K}$.
    Let $u, v$ be the two vertices that are at distance exactly five in a reducible configuration $K$. Let $S$ be a free completion of $K$ with its ring $R$. We now compare the free completion $S$ and its immersion in $G'$.

    Let the two vertices that constitute a triangle face with $u$ and $v$ in $G'$ be $x, y$. These vertices exist since $G'$ is a triangulation.
    In both contractible and non-contractible cases, a vertex of $R$ that is adjacent to $u$ in $S$ corresponds to $v$ (the same thing holds if the roles of $u$ and $v$ are swapped) by assuming that an edge $uv$ exists, which is denoted by $v'$ ($u'$, resp.). There are two vertices that correspond to $x$ in $S$, one is adjacent to $u$, and the other is adjacent to $v$. The same thing holds for $y$. The symbols $x_u,x_v,y_u,y_v$ denote a suitable vertex in $S$ described above, e.g. $x_u$ is a vertex of $S$ that corresponds to $x$ and is adjacent to $u$. By these correspondences, other correspondences can be obtained. For example, if a vertex that constitutes a triangle face with $x_u$ and $u$, but different from $v'$ exists, and a vertex that constitutes a triangle face with $x_v$ and $u'$, but different from $v$ exists in $S$, these two vertices correspond. We calculate these correspondences in the vicinity of $x$ and $y$.
    
    First, we explain the contractible case (C1), when a path of length five between $u$ and $v$, and the edge $uv$ constitute a contractible cycle $C$ of length six.
    We check whether $C$ is a cycle that contradicts Fact \ref{fact:cont-cycle}. To do so, we calculate the number of vertices of the two components separated by $C$. The cycle $C$ corresponds to the path that consists of a path between $u$ and $v$ and two edges $uv'$, $vu'$ in $S$. We calculate the number of vertices of the two components of $S$ divided by this path. We need to pay attention when counting the number of vertices since $R$ may not be a cycle but a circuit and a vertex of $R$ may correspond to a vertex of $K$. Hence, only counting the number of vertices may cause some overcounting. We need to be careful that double-counting does not happen. The specific counting scheme is explained in the pseudo-code in the Appendix, see  Section \ref{subsect:code-dist5}. 
    
    Let us consider a cycle that is obtained by finding a path between $x_u$ and $x_v$ in $S$. It is contractible since $C$ is contractible and we know that the graph embedded between $C$ and the cycle is an annulus.  We check whether this cycle contradicts Fact \ref{fact:cont-cycle} by using the same method as $C$. We do the same thing about $y$.
    
    As we explained before, we calculated the correspondance of vertices in $S$ in the vicinity of $x$ and $y$. The correspondence of the two different vertices of $K$ in $S$ means that these two different vertices of $K$ are same. This contradicts Fact \ref{fact:cont-cycle} since the distance between them is at most five.
     
    In almost all of the cases, we conclude that the cycle would contradict Fact \ref{fact:cont-cycle}. There are, however, eight cases that we cannot conclude immediately. We show them in Figure \ref{fig:dist5}. For these eight cases, we actually confirm that there are separating 4-cycles, which would still contradict Fact \ref{fact:cont-cycle}. 

    Second, we explain the non-contractible case. A path of length five between $u$ and $v$, and an edge $uv$ constitutes a non-contractible cycle $C$ of length six. We find other cycles by obtaining a path between corresponding vertices in $S$ in a similar way to the contractible case. This cycle is non-contractible since it intersects the non-contractible cycle $C$ only once. We check whether or not the length of these cycles obtained above is at most five. If this is the case, this is a desired cycle.
     
    It turns out that all of the configurations would yield a non-contractible cycle of length at most five by this computer check.
\end{proof}

\begin{figure}[htbp]
  \centering
  \includesvg[width=.75\columnwidth]{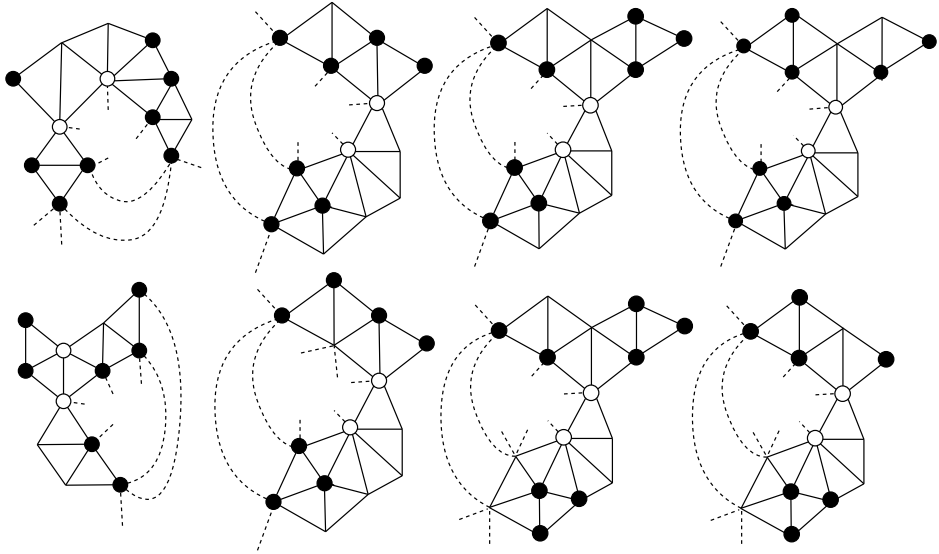}
  \caption{The eight cases that remain after computer-check in the proof of Lemma \ref{pairs6}. The dotted line represents edges when we assume two vertices of distance five are adjacent.}
  \label{fig:dist5}
\end{figure}

Therefore, for the first case, there is a contractible cycle of length at most six in $G'$ that would contradict Fact \ref{fact:cont-cycle} (i.e., a $\leq 6$-cut with one side embedded in the disk).

In the second case, we reduce to the case when the representativity of $G$ in $\Sigma$ is at most five (we shall handle in Section \ref{sect:rep}). As a corollary of Lemma \ref{reduc}, we have:

\begin{coro}\label{Kapper}\showlabel{Kappear}
 Let $G$ be a minimal counterexample that is embedded in the projective plane $\Sigma$. If the representativity of $G$ in $\Sigma$ is at least six, then none of our reducible configurations in $\mathcal{K}$ is a subgraph in $G$. 
\end{coro}
    
\begin{proof}
Assume for a contradiction that one of the reducible configurations $I(K)$ is a subgraph in $G'$.     
We claim that $I(K)$ is an induced subgraph of the dual graph $G'$ (and hence $I(K)$ appears in $G'$). 
To this end, if the representativity of $G'$ (and $G$) in $\Sigma$ is at least six this holds because $I(K)$ is of diameter at most five in $G'$ and hence $I(K)$ is an induced subgraph
by Lemma \ref{pairs6}. 


Hence $I(K)$ is induced and thus $K$ appears in $G$ which is a contradiction by Lemma \ref{reduc}. 
\end{proof}

\subsection{Bigger implies smaller}\label{subsect:imply}\showlabel{subsect:imply}

In Section \ref{sect:smaller-appendix} in the appendix, we prove that the reducibility of a configuration implies the reducibility of another configuration. We use Theorem \ref{thm:bigger to smaller degree} in the appendix by reducing the number of reducible configurations in $\mathcal{K}$. We prove the following claim.
\begin{claim}\label{clm:all-imply}\showlabel{clm:all-imply}
    Let $K$ be a configuration in $\mathcal{K}$ with no cut vertices.
    Let $V_0$ be a set of vertices of $G(K)$ such that $\gamma_K(v) = d_{G(K)}(v) + 3$.
    For each $S \subseteq V_0$, let $K_S$ be the configuration obtained from $K$ by replacing $\gamma_K(v)$ by $\gamma_{K_S}(v) = \gamma_K(v) - 1$ for each $v \in S$. Then $K_S$ can be safely reduced.\footnote{We say that $K_S$ is \emph{safely reducible} if it is reducible, and in addition, even when it is C-reducible, contractions of edges would not end up with $K_6$ after 2-3-vertex cut reductions of the resulting graph.}
\end{claim}

In Claim \ref{clm:all-imply}, we add the condition that $K$ has no cut vertices simply because we did not use it for computer-check of discharging. 

As the proof of this claim requires many of the facts proved in this paper, we give it in the appendix (see Subsection \ref{subsect:imply1}).

\section{Reducible configurations with contraction set of size $\leq 4$}
\label{sect:smallcont}\showlabel{sect:smallcont}

Let $G$ be a minimal counterexample, and let $G'$ be the dual of $G$, which is a triangulation of the projective plane. 
If a C-reducible configuration appears in any non-3-edge-colorable graph, we can delete some set of edges to obtain a smaller non-3-edge-colorable graph.
Deletions in $G$ (which are contractions in $G'$) result in a smaller graph, but we still need to check whether the deleted graph satisfies the induction hypothesis in order to confirm that it is still a counterexample.
For planar graphs, the induction hypothesis is satisfied if the reduced graph is bridgeless. 
This is almost trivial since no reducible configuration needs more than $4$ edge contractions, and internal 6-connectivity holds for minimal counterexamples.

However, we need additional arguments for the projective planar case.
In addition to the request that the resulting cubic graph must be bridgeless, it should not be a member of $\mathcal{P}_0$. 
(Let us recall that the definition of ``minimal counterexample'' is the smallest non-three-edge-colorable graph in $\mathcal{P}_1 \setminus \mathcal{P}_0$, so reducing the graph to a member of $\mathcal{P}_0$ does not satisfy the induction hypothesis.)

It is easy to show that any configuration of contraction size at most $2$ meets the requirements after reduction.
From Lemma \ref{minc}, the following condition holds.

\begin{itemize}
    \item $G$ contains no cyclic edge-cut of size 4 or less. 
    \item For any cyclic edge-cut $F$ of size 5, one of the connected components of $G - F$ must be a 5-cycle.
\end{itemize}
If $G\dotdiv e \in \mathcal{P}_0$ for any $e \in E(G)$, $G \dotdiv e = P_{10}$.
Similarly, if $G \dotdiv e \dotdiv f \in \mathcal{P}_0$ for any $e,f \in E(G)$, $G \dotdiv e \dotdiv f$ is either $P_{10}$ or a snark of order 12 obtained by replacing a vertex in $P_{10}$ with a triangle.
Such a graph can be enumerated quite easily (the largest one is a 16-vertex cubic graph), and it can be shown that every one of them can be 3-edge-colored.

On the other hand, there are infinitely many cubic graphs that can be reduced to a graph in $\mathcal{P}_0$ after three or more edge deletions (and endpoint suppressions), making it impossible to enumerate all possible cubic graphs, as done in the above case.
Therefore, we need to do some case analysis.
Before that, let us define some new terms.

\begin{dfn}
  A \emph{projective island} $I$ is a projective planar graph with the following properties.
  \begin{enumerate}
    \renewcommand{\labelenumi}{(\roman{enumi})}
    \item $I$ is 2-connected.
    \item Every vertex of $I$ has degree two or three.
    \item Every vertex of degree two is incident to a face that is homeomorphic to a 2-dimensional disk.
  \end{enumerate}
\end{dfn}

Regarding projective islands, we define the terms \textit{appear} and being \textit{reducible} in the same way as for islands defined in Section \ref{sect:reducible configurations}.

Next, we define some special projective islands.

\begin{dfn}
    For every $y \geq 3$, let $V_{2y}$ be the cubic graph obtained by adding all main diagonal edges to the cycle $C_{2y}$. 
    Let the set of the added $y$ edges be denoted by $E_{\mathrm{cross}}$.
    Let $\Gamma_{y}^k$ be the set of projective islands obtained from $V_{2y}$ by subdividing some edges on the cycle $C_{2y}$ with $k$ new vertices of degree $2$. 
    Let $\Pi_y^k$ be the subset of projective islands in $\Gamma_y^k$ which satisfy the two conditions below:
    \begin{enumerate}
        \item For every pair of opposite edges $e,f \in C_{2y}$, at least one of $e,f$ has been subdivided. 
        \item Each path in $C_{2y}$ of length $s\le y-1$ joining two vertices on this cycle, contains at least $s-1$ subdivision vertices.
    \end{enumerate}
\end{dfn}

Suppose that we want to remove a set $E_3$ of three edges from $G$. We would like to understand when $G \dotdiv E_3$ is Petersen-like.
We consider the following possibilities.

\begin{enumerate}
    \item \textbf{When $E_3$ is contained in a cyclic 6-edge-cut $F$ of $G$:} The 6-edge-cut will become a 3-edge-cut in $G \dotdiv E_3$.
    Then, one of the connected components of $G - F$ must be isomorphic to a graph in $\Pi_{3}^6$.
    \item \textbf{When 2 or 3 edges of $E_3$ are contained in a cyclic 5-edge-cut $F$ of $G$:} The 5-edge-cut will become a 2- or 3-edge-cut in $G \dotdiv E_3$.
    However, since one of the connected components of $G-F$ must be a 5-cycle, the order of $G$ can be upper-bounded. More precisely, $|G| \leq 18$ holds.
    \item \textbf{Otherwise:} $G \dotdiv E_3$ must be isomorphic to $P_{10}$, so $|G| \leq 16$ holds.
\end{enumerate}

Cases 2 and 3 create an upper bound for $|G|$, but for case 1, we need to use the reducibility checker for projective islands (in this case, we only require reducibility on the plane). 
There are 14 unique graphs in $\Pi_{3}^6$, all of which are either D-reducible or C-reducible with contraction size two or less.
Therefore, whenever a C-reducible configuration $I$ of contraction size three occurs in $G$ and all edges in $c(I)$ are contained in a particular cyclic 6-edge-cut of $G$, we can use the corresponding reducible projective island $I' \in \Pi_{3}^6$ instead. 
(Note that this argument holds for projective islands as well.)

This way, we can ensure that, even after the reduction, the induction hypothesis holds.

Therefore, the following holds.

\begin{lem}\label{lem:4cont3}\showlabel{lem:4cont3}
    Any planar or projective island $I$ that is C-reducible with contraction size at most three can be safely reduced.
\end{lem}

We can also do a case analysis for contraction size 4 in the same manner.
Next, let us define some graph families.

\begin{dfn}
    Let $P_{10}^{-}$ be the graph obtained by deleting an edge from $P_{10}$.
    Let $C$ be a cycle of length $8$ in $P_{10}^{-}$ including the two vertices of degree two.
    (There are two such cycles $C$ in the same graph, but they are equivalent through graph automorphism.)
    Let $\Delta^6$ be the set of projective islands that is obtained by subdividing some edges in the cycle $C$ with $4$ new vertices of degree two.
    Since there are two vertices of degree two in the original graph $P_{10}^{-}$, the constructed island's ring size is six.
    
    Let $\hat{\Pi}_3^6$ be a set of graphs $G$ that can be constructed by the following method:
    \begin{itemize}
        \item Pick a projective island $H \in \Pi_3^6$.
        \item Choose two edges $e,f \in E(H)$ sharing the same face but having no common endpoint.
        \item Subdivide $e$ and $f$. Let the newly formed vertex be $v_e$ and $v_f$, respectively.
        \item Connect $v_e$ and $v_f$ in $H$ to obtain $G$.
        Since $e$ and $f$ share the same face, the resulting graph $G$ is still embeddable in the projective plane.
    \end{itemize}
\end{dfn}

Now, let $E_4$ be the set of four edges we want to delete from $G$. 
Let us assume that no cyclic 5-cut $F$ with $E_4 \subset F$ exists. 
(This assumption is to avoid making a bridge in $G - E_4$.) 
Again we would like to understand when $G \dotdiv E_4$ is Petersen-like.

\begin{enumerate}
    \item \textbf{When $E_4$ is contained in a cyclic 7-edge-cut $F$ of $G$:} The 7-edge-cut will become a 3-edge-cut in $G \dotdiv E_3$.
    Then, one of the connected components of $G - F$ must be isomorphic to a graph in $\Pi_{3}^7$.
    \item \textbf{When $E_4$ is contained in a cyclic 6-edge-cut $F$ of $G$:} The 6-edge-cut will become a 2-edge-cut in $G \dotdiv E_3$.
    Then, one of the connected components of $G - F$ must be isomorphic to a graph in $\Delta^6$.
    \item \textbf{When $E_4$ is contained in a cyclic 5-edge-cut $F$ of $G$:} This would not happen by our assumption. 
    \item \textbf{When three edges of $E_4$ are  contained in a cyclic 6-edge-cut $F$ of $G$:}
    The 6-edge-cut will become a 3-edge-cut in $G \dotdiv E_3$, and there is another edge in $E_4$ somewhere else.
    Then, one of the connected components of $G - F$ must be either isomorphic to a graph in $\Pi_3^6$ or $\hat{\Pi}_3^6$.
    \item \textbf{When two or three edges of $E_4$ are contained in a cyclic 5-edge-cut $F$ of $G$:}
    Because of Lemma \ref{minc}, the size of $G$ can be upper-bounded. Precisely, $|G| \leq 20$ holds.
    \item \textbf{Otherwise:} $G - E_4$ must be isomorphic to $P_{10}$, so $|G| \leq 18$ holds.
\end{enumerate}

Therefore, all we need to check is the reducibility of graphs in $\Pi_3^7$, $\Delta^6$, and $\hat{\Pi}_3^6$. 
All projective islands in these three families are either D or C-reducible.
There are, however, two projective islands in $\Delta^6$ and three projective islands in $\Pi_3^7$, where a four-edged contraction is needed for C-reducibility.
In these cases, by showing the existence of non-contractible curves hitting at most two edges, we can show that the representativity of the resulting graph becomes two or less, which would not result in $\mathcal{P}_0$.

We would also need to check that the resulting graph is in $\mathcal{P}_1$, i.e. the resulting graph is bridgeless.
For the two projective islands $I$ in $\Delta^6$, we can see that it is impossible to find a 5-edge cut $F$ containing the four edges of $c(I)$.
This implies that even after deleting $c(I)$ in the projective island $I$, the resulting cubic graph does not have a bridge. 
For the three projective islands $I$ in $\Pi_3^7$, two of the islands have the risk of forming a bridge after deleting $c(I)$.
However, we see that in both of the cases, the bridge can be avoided by finding a different projective island in $\Delta^6$.
Therefore we can conclude that it will not become a Petersen-like graph. 
The detailed descriptions of these projective islands, together with their contraction of edges and the non-contractible curve showing low representativity, are shown in Figures \ref{fig:delta-6-dangerous} and \ref{fig:pi-3-7-dangerous}.

\begin{figure}[htbp]
  \centering
  \includegraphics[width=6.5cm]{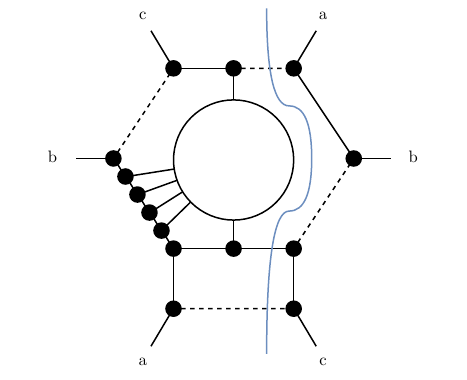}
  \includegraphics[width=6.5cm]{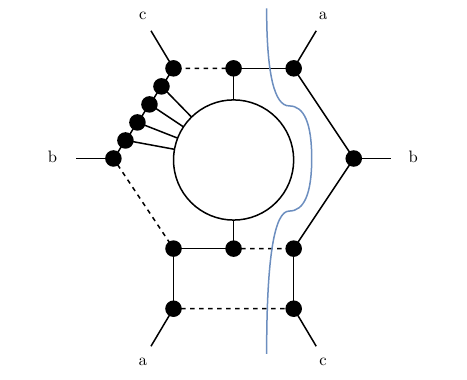}
  \caption{Two C-reducible islands in $\Delta^6$, having contraction size $4$. 
  The contraction edges are depicted by dotted lines and
  the non-contractible curve is represented with a blue curved line.}
  \label{fig:delta-6-dangerous}
\end{figure}

\begin{figure}[htbp]
  \centering
  \includegraphics[width=5cm]{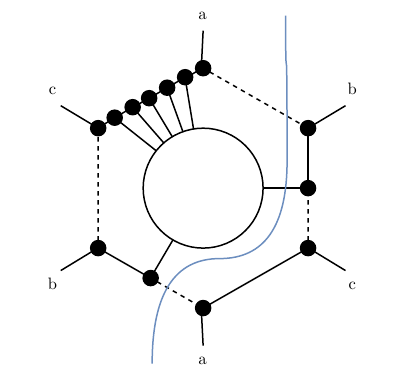}
  \includegraphics[width=5cm]{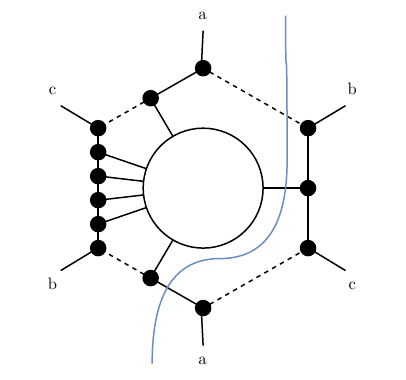}
  \includegraphics[width=5cm]{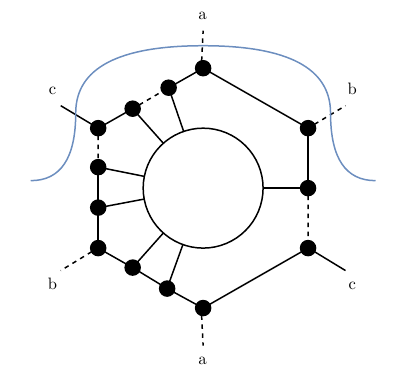}
  \caption{Three C-reducible islands in $\Pi_3^7$, having contraction size $\geq 4$.
  The contraction edges are depicted by dotted lines and
  the non-contractible curve is represented with a blue curved line.}
  \label{fig:pi-3-7-dangerous}
\end{figure}

The two cases where a bridge may occur are shown in Figure \ref{fig:pi-3-7-bridge}.

\begin{figure}[htbp]
  \centering
  \includegraphics[width=6.5cm]{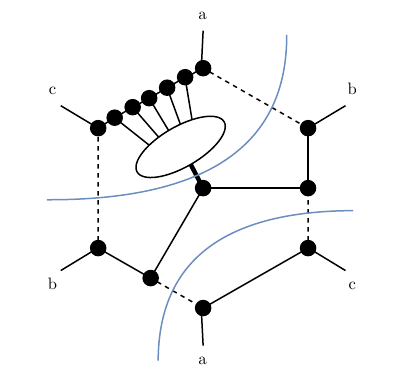}
  \includegraphics[width=6.5cm]{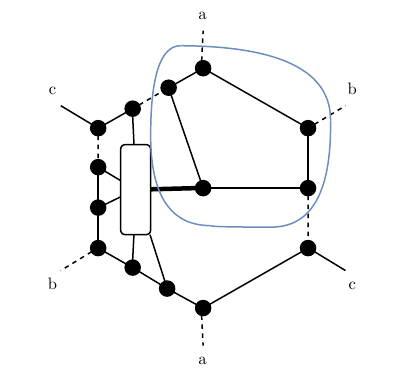}
  \caption{Two cases where a bridge (shown in bold lines) may occur in a previous 5-edge-cut (represented with a dotted curve) in $\Pi_3^7$. 
  In both cases, a larger projective island in $\Delta^6$ is contained (as shown in the image, surrounded by a blue closed curve).}
  \label{fig:pi-3-7-bridge}
\end{figure}

From these facts, the following proposition holds.

\begin{proposition}\label{lem:4cont2}\showlabel{lem:4cont2}
    Let $I$ be a $C$-reducible (projective or planar) island of contraction size at most four.
    If $I$ appears in $G$ and $c(I)$ is not contained in any cyclical 5-edge-cut of $G$, it can be safely reduced.
\end{proposition}

For every reducible island $I \in \mathcal{K}$, we have checked whether $c(I)$ could be contained in any cyclical 5-edge-cut of $G$ and confirmed that no such $c(I)$ exists (this means that after deleting $c(I)$, the resulting graph has no bridge).
The checking is conducted using computer programs, and we explain the details in Section \ref{sect:largecont}.

Thus, we obtain the following result for islands of contraction size at most four in $\mathcal{K}$.

\begin{lem}\label{lem:cont4}\showlabel{lem:cont4}
    Let $I$ be a C-reducible island of contraction size at most four in $\mathcal{K}$. 
    Then, it can be safely reduced.
\end{lem}

\section{Discharging process}
\label{sect:discharging}\showlabel{sect:discharging}

Let $G'$ be an internally 6-connected triangulation of the projective plane. We start by assigning to each vertex $v \in V(G')$ the value 
$$
    T_0(v) := 10(6 - d(v)),
$$
which we call the \emph{initial charge at $v$}. It follows by Euler's formula (see Lemma \ref{lem:120} below) that 
$$
    \sum_{v\in V(G')} T_0(v) = 60.
$$
Now we redistribute the charge among vertices by applying the \emph{discharging rules} that are shown in Section \ref{sect:rule} in the appendix. There are 169 rules (compared to 33 rules for the proof of the 4CT and that for the doublecross case).

Discharging rules are formally defined as follows.

\begin{dfn}\label{dfn:rule}\showlabel{dfn:rule}
  A \emph{rule} is six-tuple $R = (G(R), \beta_R, \delta_R, r(R), s(R), t(R))$, so that
  \begin{enumerate}
      \renewcommand{\labelenumi}{(\roman{enumi})}
      \item $G(R) = (V(R), E(R))$ is a near-triangulation, and for each $v \in V(R)$, $G(R) - v$ is connected.
      \item $\beta_R \colon V(R) \rightarrow \mathbb{N}$, $\delta_R \colon V(R) \rightarrow \mathbb{N} \cup \{\infty\}$, such that $5 \leq \beta_R(v) \leq \delta_R(v)$ for each $v \in V(R)$.
      \item $r(R)$ is a positive integer.
      \item $s(R), t(R) \in V(R)$ are distinct adjacent vertices.
  \end{enumerate}
  A configuration $K$ \emph{obeys} the rule $R$ if $G(K) = G(R)$ and every vertex $v \in V(K)$ satisfies $\beta_R(v) \leq \gamma_K(v) \leq \delta_R(v)$.
\end{dfn}

A rule is a near-triangulation in which each vertex $v$ is given a range of possible degrees, the interval $[\beta_R(v), \delta_R(v)]$. The value $r(R)$ represents the amount of charge that is sent from the vertex $s(R)$ to its neighbor $t(R)$. We have $r(R) = 1$ or $r(R)=2$ for all our rules. We describe the degree range $[\beta_R, \delta_R]$ by the same convention as in configurations. There are three cases. When $\beta_R(v) = \delta_R(v)$, we describe $\beta_R(v)$ by one of the shapes shown in Figure \ref{fig:shapes}. When $5 = \beta_R(v) < \delta_R(v)$, we describe $\delta_R(v)$ as the vertex shape and set $-$ next to the vertex. When $\beta_R(v) < \delta_R(v) = \infty$, we describe $\beta_R(v)$ as the vertex shape and set $+$ next to the vertex. All our rules have one of these three types of degree ranges. 

The set of 169 rules exhibited in Figure \ref{fig:rules} in the appendix is denoted by $\mathcal{R}$. 
An edge with an arrow represents the direction in which a charge moves. Thus, $s(R), t(R)$ are two vertices that are incident with an edge with an arrow.
The number of arrows in the figure represents $r(R)$.

We calculate the amount of charge sent along an edge $st$ in $G'$ by checking which rules can be used on this edge. For a rule $R\in \mathcal{R}$, $r(R)$ gives the amount of charge to send, $s(R)$ is the vertex that sends the charge, and $t(R)$ is the vertex that receives charge. We apply a rule $R$ when a configuration $K$, which obeys $R$ appears in $G'$. Note that each\footnote{The very first rule (R1) is the only exception that is used only once.} rule $R\in \mathcal{R}$ can be applied at the edge $st$ in two ways, where the triangles in $G(R)$ have the same orientation as in $G'$, and when they have opposite orientation.\footnote{Of course, there is no global orientation on the projective plane; we just consider orientation of triangles locally around the edge $st$.}

\begin{dfn}\label{dfn:amount}\showlabel{dfn:amount}
  For adjacent vertices $u, v \in V(G')$, we define $\phi(u,v)$ as the sum of the values $r(R)$ over all rules $R\in \mathcal{R}$ that are applied with $s(R) = u$, and $t(R) = v$. This is the \emph{charge sent from $u$ to $v$}.
\end{dfn}

\begin{dfn}\label{dfn:finalcharge}\showlabel{dfn:finalcharge}
    For a vertex $u \in V(G')$, we set
    \[
      T(u) := T_0(u) + \sum_{v \sim u} (\phi(v, u) - \phi(u,v)).
    \]
    The obtained value is said to be the \emph{final charge} at $u$.
\end{dfn}

The discharging method uses the following easy observation.

\begin{lem}
\label{lem:120}\showlabel{lem:120}
$\sum_{u \in V(G')} T(u) = 60$.
\end{lem}

\begin{proof}
By applying any discharging rule, the total sum of all charges remains the same. Thus, $\sum_{u \in V(G')} T(u) = \sum_{u \in V(G')} 10(6-d(u))$. If $n$ is the number of vertices of $G'$, then Euler's formula for the triangulation $G'$ implies that $\tfrac12 \sum_{u\in V(G')}d(u) = |E(G)| = 3n - 3$, and thus we have:
 \[
 \sum_{u \in V(G')} T(u) = \sum_{u \in V(G')} 10(6-d(u)) = 60n - 10\sum_{u \in V(G')}d(u) =
 60n-20(3n-3) = 60.
 \]
\end{proof}

Lemma \ref{lem:120} implies that there is a vertex $u$ with $T(u) > 0$. Known proofs of the 4CT show that for any vertex $u$ with $T(u)>0$, there is a reducible configuration in the vicinity of $u$. This yields a contradiction to the assumption that we have a minimum counterexample. Indeed we have a similar result on the projective plane, but there is some difference.

\begin{lem}\label{T(v)>0}\showlabel{T(v)>0}
  Let $G$ be a minimal counterexample. Then there is a vertex $v$ of the dual graph $G'$ with $T(v) >0$, and there must exist a subgraph $K'$ of $G'$ which contains either $v$ or at least one of its neighbors, and is isomorphic to one of the reducible configurations of $\mathcal{K}$ (but $K'$ may not be an induced subgraph of $G'$)   
\end{lem}

So the difference is that $K'$ may not be an induced subgraph of $G'$. 
But, if $K'$ is not induced, this is because there are two vertices in $K'$ that are of distance at least five in $K'$. This issue was already handled by Corollary \ref{Kapper}. 

We show Lemma \ref{T(v)>0} by computer check, but when $d(v)=5$ or $d(v)=6$ or $d(v) \geq 12$ (in $G'$), we can show this lemma without the help of computers. For other cases, computer-free proof would be too long.

\subsection{Final charge at vertices of degree 5 or 6 in $G'$}
\label{sect:deg5and6}
\showlabel{sect:deg5and6}

We now consider the final charge $T(v)$ at vertices of degree at most six in the dual graph $G'$. We show, without using computer,  that $T(v)=0$ when $d(v) \leq 6$. 
A set of reducible configurations used here for the computer check-free proof is given in Figure \ref{fig:confs_degree56}. The configurations in Figure \ref{fig:confs_degree56} is contained in $\mathcal{K}$.
We denote a configuration in Figure \ref{fig:confs_degree56} by conf(i) $(1 \leq i \leq 24)$ in the natural order, from top to bottom in a row, from left to right in a column.

We only use the rules for vertices of degree $5$ and $6$. 
They are collected in Figure \ref{fig:rules_degree56}. We denote the rule in Figure \ref{fig:rules_degree56} by rule(i) $(1 \leq i \leq 8)$ in the natural order. Only rule(1) gives value $r(R) = 2$.
The rules of our discharging method are stronger than that of the 4CT since many configurations that are reducible for planar graphs are not reducible for projective planar graphs.

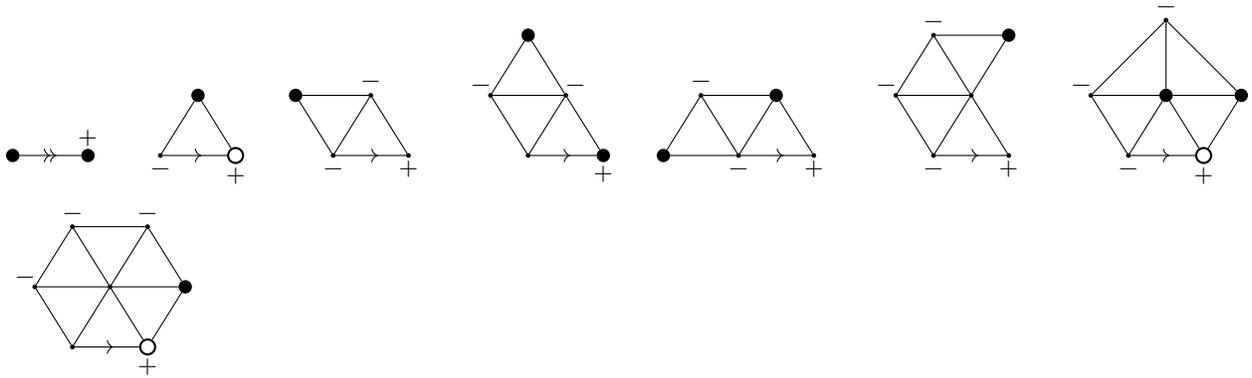
\begin{figure}[htbp]
  \tikzset{deg5/.style={thick, circle, draw, fill=black, inner sep=1.5pt,}}
  \tikzset{deg6/.style={thick, circle, draw, fill=black, inner sep=0pt,}}
  \tikzset{deg7/.style={thick, circle, draw, fill=white, inner sep=2pt,}}
  \tikzset{->-/.style={decoration={
      markings,
      mark=at position .6 with {\arrow{>}}}, postaction={decorate}}}
  \begin{tikzpicture}[baseline=1.5cm]
    \node [deg5] at (0, 0) (a) {};
    \node [deg5] at (1.0, 0) (b) {};
    \node [above = 0.15 cm of b, anchor=center] (b+) {$+$};
    \draw [decoration={markings, mark=at position .5 with {\arrow{>}};, mark=at position .6 with {\arrow{>}}; }, postaction={decorate}] (a) -- (b);
  \end{tikzpicture}
  \begin{tikzpicture} [baseline=1.5cm]
    \node [deg6] at (0, 0) (a) {};
    \node [below = 0.15 cm of a, anchor=center] () {$-$};
    \node [deg7] at (1.0, 0) (b) {};
    \node [below = 0.15 cm of b, anchor=center] () {$+$};
    \node [deg5] at (0.5, 0.8) (c) {};
    \draw [->-] (a) -- (b);
    \foreach \u / \v in {a/c, b/c}
        \draw (\u) -- (\v);
  \end{tikzpicture}
  \begin{tikzpicture} [baseline=1.5cm]
    \node [deg6] at (0.5, 0) (a) {};
    \node [below = 0.15 cm of a, anchor=center] () {$-$};
    \node [deg6] at (1.5, 0) (b) {};
    \node [below = 0.15 cm of b, anchor=center] () {$+$};
    \node [deg6] at (1.0, 0.8) (c) {};
    \node [above = 0.15 cm of c, anchor=center] () {$-$};
    \node [deg5] at (0, 0.8) (d) {};
    \draw [->-] (a) -- (b);
    \foreach \u / \v in {a/c, a/d, b/c, c/d}
        \draw (\u) -- (\v);
  \end{tikzpicture}
  \begin{tikzpicture} [baseline=1.5cm]
    \node [deg6] at (0.5, 0) (a) {};
    \node [deg5] at (1.5, 0) (b) {};
    \node [below = 0.15 cm of b, anchor=center] () {$+$};
    \node [deg6] at (1.0, 0.8) (c) {};
    \node [above right = 0.15 cm of c, anchor=center] () {$-$};
    \node [deg6] at (0, 0.8) (d) {};
    \node [above left = 0.15 cm of d, anchor=center] () {$-$};
    \node [deg5] at (0.5, 1.6) (e) {};
    \draw [->-] (a) -- (b);
    \foreach \u / \v in {a/c, a/d, b/c, c/d, c/e, d/e}
        \draw (\u) -- (\v);
  \end{tikzpicture}
  \begin{tikzpicture} [baseline=1.5cm]
    \node [deg6] at (1.0, 0) (a) {};
    \node [below = 0.15 cm of a, anchor=center] () {$-$};
    \node [deg6] at (2.0, 0) (b) {};
    \node [below = 0.15 cm of b, anchor=center] () {$+$};
    \node [deg5] at (1.5, 0.8) (c) {};
    \node [deg6] at (0.5, 0.8) (d) {};
    \node [above = 0.15 cm of d, anchor=center] () {$-$};
    \node [deg5] at (0, 0) (e) {};
    \draw [->-] (a) -- (b);
    \foreach \u / \v in {a/c, a/d, a/e, b/c, c/d, d/e}
        \draw (\u) -- (\v);
  \end{tikzpicture}
  \begin{tikzpicture} [baseline=1.5cm]
    \node [deg6] at (0.5, 0) (a) {};
    \node [below = 0.15 cm of a, anchor=center] () {$-$};
    \node [deg6] at (1.5, 0) (b) {};
    \node [below = 0.15 cm of b, anchor=center] () {$+$};
    \node [deg6] at (0, 0.8) (c) {};
    \node [above left = 0.15 cm of c, anchor=center] () {$-$};
    \node [deg6] at (1.0, 0.8) (d) {};
    \node [deg6] at (0.5, 1.6) (e) {};
    \node [above = 0.15 cm of e, anchor=center] () {$-$};
    \node [deg5] at (1.5, 1.6) (f) {};

    \draw [->-] (a) -- (b);
    \foreach \u / \v in {a/c, a/d, b/d, c/d, c/e, d/e, d/f, e/f}
        \draw (\u) -- (\v);
  \end{tikzpicture}
  \begin{tikzpicture} [baseline=1.5cm]
    \node [deg6] at (0.5, 0) (a) {};
    \node [below = 0.15 cm of a, anchor=center] () {$-$};
    \node [deg7] at (1.5, 0) (b) {};
    \node [below = 0.15 cm of b, anchor=center] () {$+$};
    \node [deg6] at (0, 0.8) (c) {};
    \node [above left = 0.15 cm of c, anchor=center] () {$-$};
    \node [deg5] at (1.0, 0.8) (d) {};
    \node [deg5] at (2.0, 0.8) (e) {};
    \node [deg6] at (1.0, 1.8) (f) {};
    \node [above = 0.15 cm of f, anchor=center] () {$-$};
    \draw [->-] (a) -- (b);
    \foreach \u / \v in {a/c, a/d, b/d, b/e, c/d, c/f, d/e, d/f, e/f}
        \draw (\u) -- (\v);
  \end{tikzpicture}
  \begin{tikzpicture} [baseline=1.5cm]
    \node [deg6] at (0.5, 0) (a) {};
    \node [deg7] at (1.5, 0) (b) {};
    \node [below = 0.15 cm of b, anchor=center] () {$+$};
    \node [deg6] at (0, 0.8) (c) {};
    \node [above left = 0.15 cm of c, anchor=center] () {$-$};
    \node [deg6] at (1.0, 0.8) (d) {};
    \node [deg5] at (2.0, 0.8) (e) {};
    \node [deg6] at (0.5, 1.6) (f) {};
    \node [above = 0.15 cm of f, anchor=center] () {$-$};
    \node [deg6] at (1.5, 1.6) (g) {};
    \node [above = 0.15 cm of g, anchor=center] () {$-$};
    \draw [->-] (a) -- (b);
    \foreach \u / \v in {a/c, a/d, b/d, b/e, c/d, c/f, d/e, d/f, d/g, e/g, f/g}
        \draw (\u) -- (\v);
  \end{tikzpicture}
  \caption{Rules used for discharging at vertices of degree 5 or 6. \label{fig:rules_degree56}}
\end{figure}
\subfile{tikz/confdeg56}

\begin{lem}
  \label{lem:lemdeg5deg6}
  \showlabel{lem:lemdeg5deg6}
  Let $G' = (V, E)$ be an internally 6-connected triangulation of the projective plane. Let $v \in V$ be a vertex of degree 5 or 6 in $G'$.
  Let $I_k$ $(O_k)$ be the amount of charge for $v$ to receive $($send$)$ by rule($k$) $(1 \leq k \leq 8)$ in Figure $\ref{fig:rules_degree56}$.
  If none of the configurations in Figure $\ref{fig:confs_degree56}$ appears in $G'$, then:
  \begin{enumerate}
      \renewcommand{\labelenumi}{$(\roman{enumi})$}
      \item $I_1 = O_2 + O_3 + O_5$ 
      \item $I_3 = O_4$
      \item $I_4 + I_5 = O_6 + O_7$, and
      \item $I_6 = O_8$.
  \end{enumerate}
\end{lem}
\begin{proof}
  We first prove (i). We consider a set of quadtuple $(x_1, x_2, x_3, x_4)$ of neighbors of $v$ such that all of them are distinct and, $x_i, x_{i+1} (1 \leq i \leq 3)$ are adjacent. This set is denoted by  $A$. The following formulas hold.
  \begin{eqnarray}
      I_1 &=& |\{(x_1,x_2,x_3,x_4) \in A \mid \deg_{G'}(x_1) = 5 \}| \nonumber \\
      O_2 &=& |\{(x_1,x_2,x_3,x_4) \in A \mid \deg_{G'}(x_1) = 5, \deg_{G'}(x_2) \geq 7 \}| \nonumber \\
      O_3 &=& |\{(x_1,x_2,x_3,x_4) \in A \mid \deg_{G'}(x_1) = 5, \deg_{G'}(x_2) \leq 6, \deg_{G'}(x_3) \geq 6 \}| \nonumber \\
      O_5 &=& |\{(x_1,x_2,x_3,x_4) \in A \mid \deg_{G'}(x_1) = 5, \deg_{G'}(x_2) \leq 6, \deg_{G'}(x_3) = 5, \nonumber \\
      && \deg_{G'}(x_4) \geq 6 \}| \nonumber
  \end{eqnarray}
  conf(1), conf(2), conf(3), conf(4) do not appear in $G'$, so $|\{(x_1,x_2,x_3,x_4) \in A \mid \deg_{G'}(x_1) = 5, \deg_{G'}(x_2) \leq 6, \deg_{G'}(x_3) = 5, \deg_{G'}(x_4) = 5 \}| = 0$. Hence $I_1 = O_2 + O_3 + O_5$ holds.

  We then prove (ii). $I_3 = O_4 = 0$ when $\deg_{G'}(v) = 5$. We consider a set of quadtuple $(y_1, y_2, y_3, y_4)$ such that all of them are distinct, $y_1, y_2, y_3$ are neighbors of $v$, $y_i, y_{i+1} (1 \leq i \leq 2)$ are adjacent and $y_4$ is adjacent $y_1, y_2$ but different from $v$. This set is denoted by $B$.
  \begin{eqnarray}
      I_3 &=& |\{(y_1, y_2, y_3, y_4) \in B \mid \deg_{G'}(y_1) \leq 6, \deg_{G'}(y_2) \leq 6, \deg_{G'}(y_4) = 5 \}| \nonumber \\
      O_4 &=& |\{(y_1, y_2, y_3, y_4) \in B \mid \deg_{G'}(y_1) \leq 6, \deg_{G'}(y_2) \leq 6, \deg_{G'}(y_3) \geq 5, \nonumber \\
      && \deg_{G'}(y_4) = 5 \}| \nonumber
  \end{eqnarray}
  So $I_3 = O_4$ holds.

  We prove (iii). We consider a set of five-tuple $(z_1, z_2, z_3, z_4, z_5)$ such that all of them are distinct, $z_1, z_2, z_3$ are neighbors of $v$, $z_i, z_{i+1} (1 \leq i \leq 2)$ are adjacent, $z_4$ is adjacent $z_1, z_2$ but different from $v$ and $z_5$ is adjacent $z_2, z_4$ but different from $z_1$. This set is denoted by $C$.
  \begin{eqnarray}
      I_4 &=& |\{(z_1, z_2, z_3, z_4, z_5) \in C \mid \nonumber \\
      && \deg_{G'}(z_1) = 6, \deg_{G'}(z_2) \leq 6, \deg_{G'}(z_4) \leq 6, \deg_{G'}(z_5) = 5 \}| \nonumber \\
      I_5 &=& |\{(z_1, z_2, z_3, z_4, z_5) \in C \mid \nonumber \\
      && \deg_{G'}(z_1) = 5, \deg_{G'}(z_2) \leq 6, \deg_{G'}(z_4) \leq 6, \deg_{G'}(z_5) = 5 \}| \nonumber \\
      O_6 &=& |\{(z_1, z_2, z_3, z_4, z_5) \in C \mid \nonumber \\
      && \deg_{G'}(z_1) \leq 6, \deg_{G'}(z_2) = 6, \deg_{G'}(z_3) \geq 6, \deg_{G'}(z_4) \leq 6, \deg_{G'}(z_5) = 5 \}| \nonumber \\
      O_7 &=& |\{(z_1, z_2, z_3, z_4, z_5) \in C \mid \nonumber \\
      && \deg_{G'}(z_1) \leq 6, \deg_{G'}(z_2) = 5, \deg_{G'}(z_3) \geq 7, \deg_{G'}(z_4) \leq 6, \deg_{G'}(z_5) = 5 \}| \nonumber \\
  \end{eqnarray}  
  conf(2), conf(3), conf(8), conf(9), conf(10) do not appear in $G'$, so $|\{(z_1, z_2, z_3, z_4, z_5) \in D \mid \deg_{G'}(z_1) \leq 6, \deg_{G'}(z_2) = 6, \deg_{G'}(z_3) = 5, \deg_{G'}(z_4) \leq 6, \deg_{G'}(z_5) = 5 \}| = 0$.
  conf(1), conf(4), conf(5), conf(6), conf(7) do not appear in $G'$, so $|\{(z_1, z_2, z_3, z_4, z_5) \in D \mid \deg_{G'}(z_1) \leq 6, \deg_{G'}(z_2) = 5, \deg_{G'}(z_3) \leq 6, \deg_{G'}(z_4) \leq 6, \deg_{G'}(z_5) = 5 \}| = 0$.
  Hence $I_4 + I_5 = O_6 + O_7$.
  
  We finally prove (iv). $I_6 = O_8 = 0$ when $\deg_{G'}(v) = 5$. We consider a set of six-tuple $(w_1, w_2, w_3, w_4, w_4, w_6)$ such that all of them are distinct, $w_1, w_2, w_6$ are neighbors of $v$, $w_i(2 \leq i \leq 6)$ and $v$ are neighbors of $w_1$ such that $\deg_{G'}(w_1) = 6$ and $w_i, w_{i+1} (2 \leq i \leq 5)$ are adjacent. This set is denoted by $D$. The following formulas hold:
  \begin{eqnarray}
      I_6 &=& |\{(w_1, w_2, w_3, w_4, w_5, w_6) \in D \mid \nonumber \\
      && \deg_{G'}(w_2) \leq 6, \deg_{G'}(w_3) \leq 6, \deg_{G'}(w_4) \leq 6, \deg_{G'}(w_5) = 5 \}| \nonumber \\
      O_8 &=& |\{(w_1, w_2, w_3, w_4, w_5, w_6) \in D \mid \nonumber \\
      && \deg_{G'}(w_2) \leq 6, \deg_{G'}(w_3) \leq 6, \deg_{G'}(w_4) \leq 6, 
      \deg_{G'}(w_5) = 5 ,\deg_{G'}(w_6) \geq 7\}| \nonumber
  \end{eqnarray}
  conf(2), conf(3), conf(8), conf(11), conf(12), conf(13) do not appear in $G'$, so $|\{(w_1, w_2, w_3, w_4, w_5, w_6) \in D \mid \deg_{G'}(w_2) \leq 6, \deg_{G'}(w_3) \leq 6, \deg_{G'}(w_4) \leq 6, \deg_{G'}(w_5) = 5, \deg_{G'}(w_6) \leq 6\}| = 0$. Hence we have that $I_6 = O_8$.
\end{proof}

\begin{lem}
  \label{lem:deg5deg6}
  \showlabel{lem:deg5deg6}
  Let $G' = (V, E)$ be an internally 6-connected triangulation. Let $v \in V$ be a vertex of degree 5 or 6.
  Suppose that configurations that are shown in Figure $\ref{fig:confs_degree56}$ do not appear in $G'$. Then,  $T(v) = 0$.
\end{lem}

\begin{proof}
  When $\deg_{G'}(v) = 5$,
  \[
    T(v) = 10 + I_1 + I_4 - O_1 - O_2 - O_3 - O_5 - O_6 - O_7
  \]
  since other $I_k, O_k$ do not affect charge of $v$. $O_1 = 10$ since $v$ has five neighbors. By Lemma  \ref{lem:lemdeg5deg6} (i), (iii) and $I_5 = 0$, $T(v) = 0$.
  When $\deg_{G'}(v) = 6$,
  \[
    T(v) = I_1 + I_3 + I_4 + I_5 + I_6 - O_2 - O_3 - O_4 - O_5 - O_6 - O_7 - O_8, 
  \]
  since other $I_k, O_k$ do not affect charge of $v$. 
  By Lemma  \ref{lem:lemdeg5deg6} (i), (ii), (iii) and (iv), $T(v) = 0$.
\end{proof}


\subsection{Maximum charge that a vertex can send}
\label{sect:maxcharge}
\showlabel{sect:maxcharge}

Let $n$ be a positive integer. We want to figure out the case that a vertex sends charge $n$ to a vertex of degree at least $7$. We calculate it by combining rules of $\mathcal{R}$ and then by checking that none of our reducible configurations of $\mathcal{K}$ is a subgraph. We use computer-check to calculate it. For this purpose, we provide a precise definition and a pseudo-code in Section \ref{subsect:code-disch} in the appendix\footnote{In Section \ref{subsect:code-disch} we use the same notation and terms as rules to explain the case when a vertex sends charge $n$ (e.g., a configuration \emph{obeys} the case).}. 
We now show the following.

\begin{lem}\label{lem:vsends}\showlabel{lem:vsends}
Assume that none of our reducible configurations in the set $\mathcal{K}$ is a subgraph in $G'$. Then, any vertex $v \in V(G')$ can send charge at most $6$ to any one of its neighbors.  
If it sends charge $6$ ($5$ respectively), then we have one of the cases shown in Figure \ref{fig:proj6} (Figure \ref{fig:proj5}) in the appendix. 
If a vertex of degree $5$ sends charge $4$, then we have one of the cases in Figure \ref{fig:proj4_deg5} in the appendix.
\end{lem}
We prove Lemma \ref{lem:vsends} by executing Algorithm \ref{alg:enum_send} which is described in the appendix.

We denote the case in row $i$ and column $j$ in Figure \ref{fig:proj6} by send6($i$, $j$). We use the same notation to denote the case in Figure \ref{fig:proj5} by send5($i$, $j$) and  Figure \ref{fig:proj4_deg5} by send4($i$, $j$), respectively.  

Using Lemma \ref{lem:vsends} we prove the following.

\begin{lem}\label{lem:edgesends}\showlabel{lem:edgesends}
Assume that none of our reducible configurations in the set $\mathcal{K}$ is a subgraph in $G'$. Let $vuw$ be a triangle face in $G'$, and degree of $w$ be at least 7. Then, $\phi(u,w) + \phi(v,w) \leq 10$.  If the equality holds, $\phi(u,w) = 6$ and $\phi(v,w) = 4$ (or the symmetric case with $u,v$ interchanged) or $\phi(u,w) = \phi(v,w) = 5$. Figure~\ref{fig:6+4=10} shows all cases when the first of these possibilities happens.
\end{lem}
\begin{proof}
    By Lemma \ref{lem:vsends}, we need to show $\phi(v,w) \leq 4$ when $\phi(u,w) = 6$. There are three cases in which a vertex sends charge $6$ by Lemma \ref{lem:vsends}. In any case, $d(u)=6$ or $7$ and $d(v) = 5$. A vertex of degree $5$ does not send charge $6$, so $\phi(v,w) \leq 5$. Let $x$ be a vertex that is adjacent to both $v$ and $w$, and that is different from $u$. 

    The degree of $x$ can be $5$ only when $u$ sends charge $6$ in send6(1, 2) since conf(1) appears in all other cases. 
    When send6(1, 2) happens and the degree of $x$ is $5$, $\phi(v,w)$ is at most 4 by Lemma \ref{lem:vsends}.

    When the degree of $x$ is $6$, $\phi(v, w)$ is exactly $4$ when $u$ sends charge $6$ in send6(1, 1) and in send6(1, 2), but $\phi(v, w)$ is at most $4$ by Lemma \ref{lem:vsends}.

    When the degree of $x$ is at least $7$, $\phi(v, w)$ is at most 3 by Lemma  \ref{lem:vsends}. 

\end{proof}

\tikzset{deg5/.style={thick, circle, draw, fill=black, inner sep=1.5pt,}}
\tikzset{deg6/.style={thick, circle, draw, fill=black, inner sep=0pt,}}
\tikzset{deg7/.style={thick, circle, draw, fill=white, inner sep=2pt,}}
\tikzset{deg8/.style={thick, rectangle, draw, fill=white, inner sep=2pt,}}
\tikzset{deg9/.style={thick, regular polygon, regular polygon sides=3, rotate=180, draw, fill=white, inner sep=1pt,}}
\tikzset{deg10/.style={thick, regular polygon, draw, fill=white, inner sep=2pt,}}
\tikzset{->-/.style={decoration={
    markings,
    mark=at position .6 with {\arrow{>}}}, postaction={decorate}}}
\tikzset{->>-/.style={decoration={
    markings, 
    mark=at position .5 with {\arrow{>}};, 
    mark=at position .6 with {\arrow{>}};}, postaction={decorate}}}

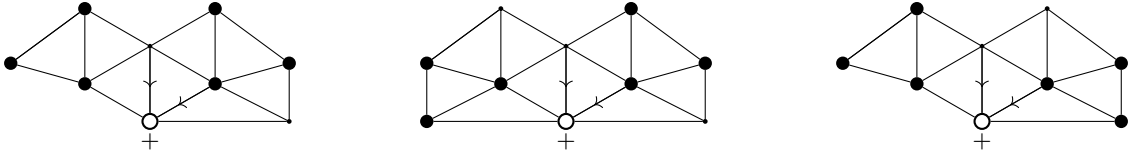
\begin{figure}[htbp]
\begin{minipage}{0.33\linewidth}
\centering
\captionsetup{width=.95\linewidth}
\begin{tikzpicture} [baseline=1.5cm]
    \node [deg6] at (2.151, 1.0) (v0) {};
    \node [deg7] at (2.151, 0.0) (v1) {};
    \node [below = 0.15 cm of v1, anchor=center] (v1+) { $+$ };
    \node [deg5] at (1.285, 0.5) (v2) {};
    \node [deg5] at (1.285, 1.5) (v3) {};
    \node [deg5] at (3.017, 0.5) (v4) {};
    \node [deg5] at (3.017, 1.5) (v5) {};
    \node [deg5] at (0.3, 0.774) (v6) {};
    \node [deg5] at (4.002, 0.774) (v7) {};
    \node [deg6] at (4.002, 0.0) (v8) {};
    \draw [->-] (v0) -- (v1);
    \draw [->-] (v4) -- (v1);
    \foreach \u / \v in {v0/v1, v0/v2, v0/v3, v0/v4, v0/v5, v1/v2, v1/v4, v2/v3, v4/v5, v3/v6, v2/v6, v3/v6, v4/v7, v5/v7, v1/v8, v4/v8, v7/v8}
        \draw (\u) -- (\v);
\end{tikzpicture}
\end{minipage}
\begin{minipage}{0.33\linewidth}
\centering
\captionsetup{width=.95\linewidth}
\begin{tikzpicture} [baseline=1.5cm]
    \node [deg6] at (2.151, 1.0) (v0) {};
    \node [deg7] at (2.151, 0.0) (v1) {};
    \node [below = 0.15 cm of v1, anchor=center] (v1+) { $+$ };
    \node [deg5] at (1.285, 0.5) (v2) {};
    \node [deg6] at (1.285, 1.5) (v3) {};
    \node [deg5] at (3.017, 0.5) (v4) {};
    \node [deg5] at (3.017, 1.5) (v5) {};
    \node [deg5] at (0.3, 0.774) (v6) {};
    \node [deg5] at (4.002, 0.774) (v7) {};
    \node [deg6] at (4.002, 0.0) (v8) {};
    \node [deg5] at (0.3, 0) (v9) {};
    \draw [->-] (v0) -- (v1);
    \draw [->-] (v4) -- (v1);
    \foreach \u / \v in {v0/v1, v0/v2, v0/v3, v0/v4, v0/v5, v1/v2, v1/v4, v2/v3, v4/v5, v3/v6, v2/v6, v3/v6, v4/v7, v5/v7, v1/v8, v4/v8, v7/v8, v1/v9, v2/v9, v6/v9}
        \draw (\u) -- (\v);
\end{tikzpicture}
\end{minipage}
\begin{minipage}{0.33\linewidth}
\centering
\captionsetup{width=.95\linewidth}
\begin{tikzpicture} [baseline=1.5cm]
    \node [deg6] at (2.151, 1.0) (v0) {};
    \node [deg7] at (2.151, 0.0) (v1) {};
    \node [below = 0.15 cm of v1, anchor=center] (v1+) { $+$ };
    \node [deg5] at (1.285, 0.5) (v2) {};
    \node [deg5] at (1.285, 1.5) (v3) {};
    \node [deg5] at (3.017, 0.5) (v4) {};
    \node [deg6] at (3.017, 1.5) (v5) {};
    \node [deg5] at (0.3, 0.774) (v6) {};
    \node [deg5] at (4.002, 0.774) (v7) {};
    \node [deg5] at (4.002, 0.0) (v8) {};
    \draw [->-] (v0) -- (v1);
    \draw [->-] (v4) -- (v1);
    \foreach \u / \v in {v0/v1, v0/v2, v0/v3, v0/v4, v0/v5, v1/v2, v1/v4, v2/v3, v4/v5, v3/v6, v2/v6, v3/v6, v4/v7, v5/v7, v1/v8, v4/v8, v7/v8}
        \draw (\u) -- (\v);
\end{tikzpicture}
\end{minipage}
\caption{When $w$ is the vertex with two arrows and $u$ is an arrow starting vertex of degree 6 and $v$ is an arrow starting vertex of degree 5, $\phi(u,w) = 6, \phi(v,w) = 4$.}
\label{fig:6+4=10}
\end{figure}


Using these lemmas, we obtain the following result.

\begin{lem}\label{lem:12}\showlabel{lem:12}
Any vertex $v$ of degree at least 12 has final charge at most\/~$0$.
\end{lem}
\begin{proof}
    By Lemma \ref{lem:edgesends}, the average charge the vertices around $v$ send to $v$ is at most $5$. This implies that $T(v)\le0$.
\end{proof}

\begin{lem}\label{lem:onesidesend}\showlabel{lem:onesidesend}
Let $H$ be an internally 6-connected triangulation in the plane (or the projective plane). Assume that none of our reducible configurations in the set $\mathcal{K}$ is a subgraph of $H$.
Let $vuw$ be a triangle face in $H$ and assume that degree of $w$ is at least\/~$9$.
\begin{itemize}
    \item If the degree of $u$ is $5$, then $\phi(u, v) \leq 4$. If the equality holds, we have one of the first three cases shown in Figure \ref{fig:proj4_deg5} in the appendix.
    \item If the degree of $u$ is $6$, then $\phi(u, v) \leq 3$. If the equality holds, we have a situation shown in Figure~\ref{fig:proj3_deg6_oneside}.
    \item If the degree of $u$ is $7$, then $\phi(u, v) \leq 2$.
    \item If the degree of $u$ is at least $8$, then $\phi(u, v) = 0$.
\end{itemize}
\end{lem}
We check Lemma \ref{lem:onesidesend} by executing an Algorithm \ref{alg:enum_send} in the appendix and check the maximum amount of charge that a vertex sends when the hypothesis of Lemma \ref{lem:onesidesend} satisfies. 

\tikzset{deg5/.style={thick, circle, draw, fill=black, inner sep=1.5pt,}}
\tikzset{deg6/.style={thick, circle, draw, fill=black, inner sep=0pt,}}
\tikzset{deg7/.style={thick, circle, draw, fill=white, inner sep=2pt,}}
\tikzset{deg8/.style={thick, rectangle, draw, fill=white, inner sep=2pt,}}
\tikzset{deg9/.style={thick, regular polygon, regular polygon sides=3, rotate=180, draw, fill=white, inner sep=1pt,}}
\tikzset{deg10/.style={thick, regular polygon, draw, fill=white, inner sep=2pt,}}
\tikzset{->-/.style={decoration={
    markings,
    mark=at position .6 with {\arrow{>}}}, postaction={decorate}}}
\tikzset{->>-/.style={decoration={
    markings, 
    mark=at position .5 with {\arrow{>}};, 
    mark=at position .6 with {\arrow{>}};}, postaction={decorate}}}

\begin{figure}[htbp]
\begin{tabular}{ccccc}
\begin{minipage}[t]{0.2\hsize}
\centering
\begin{tikzpicture} []
    \node [deg6] at (0.8, 0.3) (v0) {};
    \node [deg7] at (1.8, 0.3) (v1) {};
    \node [above = 0.15 cm of v1, anchor=center] (v1+) { $+$ };
    \node [deg6] at (1.3, 1.166) (v2) {};
    \node [deg6] at (0.3, 1.166) (v3) {};
    \node [deg5] at (2.3, 1.166) (v4) {};
    \node [deg5] at (0.8, 2.032) (v5) {};
    \node [deg5] at (1.8, 2.032) (v6) {};
    \draw [->-] (v0) -- (v1);
    \foreach \u / \v in {v0/v1, v0/v2, v0/v3, v1/v2, v1/v4, v2/v3, v2/v4, v2/v5, v2/v6, v3/v5, v4/v6, v5/v6}
        \draw (\u) -- (\v);
\end{tikzpicture}
\end{minipage}
&
\begin{minipage}[t]{0.2\hsize}
\centering
\begin{tikzpicture} []
    \node [deg6] at (0.8, 0.3) (v0) {};
    \node [deg7] at (1.8, 0.3) (v1) {};
    \node [above = 0.15 cm of v1, anchor=center] (v1+) { $+$ };
    \node [deg5] at (1.3, 1.166) (v2) {};
    \node [deg5] at (0.3, 1.166) (v3) {};
    \node [deg5] at (1.126, 2.151) (v4) {};
    \draw [->-] (v0) -- (v1);
    \foreach \u / \v in {v0/v1, v0/v2, v0/v3, v1/v2, v2/v3, v2/v4, v3/v4}
        \draw (\u) -- (\v);
\end{tikzpicture}
\end{minipage}
&
\begin{minipage}[t]{0.2\hsize}
\centering
\begin{tikzpicture} []
    \node [deg6] at (0.8, 0.3) (v0) {};
    \node [deg7] at (1.8, 0.3) (v1) {};
    \node [above = 0.15 cm of v1, anchor=center] (v1+) { $+$ };
    \node [deg5] at (1.3, 1.166) (v2) {};
    \node [deg5] at (0.3, 1.166) (v3) {};
    \node [deg5] at (2.24, 1.508) (v4) {};
    \node [deg6] at (1.126, 2.151) (v5) {};
    \draw [->-] (v0) -- (v1);
    \foreach \u / \v in {v0/v1, v0/v2, v0/v3, v1/v2, v1/v4, v2/v3, v2/v4, v2/v5, v3/v5, v4/v5}
        \draw (\u) -- (\v);
\end{tikzpicture}
\end{minipage}
&
\begin{minipage}[t]{0.2\hsize}
\centering
\begin{tikzpicture} []
    \node [deg6] at (0.8, 0.3) (v0) {};
    \node [deg7] at (1.8, 0.3) (v1) {};
    \node [above = 0.15 cm of v1, anchor=center] (v1+) { $+$ };
    \node [deg5] at (1.3, 1.166) (v2) {};
    \node [deg6] at (0.3, 1.166) (v3) {};
    \node [deg5] at (2.24, 1.508) (v4) {};
    \node [deg5] at (1.126, 2.151) (v5) {};
    \draw [->-] (v0) -- (v1);
    \foreach \u / \v in {v0/v1, v0/v2, v0/v3, v1/v2, v1/v4, v2/v3, v2/v4, v2/v5, v3/v5, v4/v5}
        \draw (\u) -- (\v);
\end{tikzpicture}
\end{minipage}
&
\begin{minipage}[t]{0.2\hsize}
\centering
\begin{tikzpicture} []
    \node [deg6] at (0.8, 0.3) (v0) {};
    \node [deg7] at (1.8, 0.3) (v1) {};
    \node [above = 0.15 cm of v1, anchor=center] (v1+) { $+$ };
    \node [deg6] at (1.3, 1.166) (v2) {};
    \node [deg5] at (0.3, 1.166) (v3) {};
    \node [deg5] at (0.8, 2.032) (v4) {};
    \node [deg5] at (1.8, 2.032) (v5) {};
    \draw [->-] (v0) -- (v1);
    \foreach \u / \v in {v0/v1, v0/v2, v0/v3, v1/v2, v2/v3, v2/v4, v2/v5, v3/v4, v4/v5}
        \draw (\u) -- (\v);
\end{tikzpicture}
\end{minipage}
\\
\end{tabular}
\caption{A vertex of degree $6$ sends charge 3 in these cases.}
\label{fig:proj3_deg6_oneside}
\end{figure}

\tikzset{deg5/.style={thick, circle, draw, fill=black, inner sep=1.5pt,}}
\tikzset{deg6/.style={thick, circle, draw, fill=black, inner sep=0pt,}}
\tikzset{deg7/.style={thick, circle, draw, fill=white, inner sep=2pt,}}
\tikzset{deg8/.style={thick, rectangle, draw, fill=white, inner sep=2pt,}}
\tikzset{deg9/.style={thick, regular polygon, regular polygon sides=3, rotate=180, draw, fill=white, inner sep=1pt,}}
\tikzset{deg10/.style={thick, regular polygon, draw, fill=white, inner sep=2pt,}}
\tikzset{->-/.style={decoration={
    markings,
    mark=at position .6 with {\arrow{>}}}, postaction={decorate}}}
\tikzset{->>-/.style={decoration={
    markings, 
    mark=at position .5 with {\arrow{>}};, 
    mark=at position .6 with {\arrow{>}};}, postaction={decorate}}}

\begin{figure}[htbp]
\begin{minipage}{0.33\linewidth}
\centering
\captionsetup{width=.95\linewidth}
\begin{tikzpicture} [baseline=1.5cm]
    \node [deg9] at (0.5, 0) (a) {};
    \node [below = 0.4 cm of a, anchor=center] (a+) { $+$ };
    \node [deg9] at (1.5, 0) (b) {};
    \node [below = 0.4 cm of b, anchor=center] (b+) { $+$ };
    \node [deg5] at (0.1, 0.9) (c) {};
    \node [deg5] at (1.0, 0.8) (d) {};
    \node [deg5] at (1.9, 0.9) (e) {};
    \node [deg6] at (1.0, 1.6) (f) {};
    \draw [->-] (d) -- (a);
    \draw [->-] (d) -- (b);
    \foreach \u / \v in {a/b, a/c, a/d, b/d, b/e, c/d, c/f, d/e, d/f, e/f}
        \draw (\u) -- (\v);
\end{tikzpicture}
\caption{When $u$ is the vertex with two arrows and $v,w$ are two vertices of degree at least $9$, $\phi(u,v) + \phi(u,w) = 8$.}
\label{fig:twoedge8}
\end{minipage}
\begin{minipage}{0.33\linewidth}
\centering
\captionsetup{width=.95\linewidth}
\begin{tikzpicture} [baseline=1.5cm]
    \node [deg9] at (0.5, 0) (a) {};
    \node [below = 0.4 cm of a, anchor=center] (a+) { $+$ };
    \node [deg9] at (1.5, 0) (b) {};
    \node [below = 0.4 cm of b, anchor=center] (b+) { $+$ };
    \node [deg5] at (0.1, 0.9) (c) {};
    \node [deg5] at (1.0, 0.8) (d) {};
    \node [deg6] at (1.9, 0.9) (e) {};
    \node [deg5] at (1.0, 1.6) (f) {};
    \draw [->-] (d) -- (a);
    \draw [->-] (d) -- (b);
    \foreach \u / \v in {a/b, a/c, a/d, b/d, b/e, c/d, c/f, d/e, d/f, e/f}
        \draw (\u) -- (\v);
\end{tikzpicture}
\caption{When $u$ is the vertex with two arrows and $v,w$ are two vertices of degree at least $9$, $\phi(u,v) + \phi(u,w) = 7$.}
\label{fig:twoedge7}
\end{minipage}
\begin{minipage}{0.33\linewidth}
\centering
\captionsetup{width=.95\linewidth}
\begin{tikzpicture} [baseline=1.5cm]
    \node [deg9] at (1.3, 0) (a) {};
    \node [below = 0.4 cm of a, anchor=center] (a+) { $+$ };
    \node [deg9] at (0.5, 0.6) (b) {};
    \node [below = 0.4 cm of b, anchor=center] (b+) { $+$ };
    \node [deg9] at (2.1, 0.6) (c) {};
    \node [below = 0.4 cm of c, anchor=center] (c+) { $+$ };
    \node [deg5] at (1.3, 0.9) (d) {};
    \node [deg5] at (0.8, 1.5) (e) {};
    \node [deg5] at (1.8, 1.5) (f) {};
    \draw [->-] (d) -- (a);
    \draw [->-] (d) -- (b);
    \draw [->-] (d) -- (c);
    \foreach \u / \v in {a/b, a/c, a/d, b/d, b/e, c/d, c/f, d/e, d/f, e/f}
        \draw (\u) -- (\v);
\end{tikzpicture}
\caption{When $u$ is the vertex with three arrows and $v,w,x$ are three vertices of degree at least $9$, $\phi(u,v) + \phi(u,w) + \phi(u,x) = 10$.}
\label{fig:threeedge10}
\end{minipage}
\end{figure}

\begin{lem}\label{lem:twoedgesends}\showlabel{lem:twoedgesends}
Let $H$ be internally 6-connected triangulation in the plane (or the projective plane). Assume that none of our reducible configurations in the set $\mathcal{K}$ is a subgraph of $H$.
Let $vuw$ be a triangle face in $H$ and assume the degree of $v, w$ is at least $9$.
\begin{itemize}
    \item If degree of $u$ is $5$, then $\phi(u, v) + \phi(u, w) \leq 8$. If $\phi(u, v) + \phi(u, w) = 8 (, 7)$, they are in Figure \ref{fig:twoedge8}(, \ref{fig:twoedge7}) respectively. 
    \item If degree of $u$ is $6, 7$, then $\phi(u, v) + \phi(u, w) \leq 4$.
    \item If degree of $u$ is at least $8$, then $\phi(u, v) + \phi(u, w) = 0$.
\end{itemize}
\end{lem}
\begin{proof}
     We consider the case when degree of $u$ is $5$. We denote the neighbors of $u$ other than $v,w$ by $x,y,z$ such that $v,w,x,y,z$ appear in the clockwise order of neighbors of $u$. $\phi(u,v), \phi(u,w) \leq 4$ by Lemma \ref{lem:vsends}. So $\phi(u,v) + \phi(u,w) \leq 8$.
     When $\phi(u,v) + \phi(u,w) \geq 7$ either of $\phi(u,v), \phi(u,w)$ is 4. We assume $\phi(u,v) = 4$ without loss of generality. The degree of $z$ must be $5$ and the degree of $y$ must be at most $6$ by Lemma \ref{lem:vsends}. The degree of $x$ is at most $6$  otherwise $\phi(u,w) = 2$, which implies $\phi(u,v) + \phi(u,w) \leq 6$.
     
     We next consider the case when the degree of $y$ is $5$. The degree of $x$ is not $5$ otherwise, conf(1) appears 
     When the degree of $x$ is $6$, $\phi(u,v) = 4, \phi(u,w) = 3$ by our rules and Lemma \ref{lem:vsends}.

     We now consider the case when the degree of $y$ is $6$. When the degree of $x$ is $5$, $\phi(u,v) = \phi(u,w) = 4$. 
     When the degree of $x$ is $6$, $\phi(u, v)$ can be exactly $4$ only when send4(1, 3) happens. Let $s$ be a vertex that is adjacent to both $x, y$ and that is different from $u$. When the degree of $s$ is $5$, $\phi(u,w)$ must be $3$, but then conf(3) appears 
     Let $t$ be a vertex that is adjacent to $x,s$ and that is different from $y$. When the degree of $t$ is $5$ and the degree of $s$ is $6$, $\phi(u,w)$ must be $3$, but then conf(15) appears 
     In other cases, $\phi(u, w)$ is $2$ so $\phi(u, v) + \phi(u, w) \leq 6$.
     
     When the degree of $u$ is at least $7$, $\phi(u,v), \phi(u,w) \leq 2$ by Lemma \ref{lem:onesidesend}. This implies $\phi(u,v) + \phi(u,w) \leq 4$.
     
     We consider the case when the degree of $u$ is $6$. We denote the neighbors of $u$ other than $v,w$ by $x_1,x_2,x_3,x_4$ such that $v,w,x_1,x_2,x_3,x_4$ appear in the clockwise order of neighbors of $u$. By Lemma \ref{lem:onesidesend}, $\phi(u,v),\phi(u,w) \leq 3$, so we need to show that $\phi(u,w)$ is at most 1 when $\phi(u,v) = 3$. There are five cases $\phi(u, v)$ is exactly $3$ as shown in Figure \ref{fig:proj3_deg6_oneside}. Let $x_5$ be a vertex that is adjacent to $x_3,x_4$ and that is different from $u$. In each case, one of the followings happens.
     \begin{enumerate}
         \item $d(x_3) = d(x_4) = 6, d(x_5) = 5$
         \item $d(x_3) = d(x_4) = 5$
         \item $d(x_3) = 6, d(x_4) = d(x_5) = 5$
         \item $d(x_3) = 5, d(x_4) = 6, d(x_5) = 5$
     \end{enumerate}
     
     When $d(x_1)= 5$ and $d(x_2) \leq 6$, a reducible configuration appears in each case; (conf(14), conf(15) in case 1, conf(2), conf(3) in case 2, conf(3), conf(15) in case 3, conf(16), conf(17) in case 4) 
     Therefore, $d(x_2) \geq 7$ when $d(x_1) = 5$. In this case, $\phi(u,w) = 1$.

     When $d(x_1) = 6$, we denote the neighbors of $x_1$ other than $x_2, u, w$ by $x_6, x_7, x_8$ such that $x_2, u, w, x_6, x_7, x_8$ appear in the clockwise order of neighbors of $x_1$. When $\phi(u, w)$ is at least $2$, one of the followings happens.
     \begin{enumerate}
         \renewcommand{\labelenumi}{\roman{enumi}}
         \item $d(x_2) = d(x_8) = 5$,
         \item at least one of $d(x_2), d(x_8)$ is $5$ under the constraint $d(x_2), d(x_8) \leq 6$, $d(x_7) = 5$,
         \item at least one of $d(x_2), d(x_8), d(x_7)$ is $5$ under the constraint $d(x_2), d(x_8), d(x_7) \leq 6$, $d(x_6) = 5$.
     \end{enumerate}
     A reducible configuration appears in all of the combinations of degrees of $x_3,x_4,x_5$ enumerated above and degrees of $x_2,x_8,x_7,x_6$ enumerated here; (conf(18), conf(19), conf(21), conf(22), conf(23) in case 1, conf(14), conf(15), conf(16), conf(17) in case 2, conf(3), conf(14), conf(18) in case 3, conf(14), conf(15), conf(16), conf(20) in case 4) 
     Therefore, $\phi(u,w) \leq 1$.

     When $d(x_1) = 7$, $\phi(u, w) = 0$.

     We finally consider the case when the degree of $u$ is $7$ or $8$, but Lemma \ref{lem:onesidesend} directly applies. 
\end{proof}

\begin{lem}\label{lem:threeedgesends}\showlabel{lem:threeedgesends}
Let $H$ be an internally 6-connected triangulation in the plane (or the projective plane). Assume that none of our reducible configurations in the set $\mathcal{K}$ is a subgraph of $H$.
Let $u$ be a vertex in $H$ and $v, w, x$ be the neighbors of $u$ so that $v,w$ are adjacent and $w,x$ are adjacent. Assume that the degree of $v, w, x$ is at least $9$.
\begin{itemize}
    \item If the degree of $u$ is $5$, then $\phi(u, v) + \phi(u, w) + \phi(u, x) \leq 10$. If the equality holds, they are in Figure \ref{fig:threeedge10}.
    \item If the degree of $u$ is $6$ ($7$), then $\phi(u, v) + \phi(u, w) + \phi(u, x) \leq 6$ ($4$).
    \item If the degree of $u$ is at least 8, then $\phi(u, v) + \phi(u, w) + \phi(u, x) = 0$.
\end{itemize}
\end{lem}
\begin{proof}
    When the degree of $u$ is at least $6$, $\phi(u,w) = 0$, so Lemma \ref{lem:onesidesend} implies the lemma. We consider the case when the degree of $u$ is $5$. We denote the neighbors of $u$ other than $v,w,x$ by $y,z$ such that $v,w,x,y,z$ appear in the clockwise order of neighbors of $u$. Then, $\phi(u,w) = 2$ and $\phi(u,v), \phi(u,x) \leq 4$ by Lemma \ref{lem:onesidesend}. This implies $\phi(u,v) + \phi(u,w) + \phi(u,x) \leq 10$. If the equality holds, $\phi(u,v) = \phi(u,x) = 4$. In this case, the degree of $y$ and $z$ must be $5$ by Lemma \ref{lem:vsends}. This is the only case where the equality holds, as shown in Figure \ref{fig:threeedge10}.
\end{proof}

\section{Looking at the smaller side}
\label{sect:smaller}
\showlabel{sect:smaller}

In this section, we show Lemma \ref{lem:conf-in-T}, which states that a reducible configuration appears strictly inside the disk that is bounded by a cycle with some specified conditions. 
This lemma is necessary to handle 6,7-edge-cuts in Section \ref{sect:6,7-cut} and low representativity cases in Section \ref{sect:rep}. 

\begin{lem}\label{lem:conf-in-T}\showlabel{lem:conf-in-T}
Let $T$ be an internally 6-connected near-triangulation in the plane, which is a subgraph of $G'$. Let $C$ be an induced cycle in $T$, which bounds the infinite region of $T$. Assume that there is no vertex in $V(T -C)$ which is adjacent to at least four consecutive vertices of $C$ and there are more than $\frac{18}{5} \cdot |C| - 12$ edges in $T$ between $C$ and $T - C$. Then a reducible configuration in the set $\mathcal{K}$ appears in $T$.
\end{lem}

We use the following lemma to show Lemma \ref{lem:conf-in-T}.
We let $n=|C|$ and denote the number of edges between $C$ and $T - C$ by $k$.

\begin{lem}\label{lem:charge-T-C}\showlabel{lem:charge-T-C}
\[
 \sum_{v \in V(T) - V(C)} 10 \cdot (6 - d(v)) = 60 - 20n + 10k.
\]
\end{lem}
\begin{proof}
The sum of the size of the faces of $T$ is $3 (|F(T)| - 1) + n$. So $2|E(T)| = 3 (|F(T)| - 1) + n$ holds.
$|V(T)| - |E(T)| + |F(T)| = 2$ holds by Euler's formula. The combination of these two formulas implies $|E(T)| = 3|V(T)| - 3 - n$, so
\[
  \sum_{v \in V(T)} 10 \cdot (6 - d(v)) = 60 |V(T)| - 20 |E(T)| = 20 (3 + n). 
\]
The number of edges between $C$ and $T - C$ is $k$.
This implies $\sum_{v \in V(C)} d(v) = 2 |V(C)| + k = 2n + k$. So
\[
  \sum_{v \in V(C)} 10 \cdot (6 - d(v)) = 60 |V(C)| - 10 (2n + k) = 40n - 10k.
\]
Therefore \[
  \sum_{v \in V(T) \setminus V(C)} 10 \cdot (6 - d(v)) = 20 (3 + n) - (40n - 10k) = 60 - 20n + 10k.
\]
\end{proof}

We extend the embedding of $T$ to an embedding of an internally 6-connected triangulation $T'$, by adding more vertices and edges inside the face bounded by $C$, in such a way that every vertex in $V(C)$ has degree of at least 12 in $T'$ and newly added vertices have degree of at least $5$.

Let vertices of $C$ be $u_0u_1 \dots u_{n-1}$ so that $u_i$ and $u_{i+1}$ are adjacent. We sometimes use an index $i$ which is out of range to denote $u_i$, which means the index $i$ module $n$ (e.g. $u_{n}$ means $u_0$.).  For each $0 \leq i < n$, there is a vertex of $V(T-C)$ that is adjacent to both $u_i$ and $u_{i+1}$. We denote these vertices by $v_i$ ($0 \leq i < n$). We denote a set of vertices of $T - C$ that are adjacent to $u_i$ by $v_{i-1} = v_{i,0}, v_{i,1}, \dots, v_{i,n_i-1} = v_i$ so that $v_{i,j}$ and $v_{i,j+1}$ are adjacent $(0 \leq j < n_i - 1)$ , where $n_i$ is the number of vertices of $T - C$ that are adjacent to $u_i$.

If $v_i = v_{i+1} = v_{i+2}$, $v_i$ would be a vertex of $T - C$ which is adjacent to at least four consecutive vertices of $C$, which would contradict the hypothesis of Lemma \ref{lem:conf-in-T}. So this would not happen. 
 
We now calculate the total amount of charge sent from $T - C$ to $C$. In order to do so, we calculate the amount of charge sent by each edge between $T - C$ and $C$, and we compute the sum. We divide into the two cases, depending on whether $v_i$ and $v_{i+1}$ are same or not. We only need the upper bound of the amount of charge sent by each edge, which is achieved when degree of $v_i$ is $5$ here, but later we need the amount of charge sent when degree of $v_i$ is at least $6$. So we consider all these cases here.

\begin{lem}\label{lem:to-C}\showlabel{lem:to-C}
    For any $0 \leq i < n$, the following holds.
    \begin{enumerate}
        \renewcommand{\labelenumi}{(\roman{enumi})}
        \item Assume $v_{i-1} \neq v_i$ and $v_i \neq v_{i+1}$. 
        \begin{itemize}
            \item If the degree of $v_i$ is $d \le 7$, then  
            $
            \phi(v_i,u_i) + \phi(v_i,u_{i+1}) + \sum_{j=1}^{n_{i+1}-2} \phi(v_{i+1,j}, u_{i+1}) \leq 5 \cdot (n_{i+1} - 2) + f_d
            $ holds, where $f_5=8,f_6=4,f_7=4$.
            \item When the degree of $v_i$ is at least $8$,
            $
            \phi(v_i,u_i) + \phi(v_i,u_{i+1}) + \sum_{j=1}^{n_{i+1}-2} \phi(v_{i+1,j}, u_{i+1}) \leq 5 \cdot (n_{i+1}-2)
            $ holds.
        \end{itemize}
        \item Assume $v_{i-1} \neq v_i$ and $v_i = v_{i+1}$.
        \begin{itemize}
            \item When the degree of $v_i$ is $d \le 7$, $\phi(v_i,u_i) + \phi(v_i,u_{i+1}) + \phi(v_i,u_{i+2}) + \sum_{j=1}^{n_{i+2}-2} \phi(v_{i+2,j}, u_{i+2}) \leq 5 \cdot (n_{i+2}-2) + g_d$ holds, where $g_5=10,g_6=6,g_7=4$.
            \item When the degree of $v_i$ is at least $8$, $\phi(v_i,u_i) + \phi(v_i,u_{i+1}) + \phi(v_i,u_{i+2}) + \sum_{j=1}^{n_{i+2}-2} \phi(v_{i+2,j}, u_{i+2}) \leq 5 \cdot (n_{i+2}-2)$ holds.
        \end{itemize}
    \end{enumerate}
\end{lem}
\begin{proof}
First, we show (i). 
We have $\sum_{j=1}^{n_{i+1}-2} \phi(v_{i+1,j}, u_{i+1}) \leq 5 \cdot (n_{i+1}-2) + 1$ by Lemma \ref{lem:vsends} and Lemma \ref{lem:edgesends}. 
$\sum_{j=1}^{n_{i+1}-2} \phi(v_{i+1,j}, u_{i+1}) \leq 5 \cdot (n_{i+1}-2)$ unless $\phi(v_{i+1,1},u_{i+1}) = 6$ and $\phi(v_{i+1,2},u_{i+1}) = 4$, in which csae 
$d(v_i) = 5$, $d(v_{i+1,1}) = 6$, $d(v_{i,n_i-2}) = 5$, and degree of the vertex that is adjacent to $v_i, v_{i+1,1}, v_{i,n_i-2}$ is $5$ 
by Lemma \ref{lem:edgesends}. In this case, $\phi(v_i,u_i) + \phi(v_i,u_{i+1}) = 7$ by Lemma \ref{lem:twoedgesends}.
For all other cases, $\phi(v_i,u_i) + \phi(v_i,u_{i+1}) \leq 8$ by Lemma \ref{lem:twoedgesends}. 
So, $\phi(v_i,u_i) + \phi(v_i,u_{i+1}) + \sum_{j=1}^{n_{i+1}-2} \phi(v_{i+1,j}, u_{i+1}) \leq 5(n_{i+1} - 2) + 8$ holds.

When degree of $v_i$ is at least $6$, $\phi(v_i,u_i) + \phi(v_i,u_{i+1}) \leq f_{d(v_i)}$ by Lemma \ref{lem:twoedgesends}. Therefore, $\phi(v_i,u_i) + \phi(v_i,u_{i+1}) + \sum_{j=1}^{n_{i+1}-2} \phi(v_{i+1,j}, u_{i+1}) \leq 5 \cdot (n_{i+1}-2) + f_{d(v_i)}$.

Second, we show (ii).
We have $\sum_{j=1}^{n_{i+2}-2} \phi(v_{i+2,j}, u_{i+2}) \leq 5 \cdot (n_{i+2}-2) + 1$ by Lemmas \ref{lem:vsends} and \ref{lem:edgesends}. 
$\sum_{j=1}^{n_{i+2}-2} \phi(v_{i+2,j}, u_{i+2}) \leq 5 \cdot (n_{i+2}-2)$ unless $\phi(v_{i+2,1}, u_{i+2}) = 6$ and $\phi(v_{ i+2,2}, u_{i+2}) = 4$. 
Suppose $\phi(v_{i+2,1}, u_{i+2}) = 6$ and $\phi(v_{i+2,2}, u_{i+2}) = 4$. By Lemma \ref{lem:edgesends}, $d(v_{i+2,1}) = 6$, $d(v_i) = 5$ and either $u_i$ or $u_{i+1}$ must be of degree $5$, but this does not happen since the degree of $u_i,u_{i+1},u_{i+2}$ each is at least 12.
$\phi(v_i,u_i) + \phi(v_i,u_{i+1}) + \phi(v_i,u_{i+2}) \leq 10$ by Lemma \ref{lem:threeedgesends}.
So, $\phi(v_i,u_i) + \phi(v_i,u_{i+1}) + \phi(v_i,u_{i+2}) + \sum_{j=1}^{n_{i+2}-2} \phi(v_{i+2,j}, u_{i+2}) \leq 5 \cdot (n_{i+2}-2) + 10$.

When the degree of $v_i$ is at least $6$, we have $\phi(v_i,u_i) + \phi(v_i, u_{i+1}) + \phi(v_i,u_{i+2}) \leq g_{d(v_i)}$ by Lemma \ref{lem:threeedgesends}. Therefore, $\phi(v_i,u_i) + \phi(v_i, u_{i+1}) + \phi(v_i,u_{i+2}) + \sum_{j=1}^{n_{i+2}-2} \phi(v_{i+2,j}, u_{i+2}) \leq 5 \cdot (n_{i+2}-2) + g_{d(v_i)}$.
\end{proof}

\begin{lem}\label{lem:charge-to-C}\showlabel{lem:charge-to-C}
    Let $a, b$ be the number of the integer $i$ $(0 \leq i < n)$ so that (i) or (ii) in Lemma \ref{lem:to-C} holds, respectively.
    The total amount of charge sent from $T - C$ to $C$ by the discharging rules of $\mathcal{R}$ is at most $5k - 2n - b$ if none of our configurations in $\mathcal{K}$ appears in $T$. 
\end{lem}
\begin{proof}
    We can easily see $a + 2b = n$ holds. We know that the average of the amount of charge sent through each edge is at most $5$ by the two cases in Lemma \ref{lem:to-C}. For case (i), the total amount of charge is smaller by 2 compared to the case when the average charge through each edge is exactly $5$. For case (ii), the total amount of charge is smaller by 5 compared to the case when the average charge through each edge is exactly $5$.
    Therefore, the amount of the total charge sent from $T - C$ to $C$ by the discharging rules of $\mathcal{R}$ is at most $5k - (2a + 5b) = 5k - 2n - b$.
\end{proof}

\begin{proof}[Proof of Lemma \ref{lem:conf-in-T}]
     The amount of initial charge accumulated in $T - C$ is $60 - 20n + 10k$ by Lemma \ref{lem:charge-T-C}.
     The total amount of charge sent from $T - C$ to $C$ by the discharging rules of $\mathcal{R}$ is at most $5k - 2n - b$ by Lemma \ref{lem:charge-to-C}. 
     Thus the sum of charge accumulated in vertices of $T - C$ is positive after applying rules since $(60 - 20n + 10k) - (5k - 2n - b) \geq 5k - 18n + 60 > 0$ by the hypothesis of this lemma. This implies that there must be a vertex of positive final charge in $V(T) - V(C)$, and hence a reducible configuration appears in $T'$. None of our reducible configurations has a vertex of degree at least 12, so one of our reducible configurations must appear in $T$.
\end{proof}

It turns out that Lemma \ref{lem:conf-in-T} is quite useful for other cases, including the planar, the doublecross, and the apex cases. Indeed, the following holds.

\begin{coro}
Let $T$ be an internally 6-connected near-triangulation in the plane and let $C$ be an induced cycle in $T$, which bounds the infinite region of $T$. Let $\mathcal{R}, \mathcal{K}$ be rules, reducible configurations for $T$, respectively. 

Assume that 
\begin{enumerate}
    \item 
    there is no vertex in $V(T -C)$ which is adjacent to at least four consecutive vertices of $C$, 
    \item 
    there are more than $\frac{18}{5} \cdot |C| - 12$ edges in $T$ between $C$ and $T - C$, 
    \item 
    from $\mathcal{R}$,  every vertex that is in $V(T-C)$ sends charge at most five to any vertex of $C$, 
    \item 
    from $\mathcal{R}$, any vertex in $V(T -C)$ which is adjacent to exactly two consecutive vertices of $C$ sends charge at most 8 together to $C$, 
    \item 
    no vertex in $\mathcal{K}$ is of degree at least 12 in $G'$, and 
    \item 
    from $\mathcal{R}$, any vertex in $V(T -C)$ which is adjacent to exactly three consecutive vertices of $C$ sends charge at most 10 together to $C$. 
\end{enumerate}
 Then, there is a vertex in $V(T-C)$ that ends up with the final charge positive. 
\end{coro}

\section{Cuts of size 6 or 7}
\label{sect:6,7-cut}\showlabel{sect:6,7-cut}

Let us recall that $G$ is a minimal counterexample and $G'$ is the dual of the embedding of $G$. We also assume that the representativity of the embedding of both $G$ and $G'$ is at least three. In this section, we will show that one side of $G$ divided by a 6-edge-cut or a 7-edge-cut is relatively small. By Lemma \ref{minc}, we know that one side of $G$ divided by a cut of order at most five is small.
Let us observe that edges in a cut of $G$ correspond to a circuit $C$ in $G'$. Throughout this section, we shall consider $G'$. The goal of this section is to prove Lemma \ref{lem:6,7-cut} that roughly says that the disk bounded by a circuit of length six or seven contains at most three or four vertices (strictly inside). 

The proof of this lemma needs a lot of case analysis, unfortunately. 
The problem is that when we find a $C$-reducible configuration $K$ inside the circuit $C$ in $G'$, after contractions in $K$, it is still possible that the resulting graph would end up with  $K_6$, after low-vertex cut reductions. There are a lot of cases we have to handle, and for some cases, we do need some computer checks, see Subsection \ref{sect:ccheck}. 

To this end, we first handle conf(1), which is a special case of Lemma \ref{lem:6,7-cut}, because the ring of conf(1) is of length exactly six, and conf(1) consists of four vertices (which would violate Lemma \ref{lem:6,7-cut} below). The configuration conf(1) has two different contractions of size exactly six for C-reducibility. We prove that $G'$ would not result in $K_6$ after one of the contractions.   

\begin{figure}
    \centering
    \includesvg[height=4cm]{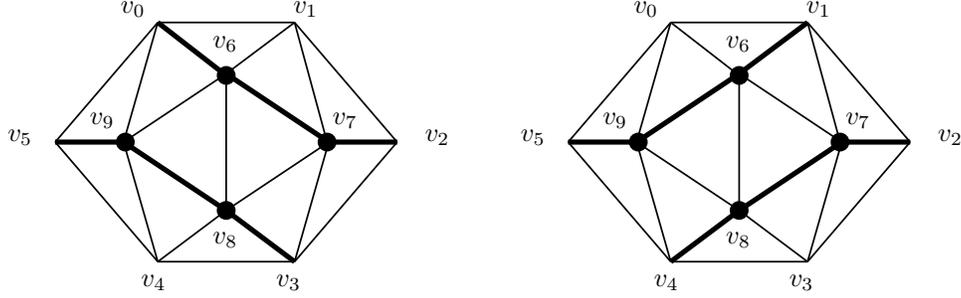}
    \caption{A free completion of $K$ with ring, where $K$ is a configuration that consists of four vertices of degree five. The bold line represents the contraction used for C-reducibility.}
    \label{fig:5555}
\end{figure}

\begin{lem}\label{lem:5555}\showlabel{lem:5555}
    If conf(1) appears in $G'$, then after some edge contraction and low-vertex cut reductions, $G'$ results in a graph in $\mathcal{P}_1 \setminus \mathcal{P}_0$.
\end{lem}

\begin{proof}
    The two contractions for C-reducibility of conf(1) are 
    $C_1 = \{v_0v_6, v_6v_7, v_7v_2, v_5v_9, v_9v_8, v_8v_3\}$ or $C_2 = \{v_1v_6, v_6v_9, v_9v_5, v_2v_7, v_7v_8, v_8v_4\}$ in Figure \ref{fig:5555}.
    We may assume that all vertices of the ring are different since otherwise, the representativity of the embedding fo $G'$ would result in at most two by contracting edges of $C_1$ or $C_2$, which in turn means that the resulting graph would not be in $\mathcal{P}_1 \setminus \mathcal{P}_0$.

    We show the following two claims. 
    \begin{claim}\label{clm:deg7}\showlabel{clm:deg7}
        At least one of $v_0, v_1$ is of degree at least 7.
    \end{claim}
    \begin{proof}
        The configurations shown in Figure \ref{fig:5555conf} are D-reducible, which can be safely reduced. If both of the degrees of $v_0, v_1$ are at most six, one of these configurations appears. We use the reducibility of these configurations, instead of conf(1) in this case. 
    \end{proof}
    We may assume that the degree of $v_0$ is at least $7$. 
    \begin{claim}\label{clm:v0v2-v0v3}\showlabel{clm:v0v2-v0v3}
        None of the edges $v_0v_2, v_0v_3$ exists.
    \end{claim}
    \begin{proof}
        Suppose that $v_0v_i$ ($i=2,3$) exists. Then, the cycle $v_0v_6v_7v_i$ exists. If the cycle is contractible, that would contradict Lemma \ref{minc}. If the cycle is noncontractible, the representativity of the resulting embedding of $G' / C_1$ would result in at most 2, which certifies $\langle G' / C_1 \rangle \neq K_6$.
    \end{proof}
    
    In order to finish our proof, we need to consider the following two cases.
    \begin{enumerate}
        \item \textbf{The degree of $v_0$ is $7$.}\\
        We first contract edges of $C_2$. If a cycle of length two or three that gets rid of $v_0$ in $G' / C_2$ exists, edges of the cycle constitute a contractible cycle of length five or six, with $v_1v_6, v_6v_9, v_9v_5$ in $G'$. This implies that there is a contractible cycle of length four or five that contains $v_0v_1, v_0v_5$.  In this case, the degree of $v_0$ must be of at most six by Lemma \ref{minc}, which is a contradiction. Therefore, $v_0$ is not deleted in $G' / C_2$. Unless $v_0$ is in some 2,3-cycle in $G' / C_2$, the degree of $v_0$ is exactly $4$ in $\langle G' / C_2 \rangle$, which certifies $\langle G' / C_2 \rangle \neq K_6$.  Assume that $v_0$ is in some 2,3-cycle in $G' / C_2$. A set of edges of the cycle constitutes a contractible cycle of length five or six, with $v_2v_7, v_7v_8, v_8v_4$ in $G'$, so the distance between $v_0$ and $v_2$ is one or two. The case when the distance is exactly one would contradict Claim \ref{clm:v0v2-v0v3}, so assume that the distance is exactly two. There is a cycle $v_0v_6v_7v_2w$ where $w$ is the vertex so that contractible cycle $v_2v_7v_8v_4v_0w$ exists in $G'$.
        The cycle $v_0v_6v_7v_2w$ is noncontractible since otherwise, this would contradict Lemma \ref{minc}. 
        
        We now contract edges of $C_1$. The representativity of the resulting embedding of $G' / C_1$ is at most two, so $\langle G' / C_1 \rangle \neq K_6$.
        \item \textbf{The degree of $v_0$ is at least $8$.}\\
        We contract edges of $C_1$. Suppose that $v_0$ is in a contractible cycle of length five or six, that contains edges $v_3v_8,v_8v_9,v_9v_5$. By Claim \ref{clm:v0v2-v0v3}, we only need to consider the case when the cycle $v_0v_5v_9v_8v_3w$ exists, where $w$ is the vertex so that contractible cycle $v_5v_9v_8v_3wv_0$ exists. If the cycle is contractible, the cycle $v_0v_6v_8v_3w$ would contradict Lemma \ref{minc}. Otherwise, the representativity of the resulting embedding of $G'$ is at most two, which is a contradiction.
        
        The vertex $v_0(,v_2)$ is adjacent to at least four(, two) vertices outside $S=\{v_0, \dots, v_9\}$ since the degree of $v_0(, v_2)$ is at least eight(, five). Hence the degree of $v_0$ in $G' / C_1$ is at least eight. Suppose $\langle G' / C_1 \rangle = K_6$. Then a contractible cycle of length five or six that contains edges of $v_0v_6, v_6v_7, v_7v_2$ in $G'$, which would result in a cycle of length two or three in $G' / C_1$ that would reduce the degree of $v_0$, exists.
        
        When the length of the cycle is six, the cycle of length five that contains $v_0v_1, v_1v_2$ exists. By Lemma \ref{minc}, we only need to consider the following cases: (1) two edges that incidents with vertices of the cycle exist, or (2) a vertex that is adjacent to all vertices on the cycle. 
        In case (1), at most one of the two edges incident with $v_0$, or $v_2$ exists, since the degree of $v_1$ in $G'$ is at least five. 
        Therefore, the degree of $v_0$ in $\langle G' / C_1 \rangle$ is at least six.
        In case (2), it is possible that the degree of $v_0$ in $\langle G' / C_1 \rangle$ is five. That would occur when the degree of $v_0, v_2$ in $G'$ is exactly $8, 5$ respectively. Assuming $\langle G' / C_1 \rangle = K_6$, $G'$ is one of the following graphs in Figure \ref{fig:5555case}. The vertex of degree exactly four or, the configuration conf(3), whose contraction size is 1, appears in $G'$. In the second case, we use the configuration instead of conf(1).
        
        Finally, when the length of the cycle is five, the vertex $w$ that is adjacent to $v_0,v_1,v_2$ exists by Lemma \ref{minc}, since the degree of $v_1$ in $G'$ is at least five. In this case, the degree of $v_1$ in $\langle G' / C_1 \rangle$ is at least six, which certifies $\langle G' / C_1 \rangle \neq K_6$.
    \end{enumerate}
\end{proof}

\begin{figure}[htbp]
  \tikzset{deg5/.style={thick, circle, draw, fill=black, inner sep=1.5pt,}}
  \tikzset{deg6/.style={thick, circle, draw, fill=black, inner sep=0pt,}}

\begin{tabular}{ccc}
\begin{minipage}[t]{0.3\hsize}
\centering
\begin{tikzpicture} [baseline=1.5cm]
    \node [deg5] at (0.5, 0) (a) {};
    \node [deg5] at (1.5, 0) (b) {};
    \node [deg5] at (0, 0.8) (c) {};
    \node [deg5] at (1.0, 0.8) (d) {};
    \node [deg6] at (2.0, 0.8) (e) {};
    \node [deg6] at (1.0, 1.8) (f) {};
    \foreach \u / \v in {a/b, a/c, a/d, b/d, b/e, c/d, c/f, d/e, d/f, e/f}
        \draw (\u) -- (\v);
\end{tikzpicture} 
\end{minipage} &
\begin{minipage}[t]{0.3\hsize}
\centering
\begin{tikzpicture} [baseline=1.5cm]
    \node [deg5] at (0.5, 0) (a) {};
    \node [deg5] at (1.5, 0) (b) {};
    \node [deg5] at (0, 0.8) (c) {};
    \node [deg5] at (1.0, 0.8) (d) {};
    \node [deg5] at (2.0, 0.8) (e) {};
    \node [deg6] at (1.0, 1.8) (f) {};
    \foreach \u / \v in {a/b, a/c, a/d, b/d, b/e, c/d, c/f, d/e, d/f, e/f}
        \draw (\u) -- (\v);
\end{tikzpicture}
\end{minipage} &
\begin{minipage}[t]{0.3\hsize}
\centering
\begin{tikzpicture} [baseline=1.5cm]
    \node [deg5] at (0.5, 0) (a) {};
    \node [deg5] at (1.5, 0) (b) {};
    \node [deg5] at (0, 0.8) (c) {};
    \node [deg5] at (1.0, 0.8) (d) {};
    \node [deg5] at (2.0, 0.8) (e) {};
    \node [deg5] at (1.0, 1.8) (f) {};
    \foreach \u / \v in {a/b, a/c, a/d, b/d, b/e, c/d, c/f, d/e, d/f, e/f}
        \draw (\u) -- (\v);
\end{tikzpicture}
\end{minipage} \\
\end{tabular}

  \caption{The reducible configurations used in the proof of Lemma \ref{lem:5555}.}
  \label{fig:5555conf}
\end{figure}
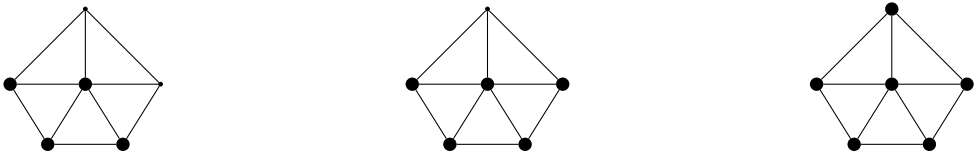

\begin{figure}[htbp]
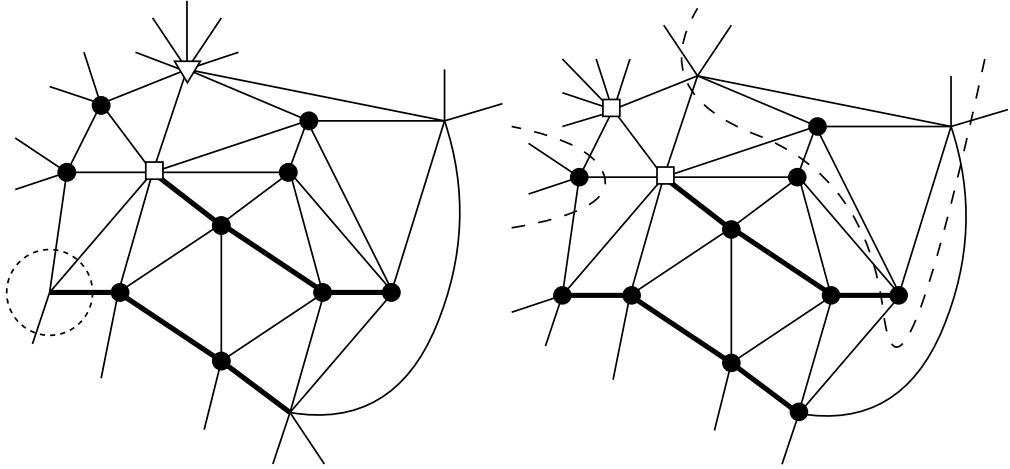

  \centering
  \includesvg[width=0.4\hsize]{./img/largecont/5555case1.svg}
  \includesvg[width=0.4\hsize]{./img/largecont/5555case2.svg}
  \caption{Two cases that conf(1) becomes $K_6$ after contraction. Each half-edge is connected to the half-edge on the opposite side, which crosses the crosscap. The bold line represents edges to be contracted. The dotted line shows the vertex of degree 4 in the left picture and conf(3) in the right picture.}
  \label{fig:5555case}
\end{figure}

\begin{lem}\label{lem:6,7-cut}\showlabel{lem:6,7-cut}
    Let $C$ be a separating circuit of length $l\in\{6, 7\}$ in $G'$ bounding a disk $D$. Suppose that no edge in $G'[C]$ inside $D$ is separating vertices inside $D$.
    \begin{itemize}
        \item If two of the connected components of $G' - C$ have at least $l-2$ vertices each, then a reducible configuration $K$ appears in $G'$ and the graph obtained from $G'$ by contracting edges of $c(K)$ is in $\mathcal{P}_1 \setminus \mathcal{P}_0$.
        \item If at most $l-3$ vertices are strictly inside the disk $D$, then $D$ is one of the graphs in Figure \ref{fig:6cut} ($l=6$) or Figure \ref{fig:7cut} ($l=7$). 
    \end{itemize}
\end{lem}

\begin{figure}[htbp]
  \centering
  \includesvg[width=0.65\hsize]{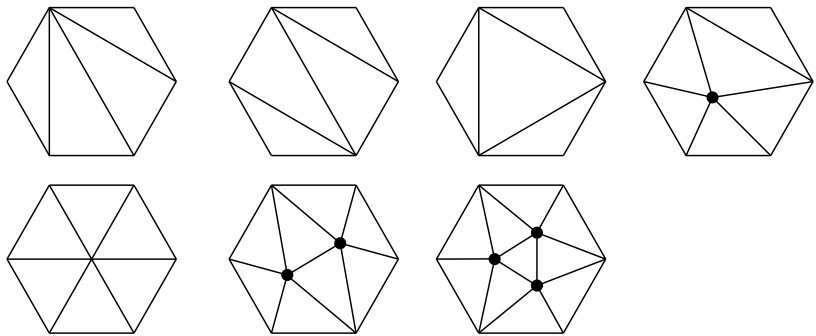}
  \caption{The disks bounded by a separating circuit of length six.}
  \label{fig:6cut}
\end{figure}

\begin{figure}[htbp]
  \centering
  \includesvg[width=0.7\hsize]{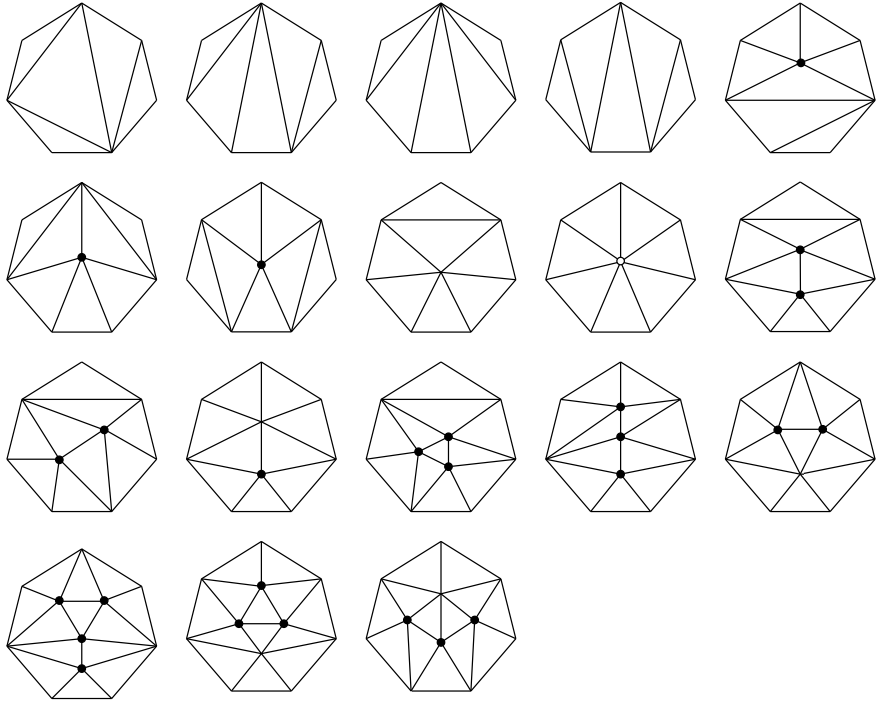}
  \caption{The disks bounded by a separating circuit of length seven.}
  \label{fig:7cut}
\end{figure}

Lemma \ref{lem:6,7-cut} states that one component of $G'$ divided by a circuit of length $l=6,7$ has at most $l-3$ vertices. Indeed, if the component is embedded inside the disk, we have the precise characterization. 
In the rest of this section, we denote vertices of $C$ by $u_i$ $(0 \leq i < l)$.

\subsection{Computer check}\label{sect:ccheck}\showlabel{sect:ccheck}

Now we explain the use of computer in the proof of Lemma \ref{lem:6,7-cut}. In our computer check, we assume that $C$ is a cycle such that the disk bounded by $C$ is minimal (in terms of the number of vertices contained in this disk) while satisfying the first condition of Lemma \ref{lem:6,7-cut} (i.e., at least $l-2$ vertices in each component of $G'-C$). Note that the length of $C$ can be six or seven. The minimality condition implies all cycles of length $k$ (where $k=6$ or $7$) in the disk bounded by $C$ bounds the disk with at most $k-3$ vertices strictly inside.
We use Lemma \ref{lem:conf-in-T} to find a reducible configuration $W$ in the disk bounded by $C$ in $G'$ later in this section, but we have to be careful not to get $K_6$ in $G'$ after contractions in $K$ (when $K$ is C-reducible) and the low-vertex cut reductions in the resulting graph. 
We want to know whether or not low-vertex cut reductions that get rid of vertices outside the disk bounded by $C$ occur.
In order to figure out, we check the resulting distance between any two vertices of $C$ in the disk after contractions in $K$. If the distance of some two vertices of $C$ is 0 or 1, such low-vertex cut reductions may happen.
When some two vertices of $C$ become identified with each other after contractions in $K$, these two vertices are in the ring of $K$.
When some vertices on $C$ become adjacent after contractions in $K$, these two vertices are in the ring of $K$, or one of them is in the ring, and the other is adjacent to a vertex of the ring. 

Let us explain how to check the argument above, using a specific example. In fact, this example corresponds to (6cut-1) enumerated in Table \ref{table:comp67cut}. We use $d_G(\cdot,\cdot)$ to denote the distance between two vertices in a graph $G$. We sometimes omit $G$ when the underlying graph $G$ is clear.
Let us assume $|C|=6$ and $d(u_0, u_2) = d(u_3,u_5) = 0$ after contraction in $K$. 
Note that these four vertices belong to a ring of $K$.
For each configuration $K \in \mathcal{K}$, we enumerate four vertices $a, b, c, d$ in the ring such that $a, b, c, d$ are in the clockwise order listed in the ring of $K$, and $d(a, b) = d(c, d) = 0$ after contracting edges of $c(K)$.
In this case, we have several constraints that two vertices of the ring of $K$ are connected by a path of specified length outside $K$ (e.g. a path of length two or four between $a$ and $b$ exists). We check by computer whether or not all of such constraints are valid. 

In doing so, we now explain two criteria to ensure validity of the process.
Let $K$ be a configuration. Let $S$ be a free completion of $K$ with the ring. We assume that a path $P$ of length $k$ between two vertices $u, v$ in the ring (and also in $C$) exists outside $K$. The path $P$ is the part of $C$. There are two paths that connect $u$ and $v$ along the ring of $K$. One of them constitutes the cycle with $P$ that bounds the disk where no vertex of the configuration $K$ exists strictly inside. We denote the path by $Q$ and $q = |Q|$.
The symbol $R$ denotes one of the shortest paths between $u$ and $v$ in $S$ and $m = |R|$. We consider two cycles $D = C - P + Q$ and $E = P + R$. 

Two criteria we use are the following.

\begin{itemize}
    \item The cycle contained in $D$ or $E$ contradicts Lemma \ref{minc}.
    \item The cycle contained in $D$ or $E$ contradicts the minimality of $C$, but we do not apply it if the length of $C$ is six, and the length of $D$ or $E$ is seven.
\end{itemize}
By definition, $D$ is a cycle. 
When $E$ is not a cycle but a circuit, it is a contractible circuit, so it contains several cycles. The length of such cycles is at most $|E|-2$.
In the case where $|E|=7$ and the length of such cycle is five, at most one vertex is strictly inside the disk bounded by $E$ by Lemma \ref{minc}. In other cases (i.e. $|E| \leq 6$), the length of the cycle is at most four, so no vertex is strictly inside the disk bounded by $E$. Hence, we can get a contradiction by assuming $E$ is a cycle and counting the number of vertices in the disk bounded by $E$, so we assume $D$ and $E$ is a cycle later.

In the second criterion, we add the condition that the length of $C$ is not six, or the length of $D$ or $E$ is not seven. We need this condition in the first part of the proof in Section \ref{subsect:7-cut}.

We now explain the two criteria more precisely. We call the subroutine \emph{forbiddenCycle}($u$, $v$, $k$). The subroutine returns true if and only if there is a cycle that leads to a contradiction.
\begin{itemize}
    \item When $q = k$, it is possible $P$ is $Q$, so return false.
    \item When $q < k$, the length of $D$ is shorter than that of $C$. Since the length of $C$ is $l$, the length of $D$ is at most $l-1 \leq 6$. The size of components divided by $D$ is at least four since $D$ includes all vertices of the configuration $K$ in one component and the size of the other component is larger than that of the component divided by $C$, which is at least $l-2 \geq 4$. 
    When $l=6$, $D$ contradicts Lemma \ref{minc}. When $l=7$, $D$ contradicts the minimality of $C$.
    \item When $q > k$, we use $\Delta$ to denote the disk bounded by $E$. We check whether or not the order of $E$ is at most $l \leq 7$. Also, we calculate the size of vertices strictly inside $\Delta$.
    Let $s$ and $t$ be the number of vertices in $\Delta$ that are in the ring of $K$, and in $K$ respectively. At least $\lceil (s-(k-1))/2 \rceil+t$ vertices are in $\Delta$ since some vertices in the ring may belong to $P$, and a vertex in the ring may be same as some other vertex in the ring. The order of the other component divided by $E$ is larger than that of a component divided by $C$, which is at least $l-2$. The cycle $E$ would contradict Lemma \ref{minc} or the minimality of $C$ unless $C=E$ depending on the order of $E$ calculated. 
\end{itemize}

We now explain another case when the criteria can be applied. 
Let $K$ be a configuration, and $S$ be a free completion of $K$ with the ring. Let $u,v$ be two vertices in the ring. We assume $u \in V(C), v \not \in V(C)$, and $v$ is adjacent to $w$, which is not in $S$ but in $C$. We also assume that a path $P$ of length $k$ between two vertices $u, w$ exists outside $K$. The path $P$ is the part of $C$. There are two paths that connect $u$ and $v$ along the ring of $K$. One of them constitutes the cycle with $P + vw$ that bounds the disk where no vertex of the configuration is strictly inside. 
We denote the path by $Q$ and $q = |Q|$.
The symbol $R$ denotes one of the shortest paths between $u$ and $v$ in $S$ and $m = |R|$. We use two cycles $D = C - P + Q + vw, E = P + R + vw$. Again even if $E$ is not a cycle but a circuit, we can get a contradiction by assuming $E$ is a cycle and calculating the number of vertices.
We use the same criteria described before.
The important point here is that both $D$ and $E$ are contained in the disk bounded by $C$ since $v$ is strictly inside the disk bounded by $C$, while $v \in V(D), v \in V(E)$. This leads to a contradiction of the minimality of $C$, depending on the number calculated as before. 
We call the subroutine \emph{forbiddenCycleOneEdge}($u$, $v$, $k$) used in this situation.



We check the combinations of the distance of vertices of $C$ in $G'$ enumerated as Table \ref{table:comp67cut}, from (6cut-1) to (7cut-4).  We show the visual representations of each case in Figure \ref{fig:check6,7cut}.
The graph $H'$ represents the disk bounded by $C$, which is formally defined later.
We also show which contraction for each configuration leads to such combinations of the distances. When we use $r_i$ to denote a vertex of the ring of $K$, it means $K$ appears and $r_i = u_i$. Now, we explain the kinds of subroutines called to check validity in each situation.

When $u_i$ and $u_j$ ($i\neq j$) belong to the ring of $K$, there are two paths of length $|i-j|$, $l-|i-j|$ between $r_i$ and $r_j$ outside $K$.
Therefore, we run two subroutines fobiddenCycle($r_i$, $r_j$, $|i-j|$), fobiddenCycle($r_i$, $r_j$, $l-|i-j|$) for $K$. If one of them returns true, it is not valid.

When $u_i$ belongs to the ring of a configuration $K$ and $u_j$ is incident with a vertex $v$ in the ring of $K$, there are two paths of length $|i-j|, l-|i-j|$ between $r_i$ and $u_j$. Therefore we run two subroutines fobiddenCycleOneEdge($r_i$, $v$, $|i-j|$), fobiddenCycleOneEdge($r_i$, $v$, $l-|i-j|$) for $K$. If one of them returns true, it is not valid.

\begin{table}[htbp]
    \caption{The combinations of the distance in $H' / c(K)$ and the situation in $S / c(K)$ that happens, where $K$ is a configuration and $S$ is a free completion of $K$ with ring.}
    \label{table:comp67cut}
    \centering
    \begin{tabular}{ccc}
    \toprule
    & 
    the combinations of the distance in $H' / c(K)$ & 
    the situation to check in $S / c(K)$ \\
    \midrule
    (6cut-1) &
    $d_{H' / c(K)}(u_0, u_2) = d_{H' / c(K)}(u_3, u_5) = 0$ & 
    $d_{S / c(K)}(r_0, r_2) = d_{S / c(K)}(r_3, r_5) = 0$ \\
    \midrule
    (6cut-2) &
    $d_{H' / c(K)}(u_0, u_2) = d_{H' / c(K)}(u_2, u_4) = 0$ & 
    $d_{S / c(K)}(r_0, r_2) = d_{S / c(K)}(r_2, r_4) = 0$ \\
    \midrule
    (7cut-1) & 
    \begin{tabular}{c}
        $d_{H' / c(K)}(u_0, u_3) = d_{H' / c(K)}(u_0, u_5)$ \\
        $= d_{H' / c(K)}(u_3, u_5) = 0$
    \end{tabular} & 
    \begin{tabular}{c}
        $d_{S / c(K)}(r_0, r_3) = d_{S / c(K)}(r_0, r_5)$ \\
        $= d_{S / c(K)}(r_3, r_5) = 0$
    \end{tabular} \\
    \midrule
    (7cut-2) & 
    $d_{H' / c(K)}(u_0, u_3) = d_{H' / c(K)}(u_4, u_6) = 0$ &
    $d_{S / c(K)}(r_0, r_3) = d_{S / c(K)}(r_4, r_6) = 0$ \\
    \midrule
    (7cut-3) &
    \begin{tabular}{c}
        $d_{H' / c(K)}(u_0, u_3) = 0$ \\
        $d_{H' / c(K)}(u_0, u_5) = d_{H' / c(K)}(u_3, u_5) = 1$
    \end{tabular} &
    \begin{tabular}{c}
        $d_{S / c(K)}(r_0, r_3) = 0$ \\
        $d_{S / c(K)}(r_0, r_5) = d_{S / c(K)}(r_3, r_5) = 1$, or \\
        a vertex $v$ of the ring is adjacent to $u_5$ \\
        satisfies $d_{S / c(K)}(r_0, r_3) = 0$ \\
        $d_{S / c(K)}(v, r_0) = d_{S / c(K)}(v, r_3) = 0$, or \\
        vertices $v_1, v_2$ of the ring are adjacent to $u_5$ \\
        satisfies $d_{S / c(K)}(r_0, r_3) =$ \\
        $d_{S / c(K)}(r_0, v_1) = d_{S / c(K)}(r_3, v_2) = 0$.
    \end{tabular} \\
    \midrule
    (7cut-4) & 
    $d_{H' / c(K)}(u_0, u_3) = 0, d_{H' / c(K)}(u_4, u_6) = 1$ &
    \begin{tabular}{c}
        $d_{S / c(K)}(r_0, r_3) = 0, d_{S / c(K)}(r_4, r_6) = 1$, or \\
        a vertex $v$ of the ring incidents with $u_4$ \\
        satisfies $d_{S / c(K)}(r_0, r_3) = 0, d_{S / c(K)}(v, r_6) = 0$.
    \end{tabular} \\
    \bottomrule
    \end{tabular}
\end{table}
\begin{figure}[htbp]
    \centering
    \includesvg[width=0.7\hsize]{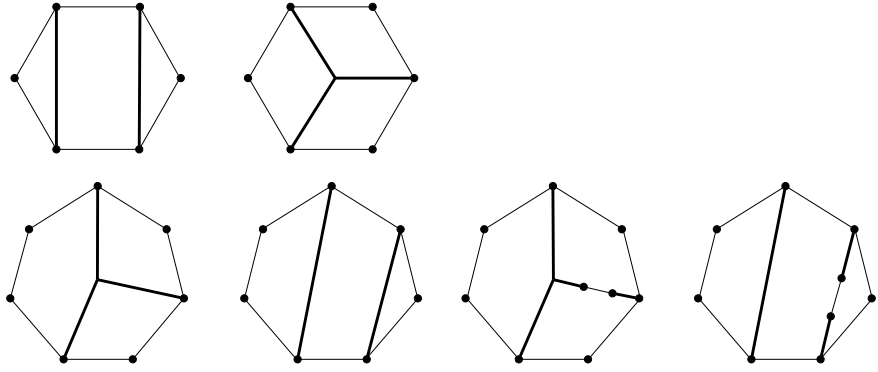}
    \caption{The six cases that were checked by computer. The graphs in the upper row correspond to (6cut-1,2) in Table \ref{table:comp67cut}, and those in the lower row correspond to (7cut,1,2,3,4) in Table \ref{table:comp67cut} from left to right. In the graphs in the upper row, $u_0$ is the vertex in the upper left, while in the lower row, $u_0$ is always the vertex at the top. The remaining vertices $u_1,u_2,...$ start after $u_0$ in the counter-clockwise order. The bold line represents contraction edges of $K$.}
    \label{fig:check6,7cut}
\end{figure}

We show the results. Only conf(1) is feasible in (6cut-1). We have already handled conf(1) in Lemma \ref{lem:5555}.
There are no valid cases in (6cut-2), (7cut-1,2,3,4).

From these results, we can conclude the following.
\begin{claim}\label{clm:comp-cut}\showlabel{clm:comp-cut}
    The combinations of distance enumerated in Table \ref{table:comp67cut} from (6cut-1) to (7cut-4) in $G' / c(K)$ would not occur for $K \in \mathcal{K}$ except $conf(1)$.
\end{claim}

\subsection{The common part of the proof for $l=6,7$.}
\label{subsect:common-6,7cut}
\showlabel{subsect:common-6,7cut}

This section contains the common part of the proof of Lemma \ref{lem:6,7-cut} for $l=6,7$. After that, we describe a specific part of the proof for each case.




We consider the two cases depending on whether or not $C$ is a cycle, but we discuss the common things first. After that, we discuss the case where $C$ is not a cycle, which is easier, and the other case separately.
We obtain the graph $H', H''$ in the following way whether $C$ is a cycle or not.

\begin{enumerate}
    \renewcommand{\labelenumi}{(\roman{enumi})}
    \item
    The case that $C$ is a cycle. \\
    When $C$ is a cycle, $C$ is contractible. There are two graphs obtained from $G'$ divided by $C$. One of them is planar, and the other one is projective planar. Let $A, B$ be the planar, and the projective planar one respectively. The graph $H'$ is $A+C$, $H''$ is $B+C$. When $|A| \geq l-2$, we choose $C$ such that $C$ is minimal subject to this condition. (i.e. no cycle of length $k(=6,\text{or} 7)$ contained in the disk bounded by $C$ bounds the disk at least $k-2$ vertices are strictly inside.). 
    \item
    The case that $C$ is not a cycle. \\
    When $C$ contains no noncontractible cycle, $C$ consists of several cycles of length at most five. That contradicts the hypothesis that no proper subset of $C$ is separating.

    When $C$ contains a noncontractible cycle, the vertex $u_i$ cannot be same as $u_{i+1}$, nor $u_{i+2}$ since it implies that the representativity of the embedding of $G'$ is at most two, which contradicts the assumption that the representativity of $G'$ is at least three.
    We assume $u_0=u_3$ without loss of generality. 
    Suppose that other two vertices $u_i, u_j$ of $C$ are same such that $0 < i < 3 < j$. When $i=1,j=4$, circuit $u_0u_5u_1u_2$ for $l=6$, or $u_0u_5u_5u_1u_2$ for $l=7$ is separating. Otherwise, the circuit which consists of the walk from $u_i$ to $u_0$ in $C$ and the walk from $u_j$ to $u_3$ in $C$ is separating.
    This contradicts the hypothesis that no proper subset of $C$ is separating. 
    Therefore, we can assume a circuit $C$ such that $u_0=u_3$ and no other vertices are same.
    There are two planar graphs obtained from $G'$ by deleting $C$. 
    Let $A, B$ be the planar graphs such that $|V(A)| \leq |V(B)|$. The graph $H'$ is $A+C$, $H''$ is $B+C$.
    The graph $H'$ is not planar, but splitting the vertex $u_0(=u_3)$ into $u'_0$ and $u'_3$ such that $u'_0$ is adjacent to $u_5,u_1$ and $u'_3$ is adjacent to $u_2,u_4$ makes $H'$ a planar near-triangulation. Now, we consider the four cases regarding $H'$. Note that when we refer to $H'$ in the case analysis below, we use the planar near-triangulation obtained in this way instead of $H'$.
\end{enumerate}

The graph $H'$ can be regarded as an internally 6-connected near-triangulation in the plane, and $C$ is the cycle of $H'$ that bounds the infinite region of $H'$.
We have the following four cases. Three of them are easy. The last one is the most difficult.
\begin{enumerate}
    \item $C$ is not an induced cycle in $H'$.\\
    There are two kinds of chords in $C$: a chord between two vertices of distance two or three in $C$.
    A chord between two vertices of distance two(, or three) implies $C$ consists of a cycle of length $l-1$(, $l-2$) and of length three(, four), respectively. The cycles are also in $G'$. We use Lemma \ref{minc} or the minimality of $C$. In case (ii), 
    by the hypothesis $|V(A)| \leq |V(B)|$, $H'$ is characterized. In case (i), we can characterize $H'$ unless $l=7$ and a chord exists between two vertices of distance two. However, if that happens, a cycle of length six contradicts the minimality of $C$.
    \item Case 1 does not occur, and there is a vertex in $H' - C$ which is adjacent to at least four consecutive vertices of $C$.\\
    A cycle of length $l-1$ exists in $H'$. That characterizes $H'$ by the same discussion as in case 1.
    \item Neither case 1 nor 2 occurs, and there are at most $\frac{18}{5}l - 12$ edges between $C$ and $H' - C$.\\
    When $l=6$, there are at most nine edges between $C$ and $H' - C$. By the condition that neither case 1 nor 2, there is only one possibility of $H'$, which is the last one in Figure \ref{fig:6cut}.
    When $l=7$, there are at most $13$ edges between $C$ and $H' - C$. By the condition that neither case 1 nor 2 occurs, there is a cycle of length at most six in $H'$, which characterizes $H'$. We have several candidates for $H'$. When a reducible configuration appears in $H'$, we go to case 4 (and we deal with this situation there). Only one case remains, which is the last one in Figure \ref{fig:7cut}.
    \item Neither case 1 nor 2 occurs, and there are more than $\frac{18}{5}l - 12$ edges between $C$ and $H' - C$.\\
    We use Lemma \ref{lem:conf-in-T}. We can find a reducible configuration in $H'$. We get a smaller graph and apply induction by contracting several edges of the reducible configuration.
    However, we have to be careful not to get a self-loop or $K_6$ after low-vertex cut reductions by contracting edges. (A self-loop corresponds to a bridge, and $K_6$ corresponds to the Petersen graph in $G$.)  We discuss this case below.
\end{enumerate}
In cases 1,2,3, we know the characterization of $H'$. We enumerate all characterizations in Figure \ref{fig:6cut}, \ref{fig:7cut}. All graphs in Figure \ref{fig:6cut}, \ref{fig:7cut} have at most $l-3$ vertices excluding vertices of the cycle.

Now we discuss case 4. All of our reducible configurations have at least four vertices, so the number of vertices in $H'$ excluding $C$ is at least four. When $l=7$, only conf(1) of our configurations has just 4 vertices, but we have already handled conf(1) in Lemma \ref{lem:5555}. So we assume that at least five vertices are in $H'$ excluding $C$. Also, we do not have to consider the case when the number of vertices in $G' - H'$ is at most $l-3$ in case (i) since $G' - H'$ is projective planar. We can also assume the same thing in case (ii) by the hypothesis $|V(A)| \leq |V(B)|$. In what follows, we assume that 
\begin{quote}
    $G'-H'$ has at least $l-2$ vertices. 
\end{quote}

We enumerate all of the possible low-vertex-cut reductions after contracting edges. A contractible cycle of length two or three after contracting edges leads to a low-vertex cut reduction. 
The edges to be contacted are only in $H'$ and not the edges of $C$ by Lemma \ref{lem:conf-in-T}. 
We only consider a 2,3-vertex cut that contains vertices of $G' \setminus H'$. 
We show an example. Let $a, b$ be two vertices of $C$.
When the path between $a$ and $b$ in $H'$ after contracting edges and the path between $a$ and $b$ in $H''$ constitutes a contractible cycle of length two or three, then a low-vertex cut reduction happens. 
Therefore, we analyze how two vertices of $C$ are connected in $H''$. We can easily show the following claims.
\begin{claim}\label{clm:chord-H''}\showlabel{clm:chord-H''}
A chord between two vertices of $C$ on $H''$, which possibly becomes a 2,3-vertex cut after contraction, exists only when $l=7$ and the distance between endpoints of the chord is exactly two in $C$.
\end{claim}
The chord that is not described above leads to 2,3,4, or 5-vertex-cut in $G'$, which contradicts Lemma \ref{minc} so Claim \ref{clm:chord-H''} holds.

The result of the computer check in Section \ref{sect:largecont} implies the following claim. 
\begin{claim}\label{clm:loop-6,7cycle}\showlabel{clm:loop-6,7cycle}
For a C-reducible configuration $K \in \mathcal{K}$, assume that $K$ appears in $G'$, and a loop appears after contracting edges of $c(K)$. Then, a contractible cycle of length $k\in \{6,7\}$ with the following properties exists.
\begin{itemize}
    \item The cycle contains an edge in $G'$ that becomes a loop after contraction, and 
    \item the cycle separates $G'$ into two components, each of which has at least $k-2$ vertices.
\end{itemize}
\end{claim}

\begin{claim}\label{clm:loopless}\showlabel{clm:loopless}
After the contraction of the edges of a configuration that appears in the disk bounded by the cycle $C$ of length 6 or 7, $G'$ becomes a loopless graph.
\end{claim}
\begin{proof}
    The edges to be contracted are strictly contained in the disk bounded by $C$. So the edge in $G'$ that becomes a loop after contraction must be one of the following: (i) an edge that yields a chord of $C$ in $H''$, (ii) an edge that is contained in the disk bounded by $C$. Let $e$ denote the edge described in (i) or (ii).
    A cycle in $G'$ that contains $e$ and other edges of the cycle are contracted, must exist in both the cases, which is denoted by $D$.
    By Claim\ref{clm:loop-6,7cycle}, $D$ is of length $k(=6,7)$ and separates $G'$ into two components, each of which has at least $k-2$ vertices. 
    In case (i), $e$ is a chord between two vertices of distance exactly two in $C$ by Claim\ref{clm:chord-H''}. The edges of $D$, except $e$, are strictly contained in the disk bounded by $C$. Hence the order of vertices strictly inside the disk bounded $D$ is smaller than that of $C$, which contradicts the minimality of $C$.
    In case (ii), the contradiction occurs in the same way as case (i). 
\end{proof}

Indeed, this is the only place that we need the minimality of $C$ for the proof of Lemma \ref{lem:6,7-cut} (i.e, we take the cycle $C$ that satisfies the second condition in Lemma \ref{lem:6,7-cut}, and subject to that, the number of vertices inside the disk $D$ bounded by $C$ is as small as possible). Thus we may obtain a new cycle $C'$ that would lead to a contradiction to the minimality of $C$ only when Claim \ref{clm:chord-H''} occurs. 

By the same argument as used in the proof of Claim \ref{clm:loopless}, we can justify using Lemma \ref{lem:6,7-cut} to prove Lemma \ref{pairs6} in Section \ref{sect:dist5}.

\begin{claim}\label{clm:adj-cont}\showlabel{clm:adj-cont}
Assuming that a loop does not appear after contraction in $G'$, $u_i$ and $u_{i+1}$ are not identified after contraction.
\end{claim}

The following claims show the criteria to detect $\langle G' / c(K) \rangle$ is $K_6$ easily. Let us remind that $C$ is of length $l=6, 7$. 

\begin{claim}\label{clm:rm4or6}\showlabel{clm:rm4or6}
Assume that $C$ is a cycle in $G'$. In order for the resulting graph of $G'$ to be $K_6$, at least $2l-8$ vertices of $C$ and $G' - H'$ must be identified by contraction or deleted by low-vertex cut reductions.
\end{claim}

The number of vertices in $C$ is $l$ and at least $l-2$ vertices are in $G' - H'$ by the assumption, so at least $2l-2$ vertices are in $G'$. The order of $K_6$ is six, so Claim \ref{clm:rm4or6} holds.  
We use the following claim to count the number of vertices identified by contraction.

\begin{claim}\label{clm:numIdent}\showlabel{clm:numIdent}
Let $D$ be a cycle of length $k$ $(3 \leq k \leq 7)$ in $G'$. The maximum number of vertices in $D$ that are identified by contraction, such that no adjacent vertices of $D$ are identified, is $\lfloor k/2 \rfloor - 1$. 
\end{claim}

The following claim also shows the criteria to detect $\langle G' / c(K) \rangle$ is $K_6$.

\begin{claim}\label{clm:5cycK6}\showlabel{clm:5cycK6}
Let $D=d_0d_1d_2d_3d_4$ be a contractible cycle in $G'$ such that the number of vertices that are strictly contained in the disk $\Delta$ bounded by $D$ is at least two. By Lemma \ref{minc}, the edge $d_id_{i+2}$ passes through the crosscap for each $i$, $0 \leq i < 5$.

If $d(d_i,d_{i+2})$ in $G'$ is at most one in $\Delta$ after contraction and low-vertex cut reductions, $G'$ does not become $K_6$.
\end{claim}

We finished explaining the common things of cases (i) and (ii).
In the rest of this section, the symbol $K$ denotes the configuration that appears in $H'$.

\subsection{6-cuts}\label{subsect:6-cut}

\begin{proof}[Proof of Lemma \ref{lem:6,7-cut} for $l = 6$]
We will first discuss the case when $C$ is not a cycle.
From the discussion above, $C$ is a noncontractible closed walk so that $u_0=u_3$ and other vertices are distinct. 
Suppose $\langle G' / c(K) \rangle = K_6$.
No two vertices of $C$ are of distance at most one and constitute a 2,3-vertex cut with a path in $H''$ since the representativity of the resulting embedding of $G'$ becomes at most two if that happens.
Therefore, no 2,3-vertex cut that contains vertices of $G' - H'$ appears. By the hypothesis that there are at least four vertices in $G' - H'$, the number of vertices of $H'$ must become two, but at least three vertices remain after contraction by Claim \ref{clm:numIdent}.

Suppose now that $C$ is a cycle.
\label{sect:6-cycle}
We show the claims about how two vertices of $C$ are connected by a path of length two in $H''$. We also use Claim \ref{clm:dis2-u0-u2} in the proof of the case when $l=7$.

\begin{claim}\label{clm:dis2-u0-u2}\showlabel{clm:dis2-u0-u2}
Suppose that there is a vertex $v \in G' - H'$ such that $v$ is adjacent to $u_i$ and $u_{i+2}$ and there is a contractible 4-cycle $vu_iu_{i+1}u_{i+2}$. The cycle bounds a disk that has no vertex strictly inside since it is impossible that the cycle bounds a disk that contains $H'$ by Lemma \ref{minc}. By Claim \ref{clm:chord-H''}, the edge $vu_{i+1}$ exists.
\end{claim}

\begin{claim}\label{clm:dis2-u0-u3}\showlabel{clm:dis2-u0-u3}
Suppose that there is a vertex $v \in G' - H'$ such that $v$ is adjacent to $u_i$ and $u_{i+3}$ and there are two contractible 5-cycles: $vu_iu_{i+1}u_{i+2}u_{i+3}$ and $vu_{i+3}u_{i+4}u_{i+5}u_i$. There are at most two vertices in $G' - H'$ to meet Lemma \ref{minc}. That contradicts the hypothesis that at least four vertices are in $G' - H'$.
\end{claim}



Based on the claims above, we divide the proof into the following five cases. We set $i$ as some value between $0$ and $5$ in the cases. We describe the cases in Figure \ref{fig:proof6cut}.

\begin{figure}[htbp]
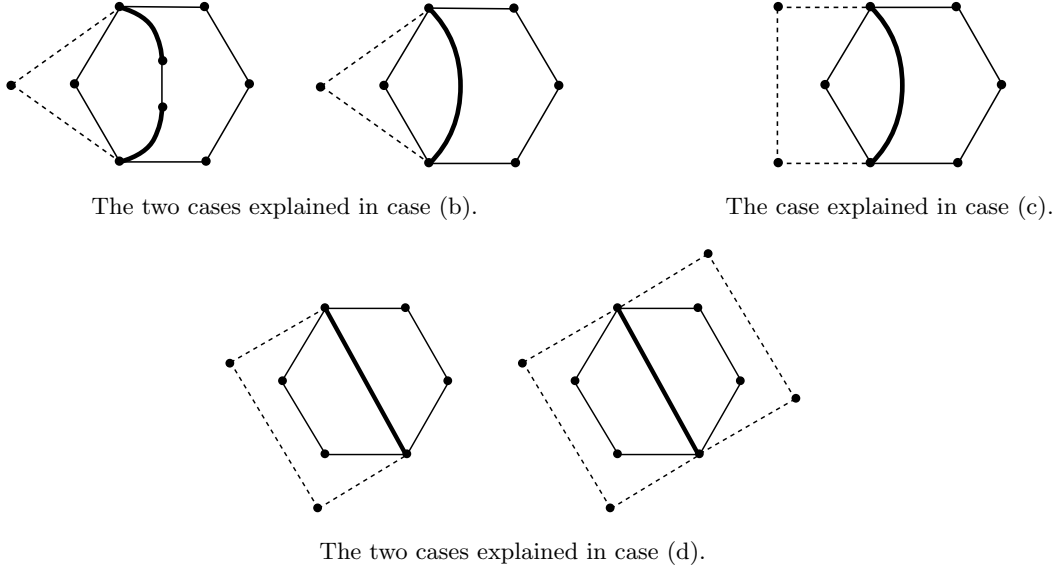

  \centering
  \begin{tabular}{c}
  \captionsetup[subfigure]{labelformat=empty, subrefformat=empty, skip=10pt}
  \begin{minipage}{0.6\hsize}
      \centering
      \includesvg[height=2.2cm]{./img/67cut/proof6cutb.svg}
      \subcaption{The two cases explained in case (b).}
      \label{fig:6cut(b)}
  \end{minipage} 
  \begin{minipage}{0.36\hsize}
      \centering
      \includesvg[height=2.2cm]{./img/67cut/proof6cutc.svg}
      \subcaption{The case explained in case (c).}
      \label{fig:6cut(c)}
  \end{minipage} \\
  \addlinespace[10pt]
  \captionsetup[subfigure]{labelformat=empty, subrefformat=empty, skip=10pt}
  \begin{minipage}{\hsize}
      \centering
      \includesvg[height=3.5cm]{./img/67cut/proof6cutd.svg}
      \subcaption{The two cases explained in case (d).}
      \label{fig:6cut(d)}
  \end{minipage} \\
  \end{tabular}
  \caption{The figures that describe the case analysis in Section \ref{sect:6-cycle}. The hexagon represents $C$, where $u_i$ is the vertex on the upper left, $u_{i+1},u_{i+2},...$ are in counter-clockwise order. The bold line represents the path that consists of contraction edges in $H'$. The dotted line represents the path in $H''$.}\label{fig:proof6cut}
\end{figure}

\begin{enumerate}
    \renewcommand{\labelenumi}{(\alph{enumi})}
    \item There are no 2,3-vertex cuts that contain vertices of $G' - H'$.\\
    The graph $G'$ does not become $K_6$ after contraction since at most two vertices are identified by Claim \ref{clm:numIdent}, but at least four vertices must be deleted or identified by Claim \ref{clm:adj-cont}.

    \item A path of length two between $u_i$ and $u_{i+2}$ in $H''$ exists, and a cycle of length two, or three that contains the path appears after contraction. No other 2,3-vertex cut appears.\\
    The vertex that is adjacent to $u_i,u_{i+2}$ is denoted by $v$. The edge $vu_{i+1}$ exists by Claim \ref{clm:dis2-u0-u2}.

    When $u_{i+2}$ is identified with $u_i$ and the cycle $vu_i$ appears after contraction, $u_{i+1}$ is deleted. In addition, at most one vertex is identified by Claim \ref{clm:numIdent}, so at most three vertices are deleted or identified, so $\langle G' / c(K) \rangle \neq K_6$ by Claim \ref{clm:rm4or6}.

    When the cycle $vu_iu_{i+2}$ appears after contraction, $u_{i+1}$ are deleted. In addition, at most one vertex is identified by Claim \ref{clm:numIdent}, so at most two vertices are deleted or identified, so $\langle G' / c(K) \rangle \neq K_6$ by Claim \ref{clm:rm4or6}.

    \item A path of length three between $u_i$ and $u_{i+2}$ in $H''$ exists, and a cycle of length three that contains the path appears after contraction. No other 2,3-cut appear.\\
    The path is denoted by $u_iv_1v_2u_{i+2}$. In this case $u_{i+2}$ is identified with $u_i$. 
    The cycle $u_iu_{i+1}u_{i+2}v_2v_1$ exists, so $\langle G' / c(K) \rangle \neq K_6$ by Claim \ref{clm:5cycK6}.
    
    Therefore, the cycle $u_iu_{i+1}u_{i+2}v_2v_1$ bounds the planar region that consists of at most one vertex by Lemma \ref{minc}. The possible pairs of vertices that may be identified in this situation are $(u_{i+3}, u_{i+5})$ and $(u_i, u_{i+4})$ by Claim \ref{clm:adj-cont} but neither of them is identified by (6cut-1), (6cut-2) in Claim \ref{clm:comp-cut}. Therefore, $u_{i+2}$ is the only vertex of $C$ that is identified, and $u_{i+1}$ and at most one vertex of $G' - H'$ are deleted by the three-vertex cut. Hence at most three vertices are deleted or identified, so $\langle G' / c(K) \rangle \neq K_6$ by Claim \ref{clm:rm4or6}.

    
    \item A path of length three between $u_i$ and $u_{i+3}$ in $H''$ exists, and a cycle of length three that contains the path appears after contraction.\\
    The path is denoted by $u_iv_1v_2u_{i+3}$. In this case, $u_{i+3}$ is identified with $u_i$, and no other vertices of $C$ are not identified by Claim \ref{clm:adj-cont}. Suppose $\langle G' / c(K) \rangle = K_6$. The degree of $u_i$ after contraction is five. The vertex $u_i$ is adjacent to four vertices $u_{i+4},u_{i+5},v_1,v_2$ in $\langle G' / c(K) \rangle$, so exactly one of the edges $v_1u_{i+5}$ or $v_2u_{i+4}$ exists. We assume that $v_1u_{i+5}$ exists, .w.l.o.g.
    The cycle $v_1v_2u_{i+3}u_{i+4}u_{i+5}$ exists, so $\langle G' / c(K) \rangle \neq K_6$ by Claim \ref{clm:5cycK6}.

    When another 2,3-vertex cut exists, this is a cycle of length three that contains another path $u_iv_3v_4u_{i+3}$ in $H''$ such that the cycle $v_1v_2u_{i+3}v_4v_3u_i$ bounds the disk that contains $H'$ by Claim \ref{clm:adj-cont} and Claim \ref{clm:dis2-u0-u3}. (If $v_1$ or $v_2$ equals $v_3$ or $v_4$, this contradicts Lemma \ref{minc} or the representativity becomes at most two after contraction, so we do not have to deal with these cases). In the same discussion before, exactly one of the edges $v_1v_3$ or $v_2v_4$ exists. We assume that $v_1v_3$ exists, w.l.o.g. The cycle $v_1v_2u_{i+3}v_4v_3$ exists, so $\langle G' / c(K) \rangle \neq K_6$ by Claim \ref{clm:5cycK6}.
    
    \item The number of 2,3-vertex cuts that contain vertices of $G' - H'$ is at least two.\\
    From the discussion before, pairs of vertices that are connected by a path in $H''$, which becomes 2,3-cut are 
    (I) $(u_i,u_{i+2}), (u_{i+2}, u_{i+4}), (u_{i+4}, u_{i+6})$,
    (II) $(u_i,u_{i+2}), (u_{i+2}, u_{i+4})$, or
    (III) $(u_i, u_{i+2}), (u_{i+3}, u_{i+5})$.
    Suppose $\langle G' / c(K) \rangle = K_6$.
    
    When the vertices of $C$ that remain in $\langle G' / c(K) \rangle$ are contained in a triangle, there is a noncontractible cycle in $\langle G' / c(K) \rangle (= K_6)$ that consists of vertices that were originally in $G' - H'$.
    A path of length two(, three) between $u_i,u_{i+2}$ constitutes a four(, five) -cycle with the path $u_iu_{i+1}u_{i+2}$, so 2,3-vertex cut that contains such a path becomes at most one(, two) edges that incident with vertices of $G' - H'$ respectively. 
    Therefore, unless three paths of length three exist in case (I), which is excluded by (6cut-2) of Claim \ref{clm:comp-cut}, at most five edges that incident with vertices of $G' - H'$ are deleted.
    By combining the fact that the degree of all vertices of $K_6$ is five, the noncontractible cycle of length three, whose sum of degree is at most 15 exists in $G'$. In fact, $G$ belongs to $\Gamma_3^{k} (3 \leq k \leq 8)$. We already handled $\Gamma_3^{k}$ in Section \ref{sect:smallcont}.

    Otherwise, two 3-vertex cuts each of which contains a path in $H''$ of length two appear in cases (II), and (III). In these cases, at most three vertices are deleted or identified so $\langle G' / c(K) \rangle \neq K_6$ by Claim \ref{clm:rm4or6}. 
\end{enumerate}
\end{proof}

\subsection{7-cuts}\label{subsect:7-cut}

\begin{proof}[Proof of Lemma \ref{lem:6,7-cut} for $l=7$]
\quad \par
When a chord of $C$ exists in $H''$, the distance between endpoints of the chord in $C$ is exactly two by Claim \ref{clm:chord-H''}, so there is a cycle of length six. In this case, we apply the same case analysis explained in Section\ref{subsect:6-cut} to this cycle. 
The different point from Section\ref{subsect:6-cut} is that the cycle is not minimal since $C$ is minimal. However, we use the minimality only when we execute computer checks. When we check the cases (6cut-1), and (6cut-2), we do not use the minimality of cycles of length seven. Hence we can apply the same case analysis here.

We will first discuss the case when $C$ is not a cycle.
From the discussion before, $C$ is a noncontractible closed walk so that $u_0=u_3$ and other vertices of $C$ are distinct. 
Pairs of vertices of $C$ so that the distance between them in $H'$ may possibly become at most one and constitutes a contractible cycle by combining a path in $H''$ after contraction, are $(u_4, u_6)$ or $(u_0, u_5)$ since the representativity becomes at most two if other pair of vertices is at distance one.
These pairs of vertices are not identified for the same reason.
When a vertex $v \in H''$ is adjacent to $u_4, u_6$, the edge $vu_5$ exists.
When the cycle $vu_4u_6$ appears after contraction, $u_5$ is deleted and in addition, at most two vertices are deleted by Claim \ref{clm:numIdent}, so at most three vertices are deleted or identified. It is same for the case $(u_0, u_5)$.
When no 2,3-vertex cut appears, at most two vertices are deleted by Claim \ref{clm:numIdent}. In each case, $\langle G' / c(K) \rangle \neq K_6$ since at least five vertices of $C$ and $G' - H'$ must be deleted or identified.

\label{sect:7-cycle}
Suppose now that $C$ is a cycle.
We show the claim that is similar to Claim \ref{clm:dis2-u0-u3} in the proof for $l=6$.
\begin{claim}\label{clm:dis2-u0-u3-7cut}\showlabel{clm:dis2-u0-u3-7cut}
Suppose that there is a vertex $v \in G' - H'$ such that $v$ is adjacent to $u_i$ and $u_{i+3}$ and there are two contractible cycles: $D_1 = vu_iu_{i+1}u_{i+2}u_{i+3}$ and $D_2 = vu_{i+3}u_{i+4}u_{i+5}u_{i+6}u_i$. Suppose that $D_1$ bounds a disk that contains $H'$. If $d_{H' / c(K)}(u_i,u_{i+3})$ is at most one, $G'$ does not become $K_6$ after contraction and low-vertex cut reductions by Claim \ref{clm:5cycK6}.
\end{claim}

We divide the proof into the following seven cases. We set $i$ as some value from 0 to 6 in the cases. We describe the cases in Figure \ref{fig:7cut}. 

\begin{figure}[htbp]
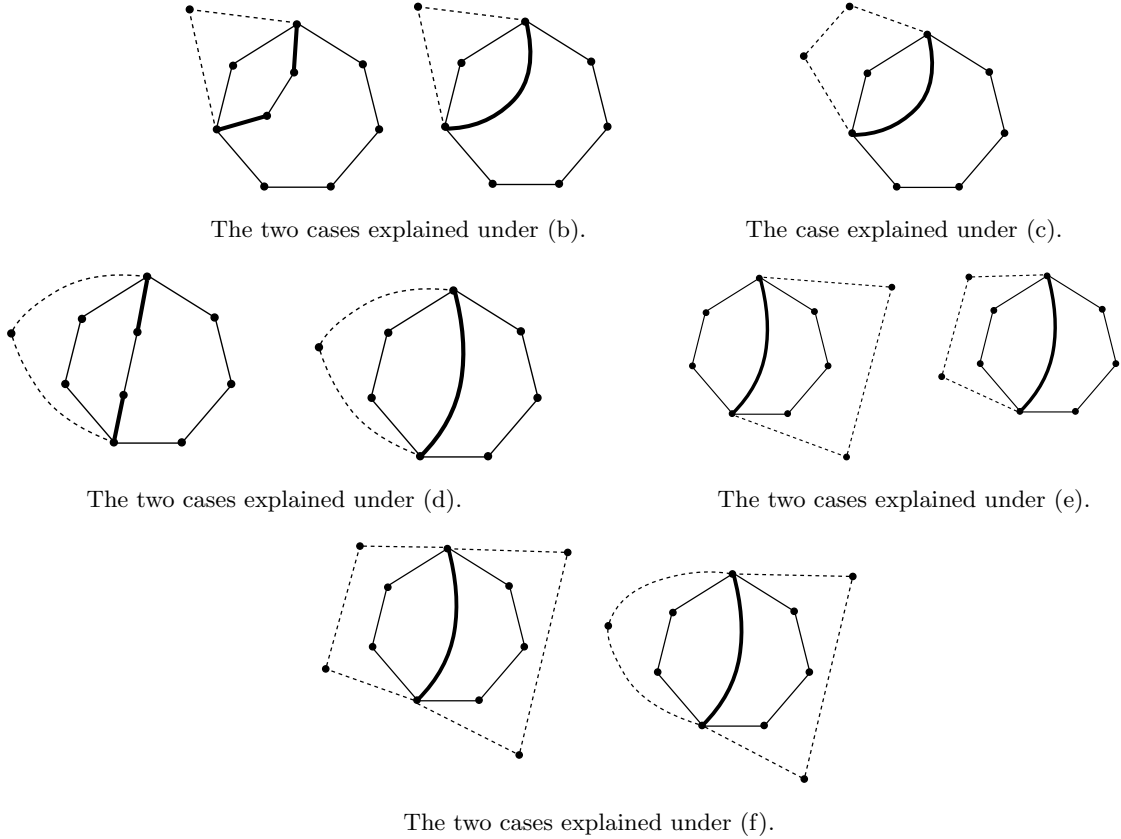

  \centering
  \begin{tabular}{c}
  \captionsetup[subfigure]{labelformat=empty, subrefformat=empty, skip=10pt}
  \begin{minipage}{0.5\hsize}
      \centering
      \includesvg[height=2.5cm]{./img/67cut/proof7cutb.svg}
      \subcaption{The two cases explained under (b).}
      \label{fig:7cut(b)}
  \end{minipage} 
  \begin{minipage}{0.3\hsize}
      \centering
      \includesvg[height=2.5cm]{./img/67cut/proof7cutc.svg}
      \subcaption{The case explained under (c).}
      \label{fig:7cut(c)}
  \end{minipage} \\
  \addlinespace[10pt]
  \captionsetup[subfigure]{labelformat=empty, subrefformat=empty, skip=10pt}
  \begin{minipage}{0.5\hsize}
      \centering
      \includesvg[height=2.5cm]{./img/67cut/proof7cutd.svg}
      \subcaption{The two cases explained under (d).}
      \label{fig:7cut(d)}
  \end{minipage} 
  \begin{minipage}{0.5\hsize}
      \centering
      \includesvg[height=2.5cm]{./img/67cut/proof7cute.svg}
      \subcaption{The two cases explained under (e).}
      \label{fig:7cut(e)}
  \end{minipage} \\
  \addlinespace[10pt]
  \captionsetup[subfigure]{labelformat=empty, subrefformat=empty, skip=10pt}
  \begin{minipage}{\hsize}
      \centering
      \includesvg[height=3.2cm]{./img/67cut/proof7cutf.svg}
      \subcaption{The two cases explained under (f).}
      \label{fig:7cut(f)}
  \end{minipage} \\
  \end{tabular}
  \caption{The figures that describe the case analysis in Section \ref{sect:7-cycle}. The heptagon represents $C$, where $u_i$ is the vertex on the upper center, $u_{i+1},u_{i+2},...$ are in counter-clockwise order. The bold line represents the path that consists of contraction edges in $H'$. The dotted line represents the path in $H''$.}\label{fig:proof7cut}
\end{figure}

\begin{enumerate}
    \renewcommand{\labelenumi}{(\alph{enumi})}
    \item There are no 2,3-vertex-cuts that contain vertices of $G' - H'$.\\
    The graph $G'$ does not become $K_6$ after contraction since at most two vertices are identified by Claim \ref{clm:numIdent}, but at least six vertices must be deleted or identified by Claim \ref{clm:adj-cont}.
    \item A path of length two between $u_i$ and $u_{i+2}$ in $H''$ exists, and a cycle of length two, or three that contains the path appears after contraction. No other 2,3-vertex cut appears.\\
    The vertex that is adjacent to $u_i,u_{i+2}$ is denoted by $v$. The edge $vu_{i+1}$ exists by Claim \ref{clm:dis2-u0-u2}. 

    When $u_{i+2}$ is identified with $u_i$ and the cycle $vu_i$ appears after contraction, $u_{i+1}$ is deleted. In addition, at most one vertex is identified by Claim \ref{clm:numIdent}, so at most three vertices are deleted or identified. Hence $\langle G' / c(K) \rangle \neq K_6$ by Claim \ref{clm:rm4or6}.

    When the cycle $vu_iu_{i+2}$ appears after contraction, $u_{i+1}$ is  deleted. In addition, at most two vertices are identified by Claim \ref{clm:numIdent}, so at most three vertices are deleted or identified.  Hence $\langle G' / c(K) \rangle \neq K_6$ by Claim \ref{clm:rm4or6}.

    \item A path of length three between $u_i$ and $u_{i+2}$ in $H''$ exists, and a cycle of length three that contains the path appears after contraction. No other 2,3-vertex-cut appears.\\
    The path is denoted by $u_iv_1v_2u_{i+2}$. In this case $u_{i+2}$ is identified with $u_i$. When the contractible 5-cycle $u_iu_{i+1}u_{i+2}v_2v_1$ bounds the disk that contains $H'$, $\langle G' / c(K) \rangle \neq K_6$ by Claim \ref{clm:5cycK6}.
    Therefore, the cycle $u_iu_{i+1}u_{i+2}v_2v_1$ bounds the disk that contains at most one vertex strictly inside, by Lemma \ref{minc}. The vertex $u_{i+1}$ and at most one vertex are deleted. In addition, at most one vertex is identified by Claim \ref{clm:numIdent}, so at most four vertices are deleted or identified. By Claim \ref{clm:rm4or6}, $\langle G' / c(K) \rangle \neq K_6$.
    
    \item A path of length two between $u_i$ and $u_{i+3}$ in $H''$ exists, and a cycle of length two, or three that contains the path appears after contraction. No other 2,3-vertex-cut appears.\\
    The vertex that is adjacent to $u_i, u_{i+3}$ is denoted by $v$. The cycle $u_iu_{i+1}u_{i+2}u_{i+3}v$ bounds the disk that contains at most one vertex of $G' - H'$ strictily inside, by Claim \ref{clm:dis2-u0-u3-7cut} and Lemmma \ref{minc}.

    When $u_{i+3}$ is identified with $u_i$ and the cycle $vu_i$ appears after contraction, the distance between $u_i$ and $u_{i+5}$ or $u_{i+4}$ and $u_{i+6}$ is at most one by Lemma \ref{minc}. We have already handled all of the cases by (7cut-1), (7cut-2), (7cut-3), and (7cut-4) in Claim \ref{clm:comp-cut}. 

    When the cycle $vu_iu_{i+3}$ appears after contraction, $u_{i+1},u_{i+2}$ and at most one vertex in $G' - H'$ are deleted. In addition, at most one vertex is identified by Claim \ref{clm:numIdent}, so at most four vertices are deleted or identified. By Claim \ref{clm:rm4or6}, $\langle G' / c(K) \rangle \neq K_6$.

    \item A path of length three between $u_i$ and $u_{i+3}$ in $H''$ exists, and a cycle of length three that contains the path appears after contraction. No other 2,3-vertex-cut appears.\\
    The path is denoted by $u_iv_1v_2u_{i+3}$. The symbol $D$ denotes the cycle $u_iu_{i+1}u_{i+2}u_{i+3}v_2v_1$. We assume that $u_{i+3}$ is identified with $u_i$.
    
    When $D$ bounds the disk that contains $H'$, $u_i$ is adjacent to four vertices $u_{i+1},u_{i+2},v_1,v_2$ in $\langle G' / c(K) \rangle$. Suppose $\langle G' / c(K) \rangle = K_6$. Exactly one of the edges $v_1u_{i+1}$ or $v_2u_{i+2}$ exists since the degree of $u_i$ must be five. We assume that $v_1u_{i+1}$ exists , w.l.o.g. The cycle $u_{i+1}u_{i+2}u_{i+3}v_2v_1$ exists, so $\langle G' / c(K) \rangle \neq K_6$ by Claim \ref{clm:5cycK6}.

    When $D$ bounds the disk that does not contain $H'$, four vertices $u_i,u_{i+4},u_{i+5},u_{i+6}$ constittues a circuit in $G' / c(K)$. 
    We have already handled all of the cases by (7cut-1), (7cut-2), (7cut-3), and (7cut-4) in Claim \ref{clm:comp-cut}. 
    
    \item Two paths of length two, or three between $u_i$ and $u_{i+3}$ in $H''$ exist, and two cycles of length two, or three that contain the paths appear after contraction.\\
    In the following explanation, $v_i(1 \leq i \leq 4)$ are pairwise different since if some of them are same, either that would contradict Lemma \ref{minc} or the representativity of the resulting embedding would become at most two after contraction.
    When there are two paths $u_iv_1v_2u_{i+3}$, $u_iv_3v_4u_{i+3}$, and $u_{i+3}$ is identified with $u_i$, $u_i$ is adjacent to four vertices $v_1,v_2,v_3,v_4$ in $\langle G' / c(K) \rangle$. Suppose $\langle G' / c(K) \rangle = K_6$. The degree of $u_i$ is five, so either $v_1v_3$ or $v_2v_4$ exists. We assume that $v_1v_3$ exists, w.l.o.g. The cycle $v_1v_3v_4u_{i+3}v_2$ exists, so $\langle G' / c(K) \rangle \neq K_6$ by Claim \ref{clm:5cycK6}.

    When there are two paths $u_iv_1v_2u_{i+3}$, $u_iv_3u_{i+3}$, and $u_{i+3}$ is identified with $u_i$, the cycle $u_iv_1v_2u_{i+3}v_3$ exists. By Claim \ref{clm:5cycK6}, $\langle G' / c(K) \rangle \neq K_6$.
    
    \item The number of 2,3-vertex cuts that contain vertices of $G' - H'$ is at least two.\\
    From the discussion above, pairs of vertices that are connected by a path in $H''$, which becomes a 2,3-vertex cut, are 
    (I) $(u_i,u_{i+3}), (u_{i+3}, u_{i+5}), (u_{i+5}, u_{i+7})$,
    (II) $(u_i,u_{i+3}), (u_{i+3}, u_{i+5})$,
    (III) $(u_i,u_{i+3}), (u_{i+4}, u_{i+6})$,
    (V) $(u_i,u_{i+2}), (u_{i+2}, u_{i+4})$, or
    (VI) $(u_i,u_{i+2}), (u_{i+3}, u_{i+5})$.
    Suppose $\langle G' / c(K) \rangle = K_6$.
    The case $d_{H' / c(K)}(u_i, u_{i+3}) = 0$ is impossible by (7cut-1), (7cut-2), (7cut-3), and (7cut-4) in Claim \ref{clm:comp-cut}.

    When the vertices of $C$ that remain in $\langle G' / c(K) \rangle$ are contained in a triangle, there is a noncontractible cycle in $\langle G' / c(K) \rangle (=K_6)$ that consists of vertices that were originally in $G' - H'$.
    A path of length two between $u_i,u_{i+3}$ constitutes a five-cycle with the path $u_iu_{i+1}u_{i+2}u_{i+3}$, so a 2,3-vertex-cut that contains such a path gets rid of at most two edges incident with vertices of $G' - H'$.
    A path of length two(, three) between $u_i,u_{i+2}$ constitutes a four(, five) cycle with the path $u_iu_{i+1}u_{i+2}$, so a 2,3-vertex cut that contains such a path gets rid of at most one(, two) edges incident with vertices of $G' - H'$ respectively.
    Therefore, unless paths of length three in $H''$ exist between $u_{i+3}$ and $u_{i+5}$ and between $u_{i+5}$ and $u_{i+7}$ in case (I), which is excluded by (7cut-1) in Claim \ref{clm:comp-cut}, at most five edges incident with vertices of $G' - H'$ are deleted.
    By combining the fact that the degree of all vertices of $K_6$ is five, the noncontractible cycle of length three, whose sum of degree is at most 15, exists in $G'$. In fact, $G$ belongs to $\Gamma_3^{k} (3 \leq k \leq 8)$. We have already handled $\Gamma_3^{k}$ in Section \ref{sect:smallcont}.

    Otherwise, one or two 3-vertex cuts, each of which contains a path in $H''$ of length two appear.
    When the number is one, $d_{H' / c(K)}(u_i,u_{i+2})=1$ in cases (V), (VI). In this case, at most five vertices are deleted or identified, so $\langle G' / c(K) \rangle \neq K_6$ by Claim \ref{clm:rm4or6}.
    When the number is two, case (II), (III), (V), or (VI). In these cases, at most five(, five, three, or three) vertices are deleted or identified so $\langle G' / c(K) \rangle \neq K_6$ by Claim \ref{clm:rm4or6}. 
\end{enumerate}
\vskip -5mm
\end{proof}

In Lemma \ref{lem:6,7-cut}, we do not handle the case that $C$ is a cycle of length $l$, and at most $l-3$ vertices are in a connected component of $G' - C$, which does not constitute a disk with $C$. In this case, the representativity of the embedding of $G'$ is at most four, which is handled in Section \ref{sect:rep}.

\medskip

{\bf Remark: } As pointed out above, we only get a new cycle $C$ that would lead to a contradiction to the minimality of $C$ only when Claim \ref{clm:chord-H''} occurs. 

 It turns out that this is necessary for our algorithm in Theorem \ref{algo}.

\section{Low representativity}
\label{sect:rep}\showlabel{sect:rep}

Let us recall that $G$ is a minimal counterexample. 
Throughout this section, we assume that the representativity of the embedding of $G$ is $y$. 
Let $\gamma$ be the non-contractible curve that certifies this representativity, and $F$ be the set of edges of $G$ that intersects $\gamma$. 
By definition, $|F| = y$. 
In this case, we are interested in the case when  $3 \leq y \leq 5$. 
The main result of this section, Lemma \ref{lem:repfinal}, says that either at least one island of a reducible configurations appears in $G - C$, where $C$ is a cycle defined in Lemma \ref{lem:repcycle}, or there is a projective island $I$ that can be safely reduced in $G$. 
To this end, we need a series of lemmas. 

\begin{lem}\label{lem:repcycle}\showlabel{lem:repcycle}
     $G$ has a cycle that contains all endpoints of $F$.
\end{lem}
\begin{proof}
    Let $\Phi_G$ be an embedding of $G$ in the projective plane.
    We obtain the planar graph $H$ by deleting all edges of $F$ from $G$.
    Let $\Phi_H$ be the restriction of $\Phi_G$ to $H$.  
    We first show that $H$ is connected.
    Suppose that $H$ has a component $C$ so that $V(H - C) \neq \emptyset$. The component $C$ contains some endpoints of $F$ since $G$ is cyclically 5-edge-connected. 
    But in this case, we can easily find a new curve $\gamma'$ across $C$ and get a non-contractible curve that intersects $\Phi_G$ less than $y$ times. This would contradict the assumption that the representativity of $G$ is $y$.
    We next show that $H$ is bridgeless. 
    Suppose $H$ has a bridge. Let $C$ be an induced subgraph of $H$ separated by the bridge. $C$ contains at least two endpoints of $F$ since $G$ is cyclically 5-edge connected. But again, we can find a noncontractible curve that intersects $\Phi_G$ less than $y$ times in the same way as above.
    
    We thus obtain a closed walk $C'$ of $H$ that surrounds the infinite region in the embedding $\Phi_H$, since $H$ is 2-edge-connected, which means that $H$ is 2-connected. 
    Any endpoint of $F$ appears in $C'$ since any edge of $F$ intersects $\gamma$. Moreover, since $H$ is 2-connected, every vertex of $C'$ appears just once in $C'$. So $C'$ is a cycle as desired. 
\end{proof}

We now try to find one reducible island in $G$.
Let $C$ be the cycle defined in Lemma \ref{lem:repcycle}.
$C + F$ is one of the projective islands in $\Pi_y^k$, where $k$ is the number of edges between $C$ and $G - C$. 
($\Pi_y^k$ is a set of projective islands of ring size $k$, defined in Section \ref{sect:smallcont}.)
Let us denote this projective island by $I$.
If $I$ is reducible and $G - c(I) \in \mathcal{P}_1 \setminus \mathcal{P}_0$, we can use $I$ to reduce $G$.

In the rest of this section, with $k$ under a certain threshold (for $y = 3,4,5$, $k \leq 3,8,12$, respectively), we shall show that every projective island in $\Pi_y^k$ is reducible. 
When $k$ exceeds this threshold, we apply the discharging method and obtain a reducible configuration in an island $I^*$. which appears in $G - I$ by Lemma \ref{lem:conf-in-T}.

\begin{lem}\label{lem:bridge-of-proj-island}\showlabel{lem:bridge-of-proj-island}
    Suppose that a projective island $I \in \Pi_y^k$ appears in $G$, and that $I$ is C-reducible with contraction size exactly four.
    Let $U$ be the set of vertices of degree two in $I$. For each $u \in U$, we now add an edge between a new vertex $v_{\infty}$ and $u$ to obtain a projective planar graph $J$.

    Then $I$ can be safely reduced, i.e. $G \dotdiv c(I)$ is not Petersen-like, unless the following happens:
    \begin{itemize}
        \item $J$ has two faces $f_1, f_2$, both containing $v_{\infty}$, and 
        \item there is a curve $\gamma$ with endpoints on $f_1$ and $f_2$ respectively, and $\gamma$ does not intersect any edge in $E(J) \setminus c(I)$. Moreover, $\gamma$ passes through all edges in $c(I)$ exactly once.
    \end{itemize}
\end{lem}

\begin{proof}
    We prove the contraposition of the lemma.
    Let us assume that it is not safe to reduce $I$ from $G$.
    By Proposition \ref{lem:4cont2}, $G - c(I)$ contains a bridge.
    Since $G$ is cyclically 5-edge-connected, there must be a 5-edge-cut of $G$ containing $c(I)$. 
    Let that 5-edge-cut be $F'$, and $e_{\mathrm{bridge}} \in F'$ be the one edge not in $c(I)$.
    If $e_{\mathrm{bridge}} \in E(I)$, $G - I$ is disconnected, and we can find either a small cyclical cut or a small representativity of $G$, which is a contradiction.
    Therefore, $e_{\mathrm{bridge}} \notin E(I)$. 
    In this case, we can find a curve $\gamma$ as described in the lemma. This completes the proof.
\end{proof}

\begin{lem}\label{lem:lowrepcomp}\showlabel{lem:lowrepcomp}
    Suppose that a projective island $I = C + F$ in $\Pi_y^k$ appears in $G$, where either $y = 3$, and $k \leq 3$, or $y = 4, k \leq 8$, or $y = 5, k \leq 12$.
    Then $I$ can be safely reduced in $G$, i.e. $G \dotdiv c(I)$ is in $\mathcal{P}_1 \setminus \mathcal{P}_0$.
\end{lem}

\begin{proof}
    If $I \in \Pi_y^k$ appears in $G$, there is a cyclical $k$-edge-cut. Thus, for $y = 3, k \leq 3$, $I$ would not appear.
    For $y = 4, k\leq 8$ and $y=5, k \leq 12$, we simply conduct a reducibility check over all graphs in $\Pi_y^k$, using computer programs.
    As a result, we confirm that every $G \in \Pi_y^k$ is either D or C-reducible with contraction size no bigger than three, except for 12 cases in $\Pi_y^k$, where the contraction size is exactly four.
    
    We now refer to Lemma \ref{lem:4cont3} and Lemma \ref{lem:bridge-of-proj-island}.
    When $|c(I)| \leq 3$, $I$ can be safely reduced. 
    When $|c(I)| = 4$, $I$ can be safely reduced if there is no $\gamma$ as described in Lemma \ref{lem:bridge-of-proj-island}.
    Using computer checks, we have checked that none of the 12 cases contain such a curve $\gamma$. (One can also check by hand as well.)
    
     The pseudocodes and the detailed results are given in the Appendix, see Section \ref{sec:arp1}.
\end{proof}

Let $\Phi_G$ be an embedding of $G$ in the projective plane. 
Apparently, $G - F$ is planar. 
We now obtain the planar graph $H$ from $G - F$ by adding a new vertex $v_{\infty}$ in the outer face $C$ of $H$, and adding edges between $v_{\infty}$ and each of the $2y$ vertices of degree two in $C$. 
We denote the dual graph of $H$ by $H'$. 
A vertex $v_{\infty}$ is of degree exactly $2y$, and all other vertices of $H$ are of degree three, so one face of $H'$ has size $2y$ and other faces are triangles. 
We denote the cycle that bounds the face of length $2y$ in $H'$ by $C'$.

\begin{lem}\label{lem:lowrephand}\showlabel{lem:lowrephand}
 If $k > \frac{26}{5} \cdot y - 12$, then a reducible configuration appears in $G - F$.
\end{lem}
\begin{proof}
    $C'$ must be an induced cycle in $H'$, as $C$ is a cycle in $G$.
    
    If there is a vertex of $H' - C'$ which is adjacent to at least four consecutive vertices of $C'$, three endpoints of $F$ in $C$ are consecutive (i.e., co-facial in $G$).
    Therefore, there is a non-contractible curve that intersects less than $y$ edges of $G$. 
    This contradicts the assumption of representativity of $G$.

    Thus there are $2y+k$ edges between $C'$ and $H' - C'$. 
    By Lemma \ref{lem:conf-in-T}, if $2y+k > \frac{18}{5} \cdot 2y - 12$, at least one of our reducible configurations appears in $H'$, so an island of the reducible configurations must appear in $G - F$, as desired.
\end{proof}

With Lemmas \ref{lem:lowrepcomp} and \ref{lem:lowrephand}, we could find a reducible island or a reducible projective island in most of the cases.
The cases left are when $y = 5$ and $13 \leq k \leq 14$.
We handle the case $y=5$, $k=14$ by examining the structure when the equality holds in Lemma \ref{lem:conf-in-T}. 

In the following proofs of the lemmas, we use the same notations used in the proof of the Lemma \ref{lem:conf-in-T}. For example, we denote the vertices of $C'$ by $u_0u_1,\dots,u_{2y-1}$ and the vertex in $H'$ which is adjacent to $u_i,u_{i+1}$ is denoted by $v_i$.

\begin{lem}\label{lem:rep5-14}\showlabel{lem:rep5-14}
If $y=5$ and $k=14$, then a reducible configuration appears in $G - F$.
\end{lem}

\begin{proof}
When $y=5$ and $k=14$, $\sum_{v \in V(H') - V(C')} 10 \cdot (6 - d(v)) = 100$ by Lemma \ref{lem:charge-T-C} by assigning $n=2y=10,k=2y+k=24$. The amount of charge sent from $V(H') - V(C')$ to $V(C')$ is at most $100 - b$ by Lemma \ref{lem:charge-to-C}. 

We only need to handle the case that $b=0$. In this case there are $24$ edges between $V(C')$ and $V(H') - V(C')$. Moreover, there are four edges between $V(C')$ and $V(H') - V(C')$, whose endpoints are different from $v_i$ $(0 \leq i < 2y)$, and there exists $i$ $(0 \leq i < 2y)$ such that $v_i$ and $v_{i+1}$ are adjacent and $v_{i+1}, v_{i+2}$ are adjacent. The equality $\phi(v_{i+1},u_{i+1}) + \phi(v_{i+1}, u_{i+2}) = 8$ must hold to send charge $100$ from $V(H') - V(C')$ to $V(C')$, so $v_i, v_{i+1}, v_{i+2}$ must all be of degree $5$, and the degree of the vertex that is adjacent to all of $v_i, v_{i+1}, v_{i+2}$ must be $6$ by Lemma \ref{lem:to-C}. Also, the sum $\phi(v_i,u_i) + \phi(v_i, u_{i+1})$ is either $7$ or $8$ to send charge $100$ from $V(H') - V(C')$ to $V(C')$, but the case of $7$ is impossible by our degree conditions above. When $\phi(v_i,u_i) + \phi(v_i, u_{i+1}) = 8$, the degree of $v_{i-1, n_{i-1}-2}$ is $5$, so conf(2) appears.
\end{proof}

For $y=5$ and $k=13$, we filter some cases in $\Pi_y^k$ by some lemmas below, and for the rest of the cases, we do some computer checks.

\begin{lem}\label{lem:rep5-13-0}\showlabel{lem:rep5-13-0}
Assume that $y=5,k=13$, no island of the reducible configurations appears in $G - C$, and there are at least two integers $i$ $(0 \leq i < 2y)$ so that $v_{i-1} \neq v_i$ and $v_i = v_{i+1}$. Then for some $j$ $(0 \leq j < 2y)$, the degree of $v_j$ is $5$.
\end{lem}

\begin{proof}
When $y=5,k=13$, $\sum_{v \in V(H') - V(C')} 10 \cdot (6 - d(v)) = 90$ by Lemma \ref{lem:charge-T-C}. The total amount of charge sent from $V(H') - V(C')$ to $V(C')$ is at most $95 - b$ by Lemma \ref{lem:charge-to-C} ($b$ is defined as in Lemma \ref{lem:charge-to-C}). By the hypothesis $b \geq 2$, the total amount of charge sent is at most $93$. 

Suppose that the degree of $v_j$ is at least $6$. 
In each case of Lemma \ref{lem:to-C}, the possible maximum amount of charge sent is lower by $4$ than that of the case when the degree of $v_j$ is $5$, so the total amount of charge sent is at most $89$.
This implies that the sum of the final charge accumulated in vertices of $V(H') - V(C')$ is positive. Thus we can find an island of reducible configurations 
by Lemma \ref{lem:conf-in-T}.
\end{proof}

\begin{lem}\label{lem:rep5-13-1}\showlabel{lem:rep5-13-1}
Assume that $y=5$ and $k=13$, no reducible configuration appears in $G - C$, and there are at least two integers $i$ $(0 \leq i < 2y)$ so that $v_{i-1} \neq v_i$ and $v_i = v_{i+1}$. 
Then for any $j$ $(0 \leq j < 2y)$, the triple
$(n_j, n_{j+1}, n_{j+2})$ is different from the following four possibilities: $(1, 2, 1)$, $(2, 1, 2)$, $(1, 2, 2)$, and $(2, 2, 1)$.
\end{lem}

\begin{proof}
Suppose that there is a $j$ such that $(n_j, n_{j+1}, n_{j+2}) = (1, 2, 1)$. 
By Lemma \ref{lem:rep5-13-0}, the degree of $v_j (= v_{j-1})$, $v_{j+1}(= v_{j+2})$ is $5$. The two vertices $v_j$ and $v_{j+1}$ are adjacent and there is a vertex $w$ which is adjacent to all of $u_{j-1},v_j,v_{j+1},u_{j+3}$. 
In the dual graph $G'$ of $G$ on the projective plane, the three vertices which correspond to $w, u_{j-1}, u_{j+3}$, respectively, form a non-contractible cycle of length $3$ in $G'$ since $u_{j-2}$ and $u_{j+3}$ are indeed the same vertex in $G'$. This contradicts the assumption $y=5$.

Suppose that there is a $j$ such that $(n_j, n_{j+1}, n_{j+2}) = (2, 1, 2)$.
By Lemma \ref{lem:rep5-13-0}, the degree of $v_{j-1}, v_j(= v_{j+1}), v_{j+2}$ is $5$. 
The three vertices $v_{j-1}, v_j, v_{j+2}$ are adjacent to each other and there is a vertex $w$ which is adjacent to $u_{j-1}, v_{j-1}, v_{j+2}, u_{j+3}$. 
In the same way as above, we derive a contradiction.

Suppose that there is a $j$ such that $(n_j, n_{j+1}, n_{j+2}) = (1, 2, 2)$.
By Lemma \ref{lem:rep5-13-0}, the degree of $v_{j-1}(=v_j), v_{j+1}, v_{j+2}$ is $5$. The vertex $v_{j-1}$ is adjacent to $v_{j+1}$ and $v_{j+1}$ is adjacent to $v_{j+2}$. There is a vertex $w$ which is adjacent to $u_{j-1}, v_{j-1}, v_{j+1}, v_{j+2}$. 
There is also a vertex $x$ which is adjacent to $w, v_{j+2}, u_{j+3}$. 
Then the four vertices which correspond to $w, x, u_{j+3}, u_{j-1}$ form a noncontractible cycle of length $4$ in $G'$, since $u_{j-2}$ and $u_{j+3}$ are indeed the same vertex in $G'$. This contradicts the assumption $y=5$.
\end{proof}

\begin{lem}\label{lem:rep5-13-2}\showlabel{lem:rep5-13-2}
Assume that $y=5,k=13$, and there are at least two integers $i (0 \leq i < 2y)$ so that $v_{i-1} \neq v_i$ and $v_i = v_{i+1}$, and there exists $j (0 \leq j < 2y)$ such that either $n_j + n_{j+5} = 3$ and $n_{j+1} + n_{j+6} = 3$, or $n_j + n_{j+5} = 3$ and $n_{j+1} + n_{j+6} = 4$.
Then, conf(1) or conf(4) appears in $G'$.
\end{lem}
\begin{proof}
When $n_j + n_{j+5} = 3$, the size of the face of $G$ which corresponds to $u_j$(and $u_{j+5}$) is $5$. Since $n_j \geq 1$ for all $j$, $n_j + n_{j+5} = 3$ implies $(n_j, n_{j+5}) = (1, 2)$ or $(2, 1)$. We assume $(n_j, n_{j+5}) = (1, 2)$ without loss of generality. Each pair of $(v_{j+4},v_{j+5})$, $(v_{j+4},u_{j+5})$, $(v_{j+5},u_{j+5})$, $(v_j,u_j)$ are adjacent. By Lemma \ref{lem:rep5-13-0}, the degree of $v_j,v_{j+4},v_{j+5}$ is $5$. Similarly, when $n_{j+1} + n_{j+6} = 3$, the size of the face of $G$ which corresponds to $u_{j+1}$(and $u_{j+6}$) is $5$. The vertex $u_{j+1}$ is  adjacent to $v_j$, $u_j$. The vertex $u_{j+6}$ is adjacent to $v_{j+5}$. 
In both the cases, the island which consists of the face that corresponds to $v_j,u_j,u_{j+1},v_{j+5}$ constitutes the island of conf(1).

When $n_{j+1} + n_{j+6} = 4$, the length of the face of $G$ which corresponds to $u_{j+1}$(and $u_{j+6}$) is $6$.
The island which consists of the face that corresponds to $v_j,u_j,u_{j+1},v_{j+5},v_{j+4}$ constitues the island of conf(4).

The diameter of both conf(1) and conf(4) is two, and the representativity of $G$ is five. 
Thus, the configuration appears in $G'$.
\end{proof}

\begin{lem}\label{lem:rep5k13-charge}\showlabel{lem:rep5k13-charge}
Assume that $y=5,k=13$. 
If one of the following conditions is satisfied, at least one island of the reducible configurations appears in $G - C$.  
\begin{itemize}
    \item Some $j$ $(0 \leq j < 2y)$ satisfies $v_j = v_{j+1}$ and the degree of the vertex is at least $7$. 
    \item For some $j_1, j_2$ $(0 \leq j_1 < j_2 < 2y)$, both $v_{j_1}, v_{j_2}$ are of degree at least $6$. 
\end{itemize}
\end{lem}
\begin{proof}
As mentioned in the proof of Lemma \ref{lem:rep5-13-0}, the initial amount of charge accumulated in $V(H') - V(C')$ is $90$, and the maximum amount of charge sent from $V(H') - V(C')$ to $V(C')$ is $95$. If the first condition holds, the total amount of charge sent is lower than the maximum by $6$ by Lemma \ref{lem:to-C}. If the second condition holds, the total amount of charge sent is lower than the maximum by $8(=4+4)$. 
This implies that the sum of the final charge accumulated in vertices of $V(H') - V(C')$ is positive. Thus we can find an island of reducible configurations 
by Lemma \ref{lem:conf-in-T}.
\end{proof}

We are now ready to prove the main result in this section. 

\begin{lem}\label{lem:repfinal}\showlabel{lem:repfinal}
Suppose that $3 \leq y \leq 5$. Then one of the following holds.
\begin{enumerate}
    \item At least one island of the reducible configurations appears in $G - C$.
    \item There is a projective island $I$ appearing in $G$, and it can be safely reduced, i.e. $G \dotdiv c(I)$ is in $\mathcal{P}_1 \setminus \mathcal{P}_0$.
\end{enumerate}
\end{lem}
\begin{proof}
    Assume $y=3$. 
    Since $G$ is cyclically 5-edge connected, $k \geq 5$. 
    By Lemma \ref{lem:lowrephand}, the first conclusion holds. 

    Assume $y=4$. 
    If $k \leq 8$, the second conclusion holds by Lemma \ref{lem:lowrepcomp}. 
    On the other hand, if $k \geq 9$ by Lemma \ref{lem:lowrephand}, the first conclusion holds.

    Assume $y=5$. 
    If $k \leq 12$, the second conclusion holds by Lemma \ref{lem:lowrepcomp}. 
    On the other hand, if $k \geq 14$ by Lemma \ref{lem:lowrephand} and Lemma \ref{lem:rep5-14}, the first conclusion holds. 

    The remaining case is when $y = 5, k = 13$. 
    If one of Lemma \ref{lem:rep5-13-1}, \ref{lem:rep5-13-2}, or \ref{lem:rep5k13-charge} applies, the first conclusion holds. 
    If neither of the three lemmas is satisfied, the projective island appearing in $G$ must be a member of $\Pi_5^{13*}$, where $\Pi_5^{13*}$ is defined as follows.

    \begin{dfn}
        Let $I$ be some projective island in $\Pi_y^{k}$, and let $C(I)$ be the path of length $2y + k$, which yields the boundary of the outer face of $I$.
        There must be $2y$ vertices of degree 3 in $C(I)$.
        Let those $2y$ vertices be labeled $c_{0}, c_{1}, \dots, c_{2y-1}$ counterclockwise, and let $x_{i}$ be the distance between $c_{i}$ and $c_{i+1}$ in $C(I)$.
        (Let the labels wrap around, so that $c_{2y} = c_0, c_{2y+1} = c_1, \dots$, etc.)
        
        Now, let $\Pi_5^{13*}$ be the set of projective islands $I$ in $\Pi_5^{13}$ that satisfies the following two conditions.
        \begin{itemize}
            \item There is at most one index $i$ satisfying $x_{i} = 0$.
            \item For each $i$ $(0 \leq i < 5)$ the following holds:
            \begin{itemize}
                \item[(a)] $x_{i} + x_{i+1} + x_{i+2} \geq 3$, and
                \item[(b)] $x_{i} + x_{i+1} + x_{i+5} + x_{i+6} \geq 4$.
            \end{itemize}
        \end{itemize}
    \end{dfn}

    Luckily, there are no non-reducible projective islands in $\Pi_5^{13*}$. 
    By Lemma \ref{lem:4cont3} if the contraction size is at most three (i.e., some of the reducible projective islands in $\Pi_5^{13*}$ are C-reducible), they are safely reduced. 
    By Proposition \ref{lem:4cont2}, if the contraction size is exactly four, we need to check whether or not deleting these four edges would result in a bridge in the resulting graph of $G$. 
    This can be computer-checked by the exact same procedure explained in the proof of Lemma \ref{lem:lowrepcomp}, and we confirmed that no bridges would occur.
    

\medskip

    Finally, there are 37 projective islands that have a contraction size of five or more. 
    To further narrow down the cases, we need the following claim. 
    
    \begin{claim}
        Let $I$ be a projective island in $\Pi_5^{13*}$, and let $W$ be the dual of $I$, which is a triangulation with an exception of one of the faces that is not a triangle.
        Let $c_{i}$ and $x_{i}$ be defined as above.
        Let $w_0, \ldots, w_4$ be the vertices in $W$ corresponding to the inner faces in $I$.
        If $x_{0} = 0$ and $x_{1} = 1$, there are vertices $v_0$ and $v_1$ in $W$ which are adjacent to at least two of $w_0, \ldots, w_4$, and moreover the degree of $v_0, v_1$ is equal to the degrees of $w_0$ and $w_1$, respectively.
    \end{claim}

    This claim follows from  Lemmas  \ref{lem:lowrepcomp} and \ref{lem:lowrephand};  
    indeed, $v_0, w_1, w_2, w_3, w_4$ and $v_0, v_1, w_2, w_3, w_4$ are both vertices that yield non-contractible 5-cycles, so if the degrees of $v_0, v_1$ differ from $w_0, w_1$, we can reroute the noncontractible curve $\gamma$ so that $k \neq 13$.

    With this claim, we can determine the degrees of some of the vertices surrounding the ring, and using Lemma \ref{lem:rep5k13-charge}, we find that for 34 projective islands (of the 37 projective islands), we can reroute the non-contractible $\gamma$ so that $k \neq 13$.
    For the remaining three cases, we can only obtain a non-contractible 5-cycle of $k = 13$, but the orientation of the surrounding edges is different.
    Even in this case, we can find another projective island $I'' \in \Pi_5^{13}$, and we confirm that each of the $I''$ is C-reducible with at most one contraction. 
\end{proof}

{\bf Remarks:} We would like that the proofs in this section would have algorithmic consequences. 
More specifically, assuming that the first case in Lemma \ref{lem:repfinal} would not occur, we would do the following. Assume that we take a noncontractible curve $\gamma$. By this assumption, we may assume either (i) $y=4, k \leq 8$, or (ii) $y=5, k \leq 13$. If (i) or (ii) with $k \leq 12$ occurs, we can find a projective island. On the other hand, if (ii) occurs and $k=13$, then we follow the proof of Lemma \ref{lem:repfinal}, but what we really need is to reroute the non-contractible curve $\gamma$ once, and this would result in the second conclusion of Lemma \ref{lem:repfinal}, which can be clearly done in linear time, 

\section{Contractions of size $\geq 5$}
\label{sect:largecont}\showlabel{sect:largecont}

\subsection{Algorithm to filter out safe configurations} \label{subsec:five}\showlabel{subsec:five}
In Section \ref{sect:smallcont}, we have discussed configurations with contraction sizes of at most four.
However, for the configurations of contraction size above four, it seems too tedious and time-consuming to account for every possible way of edges that the whole graph could be reduced to a Petersen-like graph.
Hence, we will now try to check that for each configuration, no reduction exists such that the resulting graph will be Petersen-like.
Throughout this section, we will call a C-reducible configuration $K$ \emph{safe} if no reduction of $I(K)$ in a minimal counterexample $G$ results in a Petersen-like graph.

In our proof, there are 201 reducible configurations in $\mathcal{K}$ that have a contraction size of five or more. 
For each of those configurations, we first check whether a loop in $G'$ (which is a bridge in $G$) could occur after reduction.
For this, we enumerate all possible paths $P$ in configuration $W$ in $G'$, where $E(P) \subseteq c(W)$ and the endpoints of $P$ are in the ring vertices $r_1$ and $r_2$.
In such a path $P$, we try by adding a new edge connecting $r_1$ and $r_2$ and check if Lemmas \ref{minc} or \ref{lem:6,7-cut} could be applied.
If so, such an edge cannot exist, and there will be no loops after contracting edges of $P$.

Next, we check whether $G$ would be Petersen-like after reduction. 
Let us first give sketch for this process. 
For this, we enumerate all of the possible edges and paths that could form outside $W$ in $G'$. 
We filter out the vertices that could be deleted due to some structure outside $W$ (i.e., some low cut reductions in $G$), and we obtain a graph $W^-$ in $G'$, which consists of the vertices that will never be deleted regardless of structures outside of $W$ (note that we provided similar arguments in Section \ref{sect:6,7-cut}). We will give more details below in the next paragraph. 
Finally, we calculate the size of $W^-$ and degree of each vertex in $W^-$ after low-cut reductions.
If $W^-$ has order seven or more, the resulting graph would not be the complete graph $K_6$, which is safe. 
(Recall that $K_6$ is the dual of the projective embedding of the Petersen graph).
If the resulting graph contains a vertex of degree four or of degree of six or more, we can conclude that it will never be $K_6$. 

Let us explain the details of how we construct $W^-$.
We first obtain a set of pair of vertices $P = \{\{r_1, r_2\} \mid r_1, r_2 \in R, r_1 \neq r_2\}$ (where $R$ is the set of ring vertices of $W$). 
Then, we enumerate all possible graphs that can be obtained by the following procedure:
\begin{enumerate}
    \item Pick $0 \leq m \leq 3$ pairs of vertices $(s_i, t_i) \in P$ ($1 \leq i \leq m$).
    \item Select $m$ integers $1 \leq l_i \leq 3$ so that the sum of $l_i$ is no more than three.
    \item For every $i$, add a path of length $l_i$ between the two vertices $s_i, t_i$.
\end{enumerate}
Let the graph obtained as above be $W^+$.
Let $C$ be some cycle in $W^+$ such that $|E(C) \setminus c(W)| \leq 3$, but neither Lemma \ref{minc} nor Lemma \ref{lem:6,7-cut} applies to $C$.
The length of $C$ after reduction is $|E(C) \setminus c(W)|$. 
If $5 \leq |C| \leq 7$, we flag the vertices in the smaller one of the two connected components of $G - C$.
If $8 \leq |C|$, we flag all vertices in $W - C$.

After all of the possible $W^+$ and its cycles $C$ are enumerated, we delete the flagged vertices to get $W^-$. 

With the above method of obtaining $W^-$ and checking its size (if it is at least 7 we are done) and vertices' degrees (if there is a vertex of degree at least 6 or at most 4, we are done), we can greatly reduce the number of unsafe configurations.
We implement an algorithm specifically to verify the two conditions. 
After running the program, 195 of the 201 potentially unsafe configurations were found to be safe.

The method used here is quite complex, so a naive implementation may result in a program taking too much time. 
We actually use some simple pruning techniques that are irrelevant to this proof. We show the pseudocode of the algorithm in Subsection \ref{subsec:checkconfgeq5}.

\subsection{The remaining unsafe configurations} \label{subsec:remaining}\showlabel{subsec:remaining}

The remaining six configurations require further analysis, but we are able to provide our proof for each case with further analysis by hand. One of them is done by Lemma \ref{lem:5555}.

For the other three configurations, we give details in the appendix (see Section \ref{subsec:remaining1}). 
In conclusion, we obtain the following. 

\begin{lem}
    Let $I$ be a C-reducible island of contraction size at least five in $\mathcal{K}$.
    If $I$ appears in $G$, it can be safely reduced.
\end{lem}

Finally, together with Lemma \ref{lem:cont4}, we get the same conclusion for all reducible configurations.

\begin{lem}\label{lem:safely}
    Let $I$ be an island in $\mathcal{K}$. 
    If $I$ appears in $G$, it can be safely reduced.
\end{lem}

\section{Main proof}
\label{sect:mainproof}\showlabel{sect:mainproof}

We now discuss Theorem \ref{mainth}. Let $G$ be a minimal counterexample. We know the following properties by Lemmas \ref{minc} and \ref{lem:safely}:

\begin{enumerate}
    \item  $G$ is 2-cell embedded in the projective plane $\Sigma$, 
    \item $G$ is cyclically 5-edge-connected, and if there is an edge-cut $F$ of size exactly five, one of the connected components of $G-F$ is a 5-cycle, and 
    \item 
     no island of a reducible configuration from $\mathcal{K}$ appears in $G$. 
\end{enumerate}

We first show the following lemma.
\begin{lem}\label{3-5}\showlabel{3-5}
   Representativity of $G$ in $\Sigma$ is either three, four, or five.
\end{lem}
\begin{proof}
If representativity of $G$ in $\Sigma$ is at most one, then $G$ is planar (see \cite{MT}), and hence by the Four Color Theorem \cite{4ct1, 4ct2, RSST}, $G$ is three edge-colorable (a contradiction). 

If representativity of $G$ in $\Sigma$ is at most two, then $G$ can be embedded in the plane with one cross, and by \cite{jaeger1980tait} or 2.1 of \cite{doublecross}, $G$ is three edge-colorable (a contradiction), 

Assume now that the representativity of $G$ in $\Sigma$ is at least six. If there is a separating cycle $C$ of length six or seven that satisfies Lemma \ref{lem:6,7-cut}, we take the cycle $C$ that satisfies the second condition in Lemma \ref{lem:6,7-cut}, and subject to that, the number of vertices inside the disk $D$ bounded by $C$ is as small as possible. 

By Lemmas  \ref{lem:6,7-cut} and \ref{T(v)>0}, there is a vertex $v$ of $G$ with $T(v) > 0$, and $G$ contains one of the reducible configurations $K$ as a subgraph strictly inside the disk bounded by $C$ (if $C$ exists). If it is not induced, this contradicts Corollary \ref{Kapper}. By Lemma \ref{lem:safely}, $K$ can be safely reduced. This completes the proof. 
\end{proof}

By Lemma \ref{3-5}, it remains to consider the cases that the representativity of $G$ in $\Sigma$ is either three, four or five. 
Let $C$ be the cycle of $G$ as in Lemma \ref{lem:repcycle}. 
By Lemma \ref{lem:repfinal}, either at least one island of the reducible configurations appears in $G - C$ or there is a projective island $I$ that can be safely reduced in $G$, i.e. $G \dotdiv c(I)$ is in $\mathcal{P}_1 \setminus \mathcal{P}_0$. In the first case, there is an island $I$ that can be safely reduced in $G$. 
This is a contradiction to the minimality of $G$. 


This completes the proof of Theorem \ref{mainth}.\qed

\section{Finding a 3-edge-coloring}
\label{sect:algorithm}\showlabel{sect:algorithm}

In this section, we consider the algorithmic corollary of our main theorem, Theorem \ref{mainth}, see Theorem \ref{algo}. Let us mention the problem again. 

\begin{quote}{\bf Algorithmic Problem}\\
{\bf Input}: A cubic graph $G$.\\
{\bf Output}: Either a three edge-coloring of $G$ or an obstruction that is not three edge-colorable, or else conclude that $G$ cannot be embedded in the projective plane (and output a forbidden minor for embedding a graph in the projective plane).\\
{\bf Time complexity}: $O(n^2)$, where $n=|V(G)|$.
\end{quote}

\medskip

{\bf Description}

\medskip

{\bf Step 1:} Embedding $G$ in the projective plane. 

\medskip

This can be done using the algorithm by Mohar \cite{MOHARproj}. If $G$ cannot be embedded in the projective plane, we just output the obstruction (i.e., a forbidden minor) obtained by the same algorithm \cite{MOHARproj}.

Hereafter, we may assume that the input graph $G$ is embedded in the projective plane $\Sigma$.

\medskip

{\bf Step 2:} Reduction to a cyclically 5-edge-connected graph. 

\medskip

This step is almost identical to that for the planar case, see \cite{RSST-STOC}. So we just give a brief sketch here. 

We try to find a cyclic cut of size four or less.  
We actually find such a cut (if it exists) of minimum size. If there is such a cut $F$, let the two connected components of $G-F$ be $G_1$ and $G_2$, and we obtain a 3-edge-coloring by applying this algorithm to both $G_i$ recursively.
If we fail to three-edge color one of $G_i$, then we obtain an obstruction and output it. The detailed algorithm is described in Section \ref{sect:below5cuts} in the appendix.

Hereafter we assume that $G$ is a cyclically 5-edge-connected graph. If there is a cut $F$ of size five, where both connected components are not 5-cycles, again, 
we obtain a three-edge coloring by applying this algorithm to both $G_i$ recursively (we never fail because we do not reduce to the obstructions). 
As in Lemma \ref{minc}, we can obtain the coloring of $G$ by merging the edge-colorings of $G_1$ and $G_2$.

\medskip

{\bf Step 3:} When representativity of $G$ in $\Sigma$ is at most two or at least six. 

\medskip

If the representativity of $G$ in $\Sigma$ is at most two, we can easily obtain an embedding of $H$ in the plane with one cross, using some planarity testing (see \cite{MT}), and by 2.1 of \cite{doublecross}, we can get a three edge-color $G$. 
Note that the proof of 2.1 of \cite{doublecross} (also \cite{jaeger1980tait} ) is really the same as that given in \cite{RSST}. Thus we can get such a three-edge coloring of $G$ in $O(n^2)$ time.

If the representativity of $G$ in $\Sigma$ is at least six, as in the proof of Lemma \ref{3-5}, by Lemma \ref{T(v)>0}, there is a vertex $v$ of $G$ with $T(v) > 0$ and $G$ contains one of the reducible configurations $K$ as a subgraph (it may not be induced). If this is not induced, we may obtain one of the conclusions in Lemma \ref{pairs6} (but since representativity is at least six, we obtain a contractible cycle that satisfies Fact \ref{fact:cont-cycle}). If this is induced, it allows us to make a reduction. 

We may apply Lemma \ref{lem:6,7-cut} if we find a separating cycle of length six or seven in $G'$. This case needs a bit of care. We take all vertices $V'$ of $G$ with $T(v) > 0$. 
Then for each vertex $v \in V'$, we find a reducible configuration $K_v$ with the ring $R_v$. We apply Claims \ref{clm:loop-6,7cycle} and \ref{clm:loopless} to $K_v \cup R_v$. By Lemma \ref{lem:6,7-cut}, we know that at least one of the configurations does not apply Claim \ref{clm:loop-6,7cycle}, which allows us to make a reduction.

\medskip

{\bf Step 4:} When representativity of $G$ in $\Sigma$ is either three or four or five. 

\medskip

Let $C$ be the cycle of $G$ as in Lemma \ref{lem:repcycle}. 
By Lemma \ref{lem:repfinal}, either at least one island of the reducible configurations appears in $G - C$, or there is a projective island $I$ that can be safely reduced in $G$, i.e. $G \dotdiv c(I)$ is in $\mathcal{P}_1 \setminus \mathcal{P}_0$.   

In the first case, we do the following:  by Lemma \ref{T(v)>0} there is a vertex $v$ of $G$ with $T(v) > 0$, and $G$ contains one of the reducible configurations $K$ as a subgraph. We may apply Lemma \ref{lem:6,7-cut} if we find a separating cycle of length six or seven in $G'$. Then the rest of the arguments are the same as those in Step 3.  

In the second case, we can just find such a projective island $I$ by just looking at $C$. For more details, we refer the reader to Remark right after Lemma \ref{lem:repfinal}. 

In either case, we can make a reduction. 

\medskip

Now let us estimate the time complexity of this algorithm. Step 1 takes $O(n)$ time as in \cite{MOHARproj}. All of Steps 2, 3 and 4 take an $O(n)$ time and we recuse. More precisely, for the correctness and the time complexity: 

\begin{itemize}
\item 
     We can find the shortest non-contractible curve in $O(n)$ time by \cite{CABELLO}.
    \item
     By Lemma \ref{lem:safely}, all reducible configurations in $\mathcal{K}$ are safely reduced. 
    
    \item 
    The discharging process takes $O(n)$ time because we have at most 179 rules, and each can apply just once to each edge. This allows us to find the set of all vertices $V'$ that end up with a positive charge (i.e., $T(v) > 0$) in $O(n)$ time. 
    \item 
    For the reducibility process, we only need to look at a vertex $v \in V'$, the first neighbors, the second neighbors, and the third neighbors of $v$. But we do not have to look at all the vertices of degree at least 12, as none of our reducible configurations in $\mathcal{K}$ contains a vertex of degree 12 or more. Let $Q_v$ be the graph induced by these neighbors, together with $v$, after deleting the vertices of degree at least 12. 

Each reducible configuration is of order at most 20, and by our choice, $Q_v$ has at most 10000 vertices. By Lemma \ref{T(v)>0}, we know that $Q_v$ contains a reducible configuration $K_v$. Thus we can find a reducible configuration $K_v$ in $O(1)$ time. 

Thus in total, this process takes $O(n)$ time. 

\item 
As in Step 3, 
we may apply Lemma \ref{lem:6,7-cut} if we find a separating cycle of length six or seven in $G'$. This case needs a bit of care. 

For each vertex $v \in V'$, there is a reducible configuration $K_v$ with the ring $R_v$, as above. We apply Claims \ref{clm:loop-6,7cycle} and \ref{clm:loopless}  to $K_v \cup R_v$. By Lemma \ref{lem:6,7-cut}, we know that at least one of the configurations does not apply Claim \ref{clm:loop-6,7cycle}, which allows us to make a reduction. For each $K_v \cup R_v$ this process takes $O(1)$ time, and hence in total it takes $O(n)$ time. 

\item 
We can actually enumerate all 5,6,7-cycles in $G'$ in $O(n)$ time by using \cite{FOCS05,Eppstein00}. To this end, we first refer to Theorem 3.1 in \cite{FOCS05}. The important idea is that we run a
bread-first search from some root vertex $r$ in $G'$, and assign the label to each vertex $v \in V(G')$ to be the distance between $r$ and $v$
modulo eight. The union of any seven label sets is the disjoint
union of subgraphs, each consisting of at most seven breadth-first layers. By \cite{Eppstein00} each
of these subgraphs, and therefore the union, has treewidth
at most 100. 

Thus by the standard dynamic programming on bounded treewidth graphs, we can enumerate all cycles of length 5, 6, and 7 in $O(n)$ time for each of the seven label sets. It is clear that every cycle of length at most seven must be contained in some of the seven label sets, and in addition, there are eight such bounded treewidth graphs, thus we can enumerate all cycles of length 5, 6, and 7 in $G'$ in $O(n)$ time. 
We can also prove that the number of such cycles is $O(n)$ by showing that there are at most three such cycles that can cross by Lemma \ref{minc}, see also Figures \ref{fig:6cut} and \ref{fig:7cut} (and the proofs associated with these figures). 

Among these cycles, we can take a minimal one, i.e. a cycle that satisfies Lemma \ref{lem:6,7-cut}, and subject to that, the number of vertices strictly inside the disk bounded by the cycle is as small as possible. By Lemma \ref{lem:6,7-cut}, we know that at least one of the configurations appears strictly inside the disk, which allows us to make a reduction.  

\end{itemize}

Thus the total time complexity is $O(n^2)$ as claimed. 

\bibliographystyle{plain}
\bibliography{4CT}

%
\appendix
%

\section{List of cases when a vertex sends charge 5 or 6}
\label{sect:sendscharge}\showlabel{sect:sendcharge}

\tikzset{deg5/.style={thick, circle, draw, fill=black, inner sep=1.5pt,}}
\tikzset{deg6/.style={thick, circle, draw, fill=black, inner sep=0pt,}}
\tikzset{deg7/.style={thick, circle, draw, fill=white, inner sep=2pt,}}
\tikzset{deg8/.style={thick, rectangle, draw, fill=white, inner sep=2pt,}}
\tikzset{deg9/.style={thick, regular polygon, regular polygon sides=3, rotate=180, draw, fill=white, inner sep=1pt,}}
\tikzset{deg10/.style={thick, regular polygon, draw, fill=white, inner sep=2pt,}}
\tikzset{->-/.style={decoration={
    markings,
    mark=at position .6 with {\arrow{>}}}, postaction={decorate}}}
\tikzset{->>-/.style={decoration={
    markings, 
    mark=at position .5 with {\arrow{>}};, 
    mark=at position .6 with {\arrow{>}};}, postaction={decorate}}}

\begin{figure}[htbp]
\begin{tabular}{ccccc}
~\hskip 1.8cm ~ &
\begin{minipage}[t]{0.2\hsize}
\centering
\begin{tikzpicture} []
    \node [deg6] at (0.8, 2.151) (v0) {};
    \node [deg7] at (1.8, 2.151) (v1) {};
    \node [above = 0.15 cm of v1, anchor=center] (v1+) { $+$ };
    \node [deg5] at (1.3, 3.017) (v2) {};
    \node [deg5] at (0.3, 3.017) (v3) {};
    \node [deg5] at (0.3, 1.285) (v4) {};
    \node [deg5] at (1.3, 1.285) (v5) {};
    \node [deg5] at (1.126, 4.002) (v6) {};
    \node [deg5] at (1.126, 0.3) (v7) {};
    \draw [->-] (v0) -- (v1);
    \foreach \u / \v in {v0/v1, v0/v2, v0/v3, v0/v4, v0/v5, v1/v2, v1/v5, v2/v3, v2/v6, v3/v6, v4/v5, v4/v7, v5/v7}
        \draw (\u) -- (\v);
\end{tikzpicture}
\end{minipage}
&
\begin{minipage}[t]{0.2\hsize}
\centering
\begin{tikzpicture} []
    \node [deg6] at (0.8, 2.151) (v0) {};
    \node [deg7] at (1.8, 2.151) (v1) {};
    \node [above = 0.15 cm of v1, anchor=center] (v1+) { $+$ };
    \node [deg5] at (1.3, 3.017) (v2) {};
    \node [deg5] at (0.3, 3.017) (v3) {};
    \node [deg6] at (0.3, 1.285) (v4) {};
    \node [deg5] at (1.3, 1.285) (v5) {};
    \node [deg5] at (1.126, 4.002) (v6) {};
    \node [deg5] at (1.126, 0.3) (v7) {};
    \node [deg5] at (2.24, 0.943) (v8) {};
    \draw [->-] (v0) -- (v1);
    \foreach \u / \v in {v0/v1, v0/v2, v0/v3, v0/v4, v0/v5, v1/v2, v1/v5, v1/v8, v2/v3, v2/v6, v3/v6, v4/v5, v4/v7, v5/v7, v5/v8, v7/v8}
        \draw (\u) -- (\v);
\end{tikzpicture}
\end{minipage}
&
\begin{minipage}[t]{0.2\hsize}
\centering
\begin{tikzpicture} []
    \node [deg7] at (1.201, 2.057) (v0) {};
    \node [deg7] at (2.201, 2.057) (v1) {};
    \node [above = 0.15 cm of v1, anchor=center] (v1+) { $+$ };
    \node [deg5] at (1.824, 1.275) (v2) {};
    \node [deg5] at (0.978, 1.082) (v3) {};
    \node [deg7] at (0.3, 1.623) (v4) {};
    \node [above = 0.15 cm of v4, anchor=center] (v4+) { $+$ };
    \node [deg7] at (0.3, 2.491) (v5) {};
    \node [above = 0.15 cm of v5, anchor=center] (v5+) { $+$ };
    \node [deg5] at (0.978, 3.032) (v6) {};
    \node [deg5] at (1.824, 2.839) (v7) {};
    \node [deg5] at (1.355, 0.3) (v8) {};
    \node [deg5] at (1.824, 3.706) (v9) {};
    \draw [->-] (v0) -- (v1);
    \foreach \u / \v in {v0/v1, v0/v2, v0/v3, v0/v4, v0/v5, v0/v6, v0/v7, v1/v2, v1/v7, v2/v3, v2/v8, v3/v4, v3/v8, v4/v5, v5/v6, v6/v7, v6/v9, v7/v9}
        \draw (\u) -- (\v);
\end{tikzpicture}
\end{minipage}
&
\end{tabular}
\caption{A vertex sends charge 6 in these cases.}
\label{fig:proj6}
\end{figure}
\subfile{tikz/proj5}
\tikzset{deg5/.style={thick, circle, draw, fill=black, inner sep=1.5pt,}}
\tikzset{deg6/.style={thick, circle, draw, fill=black, inner sep=0pt,}}
\tikzset{deg7/.style={thick, circle, draw, fill=white, inner sep=2pt,}}
\tikzset{deg8/.style={thick, rectangle, draw, fill=white, inner sep=2pt,}}
\tikzset{deg9/.style={thick, regular polygon, regular polygon sides=3, rotate=180, draw, fill=white, inner sep=1pt,}}
\tikzset{deg10/.style={thick, regular polygon, draw, fill=white, inner sep=2pt,}}
\tikzset{->-/.style={decoration={
    markings,
    mark=at position .6 with {\arrow{>}}}, postaction={decorate}}}
\tikzset{->>-/.style={decoration={
    markings, 
    mark=at position .5 with {\arrow{>}};, 
    mark=at position .6 with {\arrow{>}};}, postaction={decorate}}}

\begin{figure}[htbp]
\begin{tabular}{ccccc}

\begin{minipage}[t]{0.18\hsize}
\centering
\begin{tikzpicture} []
    \node [deg5] at (1.109, 0.3) (v0) {};
    \node [deg7] at (2.109, 0.3) (v1) {};
    \node [above = 0.15 cm of v1, anchor=center] (v1+) { $+$ };
    \node [deg5] at (1.418, 1.251) (v2) {};
    \node [deg5] at (0.3, 0.888) (v3) {};
    \draw [->-] (v0) -- (v1);
    \foreach \u / \v in {v0/v1, v0/v2, v0/v3, v1/v2, v2/v3}
        \draw (\u) -- (\v);
\end{tikzpicture}
\end{minipage}
&
\begin{minipage}[t]{0.18\hsize}
\centering
\begin{tikzpicture} []
    \node [deg5] at (1.109, 0.888) (v0) {};
    \node [deg7] at (2.109, 0.888) (v1) {};
    \node [above = 0.15 cm of v1, anchor=center] (v1+) { $+$ };
    \node [deg5] at (1.418, 1.839) (v2) {};
    \node [deg6] at (0.3, 1.476) (v3) {};
    \node [deg5] at (0.3, 0.3) (v4) {};
    \draw [->-] (v0) -- (v1);
    \foreach \u / \v in {v0/v1, v0/v2, v0/v3, v0/v4, v1/v2, v2/v3, v3/v4}
        \draw (\u) -- (\v);
\end{tikzpicture}
\end{minipage}
&
\begin{minipage}[t]{0.18\hsize}
\centering
\begin{tikzpicture} []
    \node [deg5] at (1.109, 0.3) (v0) {};
    \node [deg7] at (2.109, 0.3) (v1) {};
    \node [above = 0.15 cm of v1, anchor=center] (v1+) { $+$ };
    \node [deg5] at (1.418, 1.251) (v2) {};
    \node [deg6] at (0.3, 0.888) (v3) {};
    \node [deg5] at (2.436, 1.839) (v4) {};
    \node [deg6] at (0.94, 2.325) (v5) {};
    \draw [->-] (v0) -- (v1);
    \foreach \u / \v in {v0/v1, v0/v2, v0/v3, v1/v2, v1/v4, v2/v3, v2/v4, v2/v5, v3/v5, v4/v5}
        \draw (\u) -- (\v);
\end{tikzpicture}
\end{minipage}
&
\begin{minipage}[t]{0.18\hsize}
\centering
\begin{tikzpicture} []
    \node [deg5] at (1.109, 1.251) (v0) {};
    \node [deg7] at (2.109, 1.251) (v1) {};
    \node [above = 0.15 cm of v1, anchor=center] (v1+) { $+$ };
    \node [deg6] at (1.418, 2.202) (v2) {};
    \node [deg5] at (0.3, 1.839) (v3) {};
    \node [deg5] at (0.3, 0.663) (v4) {};
    \node [deg6] at (1.418, 0.3) (v5) {};
    \draw [->-] (v0) -- (v1);
    \foreach \u / \v in {v0/v1, v0/v2, v0/v3, v0/v4, v0/v5, v1/v2, v1/v5, v2/v3, v3/v4, v4/v5}
        \draw (\u) -- (\v);
\end{tikzpicture}
\end{minipage}
&
\begin{minipage}[t]{0.18\hsize}
\centering
\begin{tikzpicture} []
    \node [deg5] at (0.3, 1.251) (v0) {};
    \node [deg7] at (1.3, 1.251) (v1) {};
    \node [above = 0.15 cm of v1, anchor=center] (v1+) { $+$ };
    \node [deg5] at (0.609, 0.3) (v2) {};
    \node [deg5] at (0.609, 2.202) (v3) {};
    \draw [->-] (v0) -- (v1);
    \foreach \u / \v in {v0/v1, v0/v2, v0/v3, v1/v2, v1/v3}
        \draw (\u) -- (\v);
\end{tikzpicture}
\end{minipage}
\\[4mm]
\begin{minipage}[t]{0.18\hsize}
\centering
\begin{tikzpicture} []
    \node [deg5] at (1.109, 1.251) (v0) {};
    \node [deg7] at (2.109, 1.251) (v1) {};
    \node [above = 0.15 cm of v1, anchor=center] (v1+) { $+$ };
    \node [deg5] at (1.418, 0.3) (v2) {};
    \node [deg6] at (0.3, 1.839) (v3) {};
    \node [deg6] at (1.418, 2.202) (v4) {};
    \node [deg6] at (0.587, 3.033) (v5) {};
    \node [deg5] at (1.781, 3.32) (v6) {};
    \draw [->-] (v0) -- (v1);
    \foreach \u / \v in {v0/v1, v0/v2, v0/v3, v0/v4, v1/v2, v1/v4, v3/v4, v3/v5, v4/v5, v4/v6, v5/v6}
        \draw (\u) -- (\v);
\end{tikzpicture}
\end{minipage}
&
\begin{minipage}[t]{0.18\hsize}
\centering
\begin{tikzpicture} []
    \node [deg5] at (1.109, 1.251) (v0) {};
    \node [deg7] at (2.109, 1.251) (v1) {};
    \node [above = 0.15 cm of v1, anchor=center] (v1+) { $+$ };
    \node [deg5] at (1.418, 0.3) (v2) {};
    \node [deg5] at (0.3, 1.839) (v3) {};
    \node [deg6] at (1.418, 2.202) (v4) {};
    \draw [->-] (v0) -- (v1);
    \foreach \u / \v in {v0/v1, v0/v2, v0/v3, v0/v4, v1/v2, v1/v4, v3/v4}
        \draw (\u) -- (\v);
\end{tikzpicture}
\end{minipage}
&
\end{tabular}
\caption{A vertex of degree $5$ sends charge 4 in these cases.}
\label{fig:proj4_deg5}
\end{figure}

\clearpage
\section{Discharging rules}
\label{sect:rule}\showlabel{sect:rule}
\subfile{tikz/rule}
\clearpage


\clearpage
\section{Appendix: Pseudo codes}
\label{sect:code}\showlabel{sect:code}

\subsection{Pseudocode for checking reducibility of configurations}

We show a pseudocode of the reducibility checker introduced in Section \ref{sect:reducible configurations}.
The algorithm must be provided with two inputs: a configuration $W$ and the Kempe chain information ``Kempe''.
Before the main algorithm, we show how to generate Kempe.

In this subsection, a \textit{Kempe chain} of length $2r$ is represented by a set of $r$ disjoint pairs of integers in $\{1, \cdots, 2r\}$.
We prepare a set of Kempe chains for planar and projective planar configurations so that each Kempe chain corresponds to the signed matching and projective signed matching as defined in Definition \ref{dfn:matching} and \ref{dfn:projmatching}.

\begin{algorithm}[htbp]
    \caption{GetKempe($r$, \texttt{kempeType})}
    \label{alg:kempe}
    \begin{algorithmic}[1]
        \Require {\texttt{kempeType} $=$ \texttt{projective} or \texttt{planar}}
        \Ensure {Kempe($r$), the set of Kempe chains of length $2r$}
        \State Kempe($r$) $\gets \varnothing$
        \If {\texttt{kempeType} is \texttt{projective}}
            \ForAll {$K \in $ GetKempe($r-1$, \texttt{planar})}
                \ForAll {$a, b \in \{1, \cdots, 2(r-1)\}, a \leq b$}
                    \State flip $\gets$ a function $(x) \mapsto \begin{cases} x & (x < a) \\ a + b - x & (a \leq x < b) \\ x+2 & (b \leq x) \end{cases}$
                    \State $K' \gets$ map each $x \in P \in K$ to flip($x$)
                    \State $K' \gets K' \cup \{\{a, (b + 1)\}\}$ 
                    \State Kempe($r$) $\gets$ Kempe($r$) $\cup \{K'\}$
                \EndFor
            \EndFor
        \Else
             \State lift $\gets$ a function $(d, x) \mapsto d + x$     
             \ForAll {$K \in $ GetKempe($r-1$, \texttt{planar})}
                \State $K' \gets$ map each $x \in P \in K$ to lift($1, x$)
                \State $K' \gets K' \cup \{\{1, 2r\}\}$
                \State Kempe($r$) $\gets$ Kempe($r$) $\cup \{K'\}$
             \EndFor
             \ForAll{$i \in \{1, \cdots, r-1\}$}
                \ForAll {$K_1 \in $ GetKempe($i$, \texttt{planar})}
                    \ForAll {$K_2 \in $ GetKempe($r-i$, \texttt{planar})}
                        \State $K_2' \gets$ map each $x \in P \in K_2$ to lift($2i, x$)
                        \State $K' \gets K_1 \cup K_2'$
                        \State Kempe($r$) $\gets$ Kempe($r$) $\cup \{K'\}$
                    \EndFor
                \EndFor
             \EndFor
        \EndIf
        \State \Return Kempe($r$)
    \end{algorithmic}
\end{algorithm}

Before running the main algorithm, we store the results of GetKempe($r$, \texttt{kempeType}) for $1 \leq r \leq 9$ into an external file to reduce running time.

Now, we introduce the pseudocode of the reducibility checker as below.
The code should behave exactly the same as the algorithm used in previous works, such as the proof of 4CT and the proof of doublecross cases, as long as the correct Kempe chain information is provided.

\begin{algorithm}[htbp]
    \caption{GetMaximalColorableSet($I$, \texttt{kempeType})}
    \label{alg:get-maximal}
    \begin{algorithmic}[1]
        \Require {An island $I$, \texttt{kempeType} $=$ \texttt{projective} or \texttt{planar}}
        \Ensure {$\mathcal{C}$, a set of ring-colorings of $I$}
        \State $R \gets$ the edge set of the ring of $I$
        \State $\Phi \gets$ the set of colorings of $R$
        \State $\mathcal{C} \gets \{ \phi \in \Phi \mid $ there is a coloring of $I$ such that the restriction to $R$ is $\phi \}$
        \While {\True}
            \ForAll {$\phi \in \Phi$}
                \ForAll {$c \in \{0,1,2\}$}
                    \State $R' \gets \{r \in R \mid \phi(r) \neq c\}$
                    \State $r' \gets \frac{|R'|}{2}$
                    \State \texttt{hasBadKempe} $\gets$ \False
                    \ForAll {$K \in$ GetKempe($r'$, \texttt{kempeType})}
                        \If {for all $\phi' \in \Phi$ s.t. $\phi \sim_{K} \phi'$, $\phi' \notin \mathcal{C}$}
                            \State \texttt{hasBadKempe} $\gets$ \True
                        \EndIf
                    \EndFor
                    \If {\texttt{hasBadKempe} is \False}
                        \State $\mathcal{C} \gets \mathcal{C} \cup \{\phi\}$
                        \Break
                    \EndIf
                \EndFor
            \EndFor
        \EndWhile
        \State \Return $\mathcal{C}$
    \end{algorithmic}
\end{algorithm}

\begin{algorithm}[htbp]
    \caption{CheckReducibility($W$, \texttt{kempeType})}
    \label{alg:check-red}
    \begin{algorithmic}[1]
        \Require {A configuration $W$, \texttt{kempeType} $=$ \texttt{projective} or \texttt{planar}}
        \Ensure {(a) D-reducible, (b) C-reducible (with the contraction edge set), or (c) non-reducible}
        \State $I \gets$ the island corresponding to the configuration $W$.
        \State $\mathcal{C} \gets$ GetMaximalColorableSet($I$, \texttt{kempeType})
        \If {all ring-colorings are contained in $\mathcal{C}$}
            \State \Return D-reducible
        \EndIf
        \ForAll {$E \subset I$ where each connected component of $(I + R) \dotdiv E$ is bridgeless}
            \State $I' \gets I \dotdiv E$
            \State $\mathcal{C}' \gets \{ \phi \in \Phi \mid $ there is a coloring of $I'$ such that the restriction to $R$ is $\phi \}$
            \If {$\mathcal{C}' \subseteq \mathcal{C}$}
                \State \Return C-reducible (the contraction edge set $c(I) = E$)
            \EndIf
        \EndFor
        \State \Return non-reducible
    \end{algorithmic}
\end{algorithm}

\subsection{Pseudo code for checking dist 5 in a configuration}
\label{subsect:code-dist5}\showlabel{subsect:code-dist5}

We use several terms to explain the pseudo-code in this section. \textit{A diagonal vertex} of edge $uv$ in near triangulation $S$ is a vertex that constitutes a triangle face with $uv$ in $S$. Since $S$ is a near triangulation, an edge that incidents with a non-triangle face has one diagonal vertex, whereas other edges have two diagonal vertices.
A configuration $K$ of $\mathcal{K}$ is a near triangulation with no 3-vertex-cut, so diagonal vertices of an edge are obtained by finding a vertex that constitutes a triangle with the edge without considering an embedding of it.

\begin{algorithm}
\caption{checkDist5($K$)}
\begin{algorithmic}[1]
    \Require{A configuration $K$, its free completion $S$ with its ring $R$}
    \Ensure{A boolean value that an edge $uv$ can exist}
    \State \texttt{uv\_possible} $\gets \False$
    \ForAll {$u, v \gets V(K)$ such that $d_K(u,v) = 5$}
        \ForAll {\texttt{neighbor\_u} $\gets$ a vertex of $R$ that is adjacent to $u$}
            \ForAll {\texttt{neighbr\_v} $\gets$ a vertex of $R$ that is adjacent to $v$}
                \If {checkDist5Contractible($K$, $u$, $v$, \texttt{neighbor\_u}, \texttt{neighbor\_v})}
                    \State \texttt{uv\_possible} $\gets \True$
                    \Comment{dangerous pair exists}
                \EndIf
                \If {checkDist5NonContractible($K$, $u$, $v$, \texttt{neighbor\_u}, \texttt{neighbor\_v})}
                    \State \texttt{uv\_possible} $\gets \True$
                    \Comment{dangerous pair exists}
                \EndIf
            \EndFor
        \EndFor
    \EndFor
    \State \Return \texttt{uv\_possible} 
\end{algorithmic}
\end{algorithm}

\begin{algorithm}
\caption{getCorrespondingVertices($K$, $u$, $v$, \texttt{neighbor\_u}, \texttt{neighbor\_v}, $x_u$, $y_u$, $x_v$, $y_v$)}
\label{alg:corresponding_vertices}
\begin{algorithmic}[1]
    \Require{A configuration $K$, its free completion $S$ with its ring $R$, $u,v \in V(K)$, \texttt{neighbor\_u}, \texttt{neighbor\_v} $\in V(R)$, $x_u, y_u, x_v, y_v \in V(S)$}
    \Ensure {pairs of corresponding vertices in $S$ around $x(=x_u,y_u), y(=y_u,y_v)$}
    \State \texttt{corresponding\_vertices} $\gets \{(u, $\texttt{neighbor\_v}$), ($\texttt{neighbor\_u}, $v), (x_u, x_v), (y_u, y_v)\}$
    \State \texttt{corresponding\_edges} $\gets \{(ux_u,$ \texttt{neighbor\_v}$x_v), (uy_u, $\texttt{neighbor\_v}$y_v), ($\texttt{neighbor\_u}$x_u, vx_v), ($\texttt{neighbor\_u}$y_u, vy_v)\}$
    \ForAll {(\texttt{edge\_u}, \texttt{edge\_v}) $\gets$ \texttt{correspdonding\_edges}}
        \While {not (all diagonal vertices of \texttt{edge\_u} belong to \texttt{corresponding\_vertices}, or all diagonal vertices of \texttt{edge\_v} belong to \texttt{corresponding\_vertices})}
            \State \texttt{vertex\_u} $\gets$ the diagonal vertex of \texttt{edge\_u} that does not belong to any pair of \texttt{corresponding\_vertices}
            \State \texttt{vertex\_v} $\gets$ the diagonal vertex of \texttt{edge\_v} that does not belong to any pair of \texttt{corresponding\_vertices}
            \State \texttt{corresponding\_vertices} $\gets$ \texttt{corresponding\_vertices} $\cup$ $\{$(\texttt{vertex\_u},  \texttt{vertex\_v})$\}$
            \State \texttt{edge\_u} $\gets$ an edge whose ends are \texttt{vertex\_u} and $x_u$ or $y_u$ (we choose $x_u$ or $y_u$, which is an end of \texttt{edge\_u})
            \State \texttt{edge\_v} $\gets$ an edge whose ends are \texttt{vertex\_v} and $x_v$ or $y_v$ (we choose $x_v$ or $y_v$, which is an end of \texttt{edge\_v})
        \EndWhile
    \EndFor
    \State \Return \texttt{corresponding\_vertices}
\end{algorithmic}
\end{algorithm}

\begin{algorithm}
\caption{checkDist5NonContractible($K$, $u$, $v$, \texttt{neighbor\_u}, \texttt{neighbor\_v})}
\label{alg:check_noncontractible}
\begin{algorithmic}[1]
    \Require{A configuration $K$, its free completion $S$ with its ring $R$, $u,v$ in $V(K)$, \texttt{neighbor\_u}, \texttt{neighbor\_v} in $V(R)$}
    \Ensure {A boolean value that represents whether $u$ can be adjacent to $v$ non-contractibly.}
    \State $x_u, y_u \gets$ two diagonal vertices of edge $u$\texttt{neighbor\_u} in $S$
    \State $x_v, y_v \gets$ two diagonal vertices of edge $v$\texttt{neighbor\_v} in $S$ such that $x_v, y_v$ corresponds to $x_u, y_u$ respectively. \Comment{choose $x_v, y_v$ such that $u$ is adjacent to $v$ non-contractibly.}
    \State \texttt{corresponding\_vertices} $\gets$ getCorrespondingVertices($K$, $u$, $v$, \texttt{neighbor\_u}, \texttt{neighbor\_v}, $x_u$, $y_u$, $x_v$, $y_v$)
    \ForAll{$(w_u, w_v) \gets$ \texttt{corresponding\_vertices}}
        \If {($w_u \in K$ and $w_v \in K$) or $d_S(w_u, w_v) \leq 5$} \Comment{representativity is at most 5}
            \State \Return \False
        \EndIf
    \EndFor
    \State \Return \True
\end{algorithmic}
\end{algorithm}

\begin{algorithm}
\caption{checkDist5Contractible($K$, $u$, $v$, \texttt{neighbor\_u}, \texttt{neighbor\_v})}
\label{alg:check_contractible}
\begin{algorithmic}[1]
    \Require{A configuration $K$, its free completion $S$ with its ring $R$, $u,v$ in $V(K)$, \texttt{neighbor\_u}, \texttt{neighbor\_v} in $V(R)$}
    \Ensure {A boolean value that represents whether $u$ can be adjacent to $v$ contractibly.}
    \State $x_u, y_u \gets$ two diagonal vertices of edge $u$\texttt{neighbor\_u} in $S$
    \State $x_v, y_v \gets$ two diagonal vertices of edge $v$\texttt{neighbor\_v} in $S$ such that $x_v, y_v$ corresponds to $x_u, y_u$ respectively. \Comment{choose $x_v, y_v$ such that $u$ is adjacent to $v$ contractibly.} 
    \State \texttt{corresponding\_vertices} $\gets$ getCorrespondingVertices($K$, $u$, $v$, \texttt{neighbor\_u}, \texttt{neighbor\_v}, $x_u$, $y_u$, $x_v$, $y_v$)
    \ForAll {$(w_u, w_v) \gets$ \texttt{corresponding\_vertices}}
        \If {$w_u \in K$ and $w_v \in K$}
            \State \Return \False
        \EndIf
    \EndFor
    \ForAll {$P \gets$ all paths of dist $5$, connecting $u$ and $v$ in $K$}
        \State $C_1, C_2 \gets$ two connected components of $S - $($P$ + \texttt{neighbor\_u} + \texttt{neighbor\_v})
        \If {$\min \{|C_1|, |C_2|\} > 3$}
            \State \Return \False
            \Comment{A contradiction to the condition of 6-cut}
        \EndIf
    \EndFor
    \ForAll {(\texttt{vertex\_u}, \texttt{vertex\_v}) $\gets \{(x_u, x_v), (y_u, y_v)\}$}
        \If {\texttt{vertex\_u} $\in R$ and \texttt{vertex\_v} $\in R$}
            \ForAll {\texttt{neigh\_vertex\_u} $\gets$ all neighbors of \texttt{vertex\_u} in $K$}
                 \ForAll {\texttt{neigh\_vertex\_v} $\gets$ all neighbors of \texttt{vertex\_v} in $K$}
                     \ForAll {$P \gets$ all shortest paths between \texttt{neigh\_vertex\_u} and \texttt{neigh\_vertex\_v} in $K$}
                         \State $C_1, C_2 \gets$ two connected components of $S - $($P + $ \texttt{vertex\_u} + \texttt{vertex\_v})
                         \State $l \gets$ (the number of edges in $P$) + 2
                         \State $x \gets \min(|C_1|, |C_2|)$
                         \If {$l \leq 4$, $x > 0$ or $l = 5$, $x > 1$ or $l = 6$, $x > 3$}
                             \State \Return \False 
                             \Comment{A contradiction to the condition of 2,3,4,5,6-cut}
                         \EndIf
                     \EndFor
                 \EndFor
            \EndFor
        \ElsIf {\texttt{vertex\_u} $\in R$ and \texttt{vertex\_v} $\not \in R$}
            \ForAll {\texttt{neigh\_vertex\_u} $\gets$ all neighbors of \texttt{vertex\_u} in $K$}
                \State \texttt{neigh\_vertex\_v} $\gets$ the vertex corresponding \texttt{neigh\_vertex\_u} in \texttt{corresponding\_vertices}
                \ForAll {$P \gets$ all shortest paths between \texttt{neigh\_vertex\_u} and \texttt{vertex\_v} in $K$}
                    \State $C_1, C_2 \gets$ two connected components of $S - $($P + $ \texttt{vertex\_u} + \texttt{neigh\_vertex\_v})
                         \State $l \gets$ (the number of edges in $P$) + 1
                         \State $x \gets \min(|C_1|, |C_2|)$
                         \If {$l \leq 4$, $x > 0$ or $l = 5$, $x > 1$ or $l = 6$, $x > 3$}
                             \State \Return \False
                             \Comment{A contradiction to the condition of 2,3,4,5,6-cut}
                         \EndIf
                     \EndFor
            \EndFor
        \ElsIf {\texttt{vertex\_u} $\not \in R$ and \texttt{vertex\_v} $\in R$}
            \State ...\Comment{The same procedure as \texttt{vertex\_u} 
 $\in R$ and \texttt{vertex\_v} $\not \in R$ except swapping the role of \texttt{vertex\_u} and \texttt{vertex\_v}, so we omit it}
        \EndIf
    \EndFor
    \State \Return \True
\end{algorithmic}
\end{algorithm}

The symbol $K$ denotes a configuration, and $S$ denotes free completion of $K$ with its ring $R$.
When we check $K$, we enumerate all pairs of vertices of $V(K)$ such that the distance between them is $5$. The symbol $u,v$ denotes such vertices. We fix a neighbor of $u$ as \texttt{neighbor\_u} and a neighbor of $v$ as \texttt{neighbor\_v}. We assume a vertex $u$ corresponds to \texttt{neigbor\_v}, and $v$ corresponds to \texttt{neighbor\_u}. The assumption corresponds to the case an edge $uv$ exists. There are two cases where an edge $uv$ exists. One case is contractible, the other is a non-contractible case. In a minimal counterexample, which is a triangulation, there are two diagonal vertices of an edge $uv$, denoted as $x, y$. There is a vertex around $u, v$ that corresponds to $x, y$ respectively, each of which is denoted as $x_u, x_v, y_u, y_v$. The correspondence of those vertices is different in contractible cases and non-contractible cases. The symbol $P$ denotes a path in $S$ that consists of a path of length 5 between $u$ and $v$ in $K$, and \texttt{neighbor\_u}, \texttt{neighbor\_v}. The path $P$ divides $S$ into two components. In the contractible case, $x_u$, $x_v$ belong to the same component, while $x_u$, $x_v$ belong to the different component in the non-contractible case (The same holds true for $y_u$, $y_v$). This corresponds to the fact that a non-contractible cycle in the projective plane is \emph{one-sided}, which means left and right interchange when traversing.
In both cases, we calculate other corresponding vertices in the vicinity of $x$ and $y$ by the function getCorrespondingVertices in Algorithm \ref{alg:corresponding_vertices}. In the non-contractible case, we check the representativity is at most 5 by calculating the shortest path distance between corresponding vertices in Algorithm \ref{alg:check_noncontractible}. In the contractible case, we check the condition of a low cut in Algorithm \ref{alg:check_contractible}. To do so, we calculate the shortest path distance between corresponding vertices and the size of components divided by the path. Note that the size of components $C \subseteq V(S)$ is calculated loosely since two vertices in $R$ may be the same and a vertex in $K$ and a vertex in $R$ may be the same when the distance between them is $5$. To be specific, the size is calculated by $|V(C \cap K)| + \max \{ |$neighbors of $v$ in $R \cap C$ $| \mid v \in V(K), $there is no vertex of $C$ of distance 5 between $v$ $\}$.
In the proof of Claim \ref{clm:all-imply} in Section \ref{subsect:imply1}, we need to check $K \in \mathcal{K}$ with a stronger constraint. More precisely, the size is calculated by $|V(C \cap K)| + \max \{ f(|$neighbors of $v$ in $R \cap C$ $|) \mid v \in V(K), $there is no vertex of $C$ of distance 5 between $v$ $\}$, where $f \colon \{0,1,2,3\} \rightarrow \{0, 1\}$ is defined as $f(0)=f(1)=1, f(2)=f(3)=1$.

\subsection{Pseudo code for discharging}
\label{subsect:code-disch}\showlabel{subsect:code-disch}
\begin{dfn}
    A \emph{range-configuration} is a 3-tuple $K = (G, \alpha_K, \beta_K)$, where $G$ is a near-triangulation on the plane and $\alpha_K: V(G) \rightarrow \mathbb{Z}_+$, $\beta_K : V(G) \rightarrow \mathbb{Z}_+ \cup \infty$ are maps satisfying the following conditions:
    \begin{enumerate}
        \renewcommand{\labelenumi}{(\roman{enumi})}
        \item For every vertex $v \in V(G)$, $G - v$ has at most two components, and if there are two, then $\alpha_K(v) = \beta_K(v) = \text{deg}_{G}(v) + 2$.
        \item For every vertex $v \in V(G)$, if $v$ is not incident with the infinite face, then $\alpha_K(v) = \beta_K(v) = \text{deg}_{G}(v)$, and otherwise $\beta_K(v) \geq \alpha_K(v) > \text{deg}_{G}(v)$. In either case, $\alpha_K(v) \geq 5$.
    \end{enumerate}
\end{dfn}
A definition of range-configuration is similar to that of a configuration. The only difference is a vertex of a range-configuration that is incident with the infinite face does not have a fixed degree but a range of possible degrees.
We define range-configuration here since it is useful in describing a discharging algorithm.

\begin{dfn}
    Let $T = (G, \alpha_T, \beta_T), T' = (G', \alpha_{T'}, \beta_{T'})$ be two range-configurations and $abc, a'b'c'$ be a triangle in $G, G'$ respectively.
    If there is a range-configuration $T'' = (G'', \alpha_{T''}, \beta_{T''})$ such that there exists an injective map $f \colon V(G) \rightarrow V(G''), f' \colon V(G') \rightarrow V(G'')$ with following properties, We say $T''$ is obtained by \textit{overlaping $T$ and $T'$ so that a triangle $abc$ corresponds to $a'b'c'$}.
    \begin{enumerate}
        \renewcommand{\labelenumi}{(\roman{enumi})}
        \item $f(a) = f'(a')$, $f(b) = f'(b')$, and $f(c) = f'(c')$.
        \item For each vertex $v'' \in V(G'')$, there exists $v \in V(G)$ such that $v'' = f(v)$, or there exists $v' \in V(G')$ such that $v'' = f'(v')$.
        \item $uv \in E(G)$ is equivalent to $f(u)f(v) \in E(G'')$. The same condition holds for $f'$ and $T'$.
        \item For each vertex $v'' \in V(G'')$, $\alpha_{T''}(v'')$ has the smallest value under the constraint that
        \begin{itemize}
            \item If $f^{-1}(v'') \neq \emptyset$ $\alpha_{T''}(v'') \geq \alpha_{T}(f^{-1}(v''))$
            \item If $f'^{-1}(v'') \neq \emptyset$, $\alpha_{T''}(v'') \geq \alpha_{T'}(f'^{-1}(v''))$
        \end{itemize}
        \item For each vertex $v'' \in V(G'')$, $\beta_{T''}(v'')$ has the largest value under the constraint that 
        \begin{itemize}
            \item If $f^{-1}(v'') \neq \emptyset$ $\beta_{T''}(v'') \leq \beta_T(f^{-1}(v''))$
            \item If $f'^{-1}(v'') \neq \emptyset$ $\beta_{T''}(v'') \leq \beta_{T'}(f'^{-1}(v''))$
        \end{itemize}
    \end{enumerate}
\end{dfn}

Let $T = (G, \alpha_T, \beta_T), T' = (G', \alpha_{T'}, \beta_{T'})$ be two range-configurations and $uv, u'v'$ be an edge in $G, G'$ respectively.
The number of vertices that are adjacent to $u, v$ in $G$ is 0, 1, or 2 since we only use a 4-connected near-triangulation as $G$. When the number is less than 2, we add several dummy vertices so that each added vertex is adjacent to both $u, v$ and set value $\alpha_T$ as 5, $\beta_T$ as $\infty$, so we can assume there are two vertices that are adjacent to $u, v$. We denote these two vertices as $a_1, a_2$. We do the same thing in $G'$ and denote these two vertices as $a'_1, a'_2$.
When we try to overlap two range-configurations so that $u$ corresponds $u'$ and $v$ corresponds $v'$ (i.e. $f(u) = f'(u'), f(v) = f'(v')$), we specify two pairs of triangles $(uva_1, u'v'a'_1)$ and $(uva_1, u'v'a'_2)$. We call the operation \textit{overlapping $T$, $T'$ so that an edge $uv$ corresponds to $u'v'$}. By this operation, we get 0, 1, or 2 confs.

We regard a rule as a range-configuration by taking 3-tuple $(G(R), \alpha_R, \beta_R)$ from a rule $R$.
We overlap rules of $\mathcal{R}$ when we enumerate a near-triangulation that a vertex sends charge. We overlap two rules $R_1,R_2$ of $\mathcal{R}$ so that an edge $s(R_1)t(R_1)$ corresponds to $s(R_2)t(R_2)$. We enumerate all possible cases that a vertex sends charge by overlapping all combinations of rules of $\mathcal{R}$. We write the pseudo code in Algorithm \ref{alg:enum_send}. The algorithm stops since there is no update in confs after trying to overlap all combinations of rules.

We use a set of rules $\mathcal{R}$ but we do not simply use $\mathcal{R}$. We decompose the degree of each vertex of $R \in \mathcal{R}$ and get another set of rules. The decomposition procedure is the following. For each vertex $v \in V(G(R))$, the map $\alpha_R(v), \beta_R(v)$ represents lowerbound and upperbound of degree of $v$. When $\beta_R(v)$ is not $\infty$, we decompose the range into the value $\alpha_R(v), \alpha_R(v)+1, ..., \beta_R(v)-1,\beta_R(v)$. Otherwise, $\beta_R(v) = \infty$, $\alpha_R(v)$ is at most $7$ for all rule $R \in \mathcal{R}$, so we decompose the range into $\alpha_R(v), \alpha_R(v)+1,...,8$ and $[9, \infty]$.
We do not distinguish the vertex of degree $9,10,11,...$ except the hub of the wheel since all rules of $\mathcal{R}$ do not distinguish degree $9,10,11,...$ (i.e. each vertex $v$ of $R \in \mathcal{R}$ satsify $\alpha_R(v) \leq 8$ and $\beta_R(v)= \infty$ if $9 \leq \beta_R(v)$. ). We repeat the process until all vertices have a ranges of degree $5, 6, 7, 8$, or $[9, \infty]$. By the decomposition procedure, we get another set of rules $\mathcal{R}'$ that is equivalent to $\mathcal{R}$. A set $\mathcal{R}'$ is useful in implementing an algorithm that decides the degree of a cartwheel and finds a reducible configuration that is contained in a cartwheel.

\begin{dfn}
    Let $K$ be a configuration and $L = (G, \alpha_L, \beta_L)$ be a conf.
    A range-configuration $L$ \textit{contains} $K$ when there exists a map $f \colon V(G(K)) \rightarrow V(G)$ such that $uv \in E(G(K))$ is equivalent to $f(u)f(v) \in E(G)$ for all $u,v \in V(G(K))$ and $\alpha_L(f(v)) = \beta_L(f(v)) = \gamma_K(v)$ for each vertex $v \in V(G(K))$.
\end{dfn}

\begin{algorithm}[htbp]
\caption{enumSendCase($\mathcal{R}'$, $\mathcal{C}$)}
\label{alg:enum_send}
\begin{algorithmic}[1]
\Require {A set of rules $\mathcal{R}'$, a set of reducible configurations $\mathcal{C}$}
\Ensure {A set of 4-tuple $(G, s, t, n)$, where $G$ contains no configuration of $\mathcal{C}$, the amount of charge a vertex $s \in V(G)$ sends to $t \in V(G)$ by rules of $\mathcal{R}'$ is $n$.
}
\State $T \gets$ a range-configuration $(G_{st}, \alpha_{T}, \beta_{T})$, where $G_{st}$ is a near-triangulation that consists of only $st$ edge and $\alpha_T(s) = 5, \beta_T(s) = 8, \alpha_T(t) = 7, \beta_T(t) = \infty$.
\State $\mathcal{G} \gets \{T\}$
\While {\True}
    \State $\mathcal{G}' \gets \mathcal{G}$
    \ForAll{$G \gets \mathcal{G}$}
        \State \begin{varwidth}[t]{\linewidth}
            $\mathcal{G}' \gets \mathcal{G}'\ \cup\ \{\text{ range-configurations obtained by overlapping } G \text{ and } (G(R), \alpha_R, \beta_R))$ \par
            \hskip \algorithmicindent so that $st \text{ corresponds to } s(R)t(R) \mid R \in \mathcal{R}\}$
        \end{varwidth}
    \EndFor
    \State $\mathcal{G}'' \gets \{G \in \mathcal{G}' \mid G \text{ does not contain any configuration of }\mathcal{C}\}$
    \If {$\mathcal{G} = \mathcal{G}''$}
        \Break
    \EndIf
    \State $\mathcal{G} \gets \mathcal{G}''$
\EndWhile
\State $\mathcal{G} \gets$ \{$(G,s,t,n) \mid G \in \mathcal{G}$, an amount of charge sent along an edge $st$ by rules of $\mathcal{R}'$ is $n$\}
\State \Return $\mathcal{G}$
\end{algorithmic}
\end{algorithm}

\begin{dfn}
    A \textit{wheel} is a configuration $W$ so that $G(W)$ consists of a cycle $C$ and a vertex $w$ that is adjacent to all vertices of $C$.
    A vertex $w$ is called \textit{hub} of the wheel.
\end{dfn}

\begin{dfn}
    A \textit{cartwheel} is a configuration $W$ such that there is a vertex $w$ and three cycles $C_1, C_2, C_3$ of $G(W)$ with the following properties.
    \begin{enumerate}
        \renewcommand{\labelenumi}{(\roman{enumi})}
        \item $\{w\}, V(C_1), V(C_2), V(C_3)$ are pairwise disjoint and have union $V(G(W))$.
        \item $C_1, C_2, C_3$ are all induced subgraphs of $G(W)$
        \item The distance between $w$ and each vertex of $C_i$ is $i$ $(i=1,2,3)$.
    \end{enumerate}
    We call $w$ the \textit{hub} of cartwheel.
\end{dfn}

When we calculate the amount of charge accumulated in a vertex $v$, we only consider the cartwheel whose hub is $v$ since for each rule $R$ of $\mathcal{R}$, the distance from $s(R), t(R)$ to each vertex of $R$ is at most 3.  

For each rule $R$ of $\mathcal{R}$, degree of $s(R)$ is $5, 6, 7,$ or $8$, so we set $\alpha_T(s) = 5, \beta_T(s) = 8$ in Algorithm \ref{alg:enum_send}. The degree of hub is at least $7$, so we set $\alpha_T(t) = 7, \beta_T(t) = \infty$ in Algorithm \ref{alg:enum_send}.

\begin{algorithm}[htbp]
\caption{discharge($d$, $\mathcal{R}'$, $\mathcal{C}$)}
\label{alg:discharge}
\begin{algorithmic}[1]
\Require {A degree $d$, A set of rule $\mathcal{R}'$, a set of reducible configuration $\mathcal{C}$}
\Ensure {A set of cartwheels where hub's degree is $d$ and hub' charge is positive after applying rules of $\mathcal{R}'$ and contains no configuration of $\mathcal{C}$.}
\State $\{(G, s, t, n)\} \gets \text{enumSendCase}(\mathcal{R}', \mathcal{C})$
\State $\mathcal{W} \gets$ \{$(W, \alpha, \beta) \mid W$ is a wheel, hub is $w$ such that $d(w) = \alpha(w) = \beta(w) = d$, and for other vertex $v \in V(W)$, $\alpha(v) = 5, \beta(v) = \infty$. \}
\ForAll {$v \gets \text{the neighbor of hub }w$}
    \State $\mathcal{W}' \gets \emptyset$
    \ForAll{$W \in \mathcal{W}$}
        \State $\mathcal{W}' \gets \mathcal{W}' \cup \{\text{ range-configurations obtained by overlapping }W\text{ and }G\text{ so that }vw\text{ corresponds to }st \mid G \in \{(G, s, t, n)\}\}$
    \EndFor
    \State $\mathcal{W} \gets \{W \in \mathcal{W}' \mid W\text{ does not contain any configuration of }\mathcal{C}\}$
\EndFor
\State $\mathcal{W} \gets$ \{$W \in \mathcal{W} \mid$ an amount of charge accumulated in hub of $W$ after applying rules of $\mathcal{R}'$ is positive.\}
\State \Return $\mathcal{W}$
\end{algorithmic}
\end{algorithm}
We execute a discharge algorithm for $d=7,8,9,10,11$ and confirm returned set $\mathcal{W} = \emptyset$.
The algorithm described here is a simplified version of a discharging algorithm implemented.
We explain some detailed implementation here.
We use the common result of Algorithm\ref{alg:enum_send} when we execute a discharging algorithm for $d=7,8,9,10,11$. We calculate the result of Algorithm\ref{alg:enum_send} in advance to reduce time.
Also, we execute a discharging program in parallel. First, we enumerate wheels by searching all combinations of degrees of neighbors of hub. The range of hub degree is from $7$ to $11$ and the range of neighbors of the hub is $5, 6, 7, 8, $or $[9, \infty]$. After enumerating wheels, we execute a discharging program for each wheel in parallel. 

We explain the pruning method of searching for combinations of degrees of cartwheel in the discharging algorithm. We adopt three methods. First method is pruning a cartwheel that contains a reducible configuration for each loop as described in Algorithm \ref{alg:discharge}. Second method is pruning the cartwheel when an edge $vw$ has already overlapped some range-configuration $G$ with $(G, s, t, n)$ and the amount of charge that is sent along an edge $vw$ is larger than $n$ for some neighbor $v$. By overlapping a range-configuration $G$ with $(G,s,t,n)$, we search the case the amount of charge is exactly $n$, so if it is larger than $n$, we do not need to search it. We search it by overlapping another conf. Third method is to prune the cartwheel so that the amount of charge accumulated in the hub is already known as at most 0. We calculate the largest possible amount of charge that can be sent from $v$ for each neighbor. Also, we calculate the amount of charge that is sent from $w$ to $v$ for each neighbor only by decided degree. The largest possible value of the charge is calculated by these information and initial charge of the hub. If the value is at most 0, we prune it.

\subsection{Pseudo code for checking cycle of length 6, or 7}
\label{subsect:code-cut6}\showlabel{subsect:code-cut6}

In Section \ref{sect:ccheck}, we use two subroutines \emph{forbiddenCycle}, \emph{forbiddenCycleOneEdge}. We show the pseudo-code that corresponds to these subroutines in Algorithm \ref{alg:forbiddenCycle-6,7cut}, \ref{alg:forbiddenCycleOneEdge-6,7cut}. Note that we have to implement other codes that enumerate all tuples of vertices that have a specified distance after contracting edges in its free completion of the configuration as in Table \ref{table:comp67cut}. This is easy to implement since the task is calculating the shortest path distance, so we omit it here.
We give a few remarks on the Algorithm \ref{alg:forbiddenCycle-6,7cut}, \ref{alg:forbiddenCycleOneEdge-6,7cut}. The symbol $K$ denotes a configuration, and $S$ denotes free completion of $K$ with its ring.
We calculate the number of vertices in the Algorithm as we discussed in Section \ref{sect:ccheck}. For example, in line 12 in Algorithm \ref{alg:forbiddenCycle-6,7cut}, $x = \lceil (s - (k-1))/2 \rceil + t$, where $s,t$ is the number of vertices of the component of $S - R$ such that $P + R$ bounds in the ring of $K$, and in $K$ respectively. We use the same calculation method in other parts of the Algorithms.
Another remark is about the choice of $Q$.
The choice of $Q$ depends on the assumption that how $P$ is embedded outside $K$.
For example, (6cut-1) in Table \ref{table:comp67cut}, we enumerate four vertices $a, b, c, d$ in the ring such that $a, b, c, d$ are in the clockwise order listed in the ring of $K$, and $d(a, b) = d(c, d) = 0$ after contracting edges of $c(K)$.
One of the subroutines we call is forbiddenCycle($K$, $a$, $b$, 2, 6). There are two candidates for $Q$. In this case, we choose a candidate that does not contain $c, d$ as $Q$. 
Another subroutine we call is forbiddenCycle($K$, $b$, $a$, 4, 6). In this case, we choose the path that contains $c, d$ as $Q$. These choices come from the assumption that there are paths of length 2, 1, 2, 1 between $a$ and $b$, $b$ and $c$, $c$ and $d$, $d$ and $a$ respectively such that these paths constitute a cycle of length 6 that surrounds $K$. In our practical implementation, we distinguish two choices of $Q$ by switching $u$ and $v$.
The last remark is about the condition of 7cut. The condition of 7cut is used only when \texttt{cutsize} $= 7$ since we cannot always use the minimality condition of 7cut in checking \texttt{cutsize} $= 6$.
More specified information is in the first part of Section \ref{subsect:7-cut}. Please refer to it if the reader is interested.

\begin{algorithm}
\caption{forbiddenCycle($K$, $u$, $v$, $k$, \texttt{cutsize})}
\label{alg:forbiddenCycle-6,7cut}
\begin{algorithmic}[1]
    \Require{A configuration $K$, its free completion $S$ with its ring, two vertices of the ring $u, v$, positive integers $k$, \texttt{cutsize}}
    \Ensure{A boolean value that represents whether there can be a path $P$ of length $k$ between $u$ and $v$ such that $P$ is part of a cycle of length \texttt{cutsize}}
    \State $Q \gets$ the path in the ring that connects $u$ and $v$ 
    \Comment {There are two candidates of $Q$. We choose one of them by assuming $P+Q$ bound the disk in which no vertex of $S$ exists. }
    \If {$|Q| = k$}
        \State \Return \False
    \ElsIf {$|Q| < k$}
        \State \Return \True
    \Else
        \ForAll{$R \gets$ all shortest paths in $S$ between $u$ and $v$}
            \If {all vertices of $R$ belongs the ring}
                \Continue
            \EndIf
            \State $l \gets |R| + k$ 
            \State $x \gets$ the number of vertices of the component of $S - R$ such that $P+R$ bounds
            \If {$l \leq 4, x > 0$ or $l = 5, x > 1$ or $l = 6, x > 3$ or \texttt{cutsize} = 7, $l = 7, x > 4$}
                \State \Return \True
            \EndIf
        \EndFor
        \State \Return \False
    \EndIf
\end{algorithmic}
\end{algorithm}

\begin{algorithm}
\caption{forbiddenCycleOneEdge($K$, $u$, $v$, $k$, \texttt{cutsize})}
\label{alg:forbiddenCycleOneEdge-6,7cut}
\begin{algorithmic}
    \Require{A configuration $K$, its free completion $S$ with its ring, two vertices of the ring $u, v$, positive integers $k$, \texttt{cutsize} }
    \Ensure{A boolean value that represents whether there can be a path $P$ of length $k$ and a vertex $w$ such that $w$ is adjacent to $v$ and endpoints of $P$ are $u$, $w$ and $P$ is part of a cycle of length \texttt{cutsize}}
    \State $Q \gets$ the path in the ring that connects $u$ and $v$
    \Comment {There are two candidates of $Q$. We choose one of them by assuming $P+Q+vw$ bound the disk in which no vertex of $S$ exists. }
    \State $l \gets$ \texttt{cutsize}$- k + |Q| + 1$
    \State $x \gets$ the number of vertices of $S - Q$
    \If {$l \leq 4, x > 0$ or $l = 5, x > 1$ or $l = 6, x > 3$ or \texttt{cutsize} = 7, $l = 7, x > 4$}
        \State \Return \True
    \EndIf
    \ForAll{$R \gets$ all shortest paths in $S$ between $u$ and $v$}
        \State $l \gets |R| + k + 1$ 
        \State $x \gets$ the number of vertices of the component of $S - R$ such that $P+R+vw$ bounds
        \If {$l \leq 4, x > 0$ or $l = 5, x > 1$ or $l = 6, x > 3$ or \texttt{cutsize} = 7, $l = 7, x > 4$}
            \State \Return \True
        \EndIf
    \EndFor
    \State \Return \False
\end{algorithmic}
\end{algorithm}

\subsection{Pseudocode for checking configurations with contraction size $\geq 5$}\label{subsec:checkconfgeq5}\showlabel{subsec:checkconfgeq5}

In section \ref{sect:largecont}, we introduce the method of proving the validity of a configuration $W$ of large contraction size by enumerating all possible outskirts of $W$ and of obtaining a smaller graph $W^-$ which consists of vertices that can never be deleted by a low cut reduction. 
We show the pseudocode of the algorithm that we implemented.

In the algorithm, we use the result of Lemma \ref{lem:6,7-cut} to determine which one of the connected components separated by some vertex separation $C$ could be reduced.
If $8 \leq |C|$, both of the connected components are subject to reduction no matter what.
If $5 \leq |C| \leq 7$, if the size of the connected component is larger than a certain threshold, the reduction would not occur.
The specific criteria are shown in the pseudocode below.

\begin{algorithm}[H]
    \caption{GetLowCutReducable($C_1$, $C_2$, $l$)}
    \label{alg:lowcutreducable}
    \begin{algorithmic}[1]
        \Require {Two connected components $C_1, C_2$, an integer $5 \leq l \leq 8$}
        \Ensure {Returns the possible low-cut reducible components when $C_1$ and $C_2$ are surrounded by a separation of size $l$}
        \If {$l \leq 5$}
            \State \Return $\varnothing$
        \EndIf
        \If {$l \geq 8$}
            \State \Return $C_1 + C_2$
        \EndIf
        \If {$l = 6, \min \{|C_1|, |C_2|\} > 3$ or $l = 7, \min \{|C_1|, |C_2|\} > 4$}
            \State \Return $\varnothing$
        \EndIf
        \If {$l = 6, \max \{|C_1|, |C_2|\} \leq 3$ or $l = 7, \max \{|C_1|, |C_2|\} \leq 4$}
            \State \Return $C_1 + C_2$
        \EndIf
        \State \Return if $|C_1| < |C_2|$ then $C_1$ else $C_2$ 
    \end{algorithmic}
\end{algorithm}

We now check for all possible vertex separations $C$, which by contraction turns into a low-cut reduction, and obtain the set of vertices that could be reduced.
However, finding such a vertex separation $C$ is generally an exponential-time task, and we compromise by precalculating shortest distance paths.
The specific algorithm is shown below.

\begin{algorithm}[H]
    \caption{ReducableVertices($K$, $S$, $R$)}
    \label{alg:reducablevertices}
    \begin{algorithmic}[1]
        \Require {A C-reducible configuration $K$, its free completion $S$ and its ring $R$}
        \Ensure {Returns the set of vertices that may be reduced due to low-cut reductions}
        \State $d:$ a function, receives two vertices in $S$ and returns the distance between them
        \State $d':$ a function, receives two vertices in $S / c(K)$ and returns the distance between them
        \State $U \gets \varnothing$
        \ForAll {$r_p,r_q \in R$}
            \ForAll {$c$ in $\max \{1, 5 - d(r_p, r_q)\}, \ldots, 3 - d'(r_p, r_q)$}
                \State \texttt{hasSmallCut} $\gets$ \False
                \ForAll {$P \gets$ all paths of distance $d(r_p, r_q)$, connecting $r_p$ and $r_q$ in $S$}
                    \State $C_1, C_2 \gets$ the connected components of $S - P$
                    \State $x \gets \min \{|C_1|, |C_2|\}$
                    \If {$|P| + c = 5, x > 1$ or $|P| + c = 6, x > 3$ or $|P| + c = 7, x > 4$}
                        \State \texttt{hasSmallCut} $\gets$ \True
                        \Break
                    \EndIf
                \EndFor
                \If {\texttt{hasSmallCut} is \False}
                    \ForAll {$P \gets$ all paths where $|P \setminus c(K)| = d'(r_p, r_q)$, connecting $r_p$ and $r_q$ in $S$}
                        \State $C_1, C_2 \gets$ the connected components of $S - P$
                        \State $C \gets$ \text{GetLowCutReducable}($C_1$, $C_2$, $|P| + c$)
                        \State $U \gets U \cup V(C)$
                    \EndFor
                \EndIf
            \EndFor
        \EndFor
        \algstore{reducablevertices}
    \end{algorithmic}
\end{algorithm}
\begin{algorithm}[H]
    \begin{algorithmic}[1]
        \algrestore{reducablevertices}
        \ForAll {$r_p^1,r_q^1,r_p^2,r_q^2 \in R$}
            \If {$R - r_p^1 - r_q^1$ does not disconnect $r_p^2$, $r_q^2$, and $R - r_p^1 - r_p^2$ does not disconnect $r_q^1$, $r_q^2$}
                \ForAll {$c_1$ in $\max \{1, 5 - d(r_p^1, r_q^1)\}, \ldots, 3 - d'(r_p^1, r_q^1) $}
                    \ForAll {$c_2$ in $\max \{1, 5 - d(r_p^2, r_q^2)\} , \ldots, 3 - c_1 - d'(r_p^1,r_q^1) - d'(r_p^2,r_q^2)$}
                        \State \texttt{hasSmallCut} $\gets$ \False
                        \ForAll {$P \gets$ all paths of distance $d(r_p^1, r_q^1)$, connecting $r_p^1$ and $r_q^1$ in $S$}
                            \State $C_1, C_2 \gets$ the connected components of $S - P$
                            \State $x \gets \min \{|C_1|, |C_2|\}$
                            \If {$|P| + c = 5, x > 1$ or $|P| + c = 6, x > 3$ or $|P| + c = 7, x > 4$}
                                \State \texttt{hasSmallCut} $\gets$ \True
                                \Break
                            \EndIf
                        \EndFor
                        \ForAll {$P \gets$ all paths of distance $d(r_p^2, r_q^2)$, connecting $r_p^2$ and $r_q^2$ in $S$}
                            \State $C_1, C_2 \gets$ the connected components of $S - P$
                            \State $x \gets \min \{|C_1|, |C_2|\}$
                            \If {$|P| + c = 5, x > 1$ or $|P| + c = 6, x > 3$ or $|P| + c = 7, x > 4$}
                                \State \texttt{hasSmallCut} $\gets$ \True
                                \Break
                            \EndIf
                        \EndFor
                        \If {\texttt{hasSmallCut} is \False}
                            \ForAll {$P_1 \gets$ all paths where $|P_1 \setminus c(K)| = d'(r_p^1, r_p^2)$, connecting $r_p^1$ and $r_p^2$ in $S$}
                                \ForAll {$P_2 \gets$ all paths where $|P_2 \setminus c(K)| = d'(r_q^1, r_q^2)$, connecting $r_q^1$ and $r_q^2$ in $S$}
                                    \State $C_1, C_2, C_3 \gets$ the connected components of $S - P_1 - P_2$ 
                                    \State (Let $C_3$ be the one adjacent to both $P_1$ and $P_2$)
                                    \State $C \gets$ \text{GetLowCutReducable}($C_1 + C_2$, $C_3$, $|P_1| + |P_2| + c_1 + c_2$)
                                    \State $U \gets U \cup V(C)$
                                \EndFor
                            \EndFor
                        \EndIf
                    \EndFor
                \EndFor
            \EndIf
        \EndFor
        \State \Return $U$
    \end{algorithmic}
\end{algorithm}

Finally, we check the order and each of the vertices' degrees. 
We return a reason for $G'$ never to become $K_6$, or nothing if no reason could be found.
Note that we can only calculate the lower bound of the vertices touching the ring, and we must take this into account.

\begin{algorithm}[H]
    \caption{CheckConfiguration($K$)}
    \label{alg:checklargeconf}
    \begin{algorithmic}[1]
        \Require {A configuration $K$, the contraction edge set $c(K)$}
        \Ensure {Returns {\True} if there is some reason for $K$ to be safe, {\False} otherwise}
        \State $S,R \gets$ the free completion of $K$ and its ring
        \State $U \gets$ \text{ReducableVertices}($K$, $S$, $R$)
        \State $S \gets S - U$
        \State $S' \gets S/c(K)$
        \State $R' \gets $ set of vertices in $S'$ which the vertex was (a) in $R$ or (b) some vertices were contracted which contained one or more vertices in $R$ 
        \While {\True}
            \State \texttt{hasSeparation} $\gets$ \False
            \ForAll {$V_{\text{sep}} \subset V(S')$ such that $|V_{\text{sep}}| \leq 3$}
                \ForAll {$C\gets $ connected components of $S' - V_\text{sep}$}
                    \If {$C \subset S' - R'$}
                        $S' \gets S' - V(C)$
                        \State \texttt{hasSeparation} $\gets$ \True
                    \EndIf
                \EndFor
            \EndFor
            \If {\texttt{hasSeparation} is \False}
                \Break
            \EndIf
        \EndWhile
        \If {$|S'| > 6$}
            \State \Return \True
        \EndIf
        \ForAll {$v \in V(S')$}
            \If {$v \notin R'$}
                \If {the degree of $v$ in $S'$ is $4$}
                    \State \Return \True
                \EndIf
            \EndIf
            \If {the degree of $v$ in $S'$ is $6$ or more}
                \State \Return \True
            \EndIf
        \EndFor
        \State \Return \False
    \end{algorithmic}
\end{algorithm}

\section{Bigger implies smaller: Reducibility}
\label{sect:smaller-appendix}\showlabel{sect:smaller-appendix}

In this section, we show that the reducibility of some configurations is a consequence of the reducibility of other configurations. 

Let us first recall the basic notation used to define reducibility.

  Let $I$ be an island. We add a new incident edge to each vertex of degree 2 in $I$ and denote the set of all added edges by $R$. 
  Let $\mathscr{C}^{\ast}$ be the set of all colorings $\{1,\dots,|R|\} \rightarrow \{0,1,2\}$ that satisfies parity.
  Let $\mathscr{C_0}$ be the set of all $3$-edge-colorings of $R$ that are restrictions of $3$-edge-colorings of $I \cup R$. 
  Let $\mathscr{C}'$ be the maximal consistent subset of $\mathscr{C}^{\ast} - \mathscr{C_0}$. 
  The island $I$ is D-reducible if $\mathscr{C}' = \emptyset$. 
  Recall that a set $\mathscr{C}$ of colorings $\{1,\dots,k\} \rightarrow \{0,1,2\}$ that satisfy parity is \emph{consistent} if for every $\kappa \in \mathscr{C}$ and every $\theta \in \{0,1,2\}$ there is a projective signed matching $M$ such that $\kappa$ $\theta$-fits $M$, and $\mathscr{C}$ contains every coloring that $\theta$-fits $M$.

The C-reducibility is defined analogously, except that $\mathscr{C}^{\ast}$ is a smaller set, containing only those 3-edge-colorings of $R$ that arise from 3-edge-colorings of $I'\cup R$, see Definition \ref{dfn:Cred}. 

To test reducibility, one has to determine if maximal consistent subset of $\mathscr{C}^{\ast} - \mathscr{C_0}$ is empty or not. This is done via the following process.

Let $\mathscr{C}_0$ be as above, and for $i\ge1$, let $\mathscr{C}$ be all 3-edge-colorings of $R$ in $\mathscr{C}^\ast$ that are not in $\cup_{j=0}^{i-1}\mathscr{C}_j$. The we form $\mathscr{C}_i$ by taking all colorings $\kappa\in \mathscr{C}$ for which there is a $\theta\in \{0,1,2\}$ such that for every projective signed matching $M$ of $\{e\in R\mid \kappa(e)\ne\theta\}$ ($\kappa$ $\theta$-fits $M$), there is a coloring $\kappa'$ of $R$ that $\theta$-fits $M$ and belongs to $\cup_{j=0}^{i-1}\mathscr{C}_j$. We will say that colorings in $\mathscr{C}_i$ form level $i$ and that colorings at level 0 \emph{extend to $I$}. We have the following corollary of the construction that allows us testing if a configuration is reducible.

\begin{lem}
    Let\/ $t$ be the smallest integer such that $\mathscr{C}_t = \varnothing$. Then $\mathscr{C}' = \mathscr{C}^{\ast} - \cup_{i=0}^t\mathscr{C}_i$. 
\end{lem}

\begin{thm}\label{thm:bigger to smaller degree}
  Suppose that $I$ is an island that has a path $abcd$ on its outer face, where vertices $b$ and $c$ have degree $2$ in $I$. Let $J$ be the island obtained from $I$ by contracting the edge $bc$. If $I$ is reducible and in the case that it is C-reducible not both of the edges $ab$ and $cd$ are deleted, nor after deleting edges (to reduce) $bc$ is not isolated, then $J$ is also reducible.
\end{thm}

\begin{proof}
  Let us observe that the difference between D-reducibility and C-reducibility is only in the definition of $\mathscr{C}^\ast$. Here we are dealing with two configurations, $I$ and $J=I/bc$, and we will use $\mathscr{D}^\ast$, $\mathscr{D}_i$, and $\mathscr{D}'$ when referring to the corresponding sets of edge-colorings related to $J$.
  
  Our goal is to show that $J$ is reducible. Let $Q$ be the corresponding edge-cut in $J$. It contains an edge $q\in Q$ that is incident with the vertex $bc$ corresponding to the contracted edge $bc$. Note that the edge-cut of $I$ is $R = (Q\setminus \{q\}) \cup \{r_b,r_c\}$, where $r_b$ and $r_c$ are incident with $b$ and $c$ in $I$, respectively. Let $\kappa': Q\to \{0,1,2\}$ be a coloring that satisfies the parity condition, i.e. $\kappa'\in \mathscr{D}^\ast$. Then we define a coloring $\kappa:R\to \{0,1,2\}$ by having $\kappa$ to agree with $\kappa'$ on $R\setminus\{r_b,r_c\}$, and setting $\kappa(r_b)=\kappa'(q)+1 \pmod{3}$ and $\kappa(r_c)=\kappa'(q)+2 \pmod{3}$. Then it is clear that $\kappa\in \mathscr{C}^\ast$. By denoting this injective correspondence by $\varphi: \mathscr{D}^\ast \to \mathscr{C}^\ast$, we note
  $$
      \varphi(\mathscr{D}^\ast) = \{\kappa\in \mathscr{C}^\ast\mid \kappa(r_c)=\kappa(r_b)+1 \pmod 3\}. 
  $$ 
  
  We now claim that for each $i\ge0$, 
  \begin{equation}
    \mathscr{D}_0\cup \cdots \cup \mathscr{D}_i \supseteq \varphi^{-1}(\mathscr{C}_0\cup \cdots \cup \mathscr{C}_i).
    \label{eq:D_i}
  \end{equation}
  The proof will be by induction, and we first consider the base case, when $i=0$. Suppose that $\kappa'=\varphi^{-1}(\kappa)$, where $\kappa\in \mathscr{C}_0$. If $\alpha$ is a 3-edge-coloring of $I$ that is consistent with $\kappa$, then we have that $\alpha(ab) = \alpha(r_c)$ and $\alpha(cd) = \alpha(r_b)$. Since $\kappa(r_b)\ne \kappa(r_c)$, it is clear that this yields a 3-edge-coloring of $J$ extending $\kappa'$. Thus $\kappa'\in \mathscr{D}_0$. This proves the base of induction.
  
  For the inductive step, suppose that $\varphi(\kappa')=\kappa\in \mathscr{C}_i$. Let $M'$ be a signed projective matching that $\theta$-fits $\kappa'$. Let $M$ be the signed projective matching on $R$ that coincides with $M'$ on all elements that do not contain $q$. If there is an element $(\{q,r\},\sigma)\in M'$, then we replace it in $M$ with a match that is defined as follows. If $\kappa'(q)\ne \kappa'(r)$, then $\sigma=-1$ and $\kappa'(q)\ne \theta \ne \kappa'(r)$ since $M'$ $\theta$-fits $\kappa'$. In this case we add to $M$ the match $(\{r_b,r\},+1)$ or $(\{r_c,r\},+1)$, depending on whether $\kappa(r_b)= \kappa'(r)$ or $\kappa(r_c)= \kappa'(r)$, respectively. Clearly, $M$ $\theta$-fits $\kappa$. The other possibility is that $\sigma=+1$. Then $\kappa'(q)=\kappa'(r)\ne\theta$ since $M'$ $\theta$-fits $M'$. In this case, we add to $M$ the match $(\{r_b,r\},-1)$ or $(\{r_c,r\},-1)$, depending on whether $\kappa(r_c)=\theta$ or $\kappa(r_b)=\theta$, respectively. Then, again, $M$ $\theta$-fits $\kappa$. 
  
  Now, since $\kappa\in \mathscr{C}_i$ and $M$ $\theta$-fits $M$, there is a coloring $\bar\kappa$ 
  that $\theta$-fits $M$ and belongs to $\mathscr{C}_0\cup \cdots \cup \mathscr{C}_{i-1}$. By the induction hypothesis, $\varphi^{-1}(\bar\kappa) \in \mathscr{D}_0\cup \cdots \cup \mathscr{D}_{i-1}$. Since $M'$ was arbitrary, this proves that $\kappa'\in \mathscr{D}_i$, and hence we get (\ref{eq:D_i}).

  Since $I$ is reducible, we have that $\mathscr{C}^\ast= \mathscr{C}_0\cup \cdots \cup \mathscr{C}_i$ for some $i$. Then we conclude that $\mathscr{D}^\ast = \mathscr{D}_0\cup \cdots \cup \mathscr{D}_i$.   
\end{proof}


The situation of the island in Theorem \ref{thm:bigger to smaller degree} described on the corresponding configurations $K(I)$ and $K(J)$ is as follows. We start with the assumption that a vertex $v$ in $K(I)$ has degree $d$ and $\gamma_{K(I)}(v)=d+3$. The corresponding configuration $K(J)$ is isomorphic to $K(I)$ except that $\gamma_{K(J)}(v)=d+2$. The condition that possible C-reducibility should not delete both $ab$ and $cd$ is saying that two nonconsecutive edges incident with $v$ that lead to the ring of the free completion should not both be contracted in the triangulation. There are examples that show that this condition is necessary.

\begin{coro}\label{cor:bigger to smaller degree}
  Suppose that $K$ is a configuration in $G'$ and $v\in V(K)$ is such that $\gamma_K(v) = d_K(v)+3$. Let $x,y,z$ be the neighbors of $v$ on the ring of the free completion of $K$ and suppose that $K$ is reducible such that in the case of C-reducibility not both of the edges $vx$ and $vz$ are contracted nor $x, z$ are identified. Then the configuration $K'$ obtained from $K$ by changing $\gamma_K(v)$ to $\gamma_K(v)-1$ is also reducible.
\end{coro}

Theorem \ref{thm:bigger to smaller degree} and Corollary \ref{cor:bigger to smaller degree} do not work when $\gamma_K(v) = d_K(v)+2$. They can be proved when $\gamma_K(v) = d_K(v)+1$, but we do not need this case.

\subsection{Bigger implies smaller: revisited}\label{subsect:imply1}\showlabel{subsect:imply1}

\emph{Proof of Claim \ref{clm:all-imply}}

Let $v$ be a vertex of $V_0$. The symbol $r_1, r_2, r_3$ denotes three vertices of the ring of $K$ that are adjacent to $v$ in the clockwise order listed along the ring.
A set $V_1 \subseteq V_0$ consists of a vertex $v$ such that the edge $vr_2$ is in $c(K)$.
A set $V_2 \subseteq V_0$ consists of a vertex $v$ such that $r_1, r_3$ is identified by the contraction $c(K)$. 
Note that vertices of $V_2$ do not satisfy the condition of contractions needed for applying Corollary \ref{cor:bigger to smaller degree}. Also, we need attention when $c(K)$ equals $\{vr_2 \mid v \in V_1\}$ and $S \supseteq V_1$. We can apply Corollary \ref{cor:bigger to smaller degree}, but we cannot get a smaller graph than a minimal counterexample by contracting edges of $c(K)$. Hence, we fail to apply induction. This observation leads to the choice of $\mathcal{K}_{K, \text{imply1}}$ defined below.

A set $\mathcal{K}_{K, \text{imply1}}$ consists of $K_S$ for all $S \cap (V_1 \cup V_2) = \emptyset$. A configuration in $\mathcal{K}_{K, \text{imply1}}$ is shown to be reducible by Corollary \ref{cor:bigger to smaller degree}.
We have to check the reducibility of $K_S$, which does not belong to $\mathcal{K}_{K, \text{imply1}}$, but
the number of such configurations for all $K \in \mathcal{K}$ is more than 7000, which is too large.
Hence, we devise the following strategy.
A set $\mathcal{K}_{K,1}$ consists of $K_S$ for all $S \subseteq V_1 \cup V_2$.
The reducibility of $K_S$ in $\mathcal{K}_{K,1}$ can imply $K_{S \cup T}$ for all $T \subseteq V_0 \setminus (V_1 \cup V_2)$ if the constraint of the contraction is satisfied, so we check the reducibility of $K_S$ in $\mathcal{K}_{K,1}$ and check whether or not the contraction of $K_S$ is valid for applying Corollary \ref{cor:bigger to smaller degree}. All configurations in $\mathcal{K}_{K,1}$ are reducible.
A set $\mathcal{K}_{K,\text{imply2}}$ consists of $K_{S \cup T}$ for all $T \subseteq V_0 \setminus (V_1 \cup V_2)$ such that the contraction of $K_S$ is valid.
A set $\mathcal{K}_{K, 2}$ consists of $K_{S \cup T}$ for all $T \subseteq V_0 \setminus (V_1 \cup V_2)$ such that the contraction of $K_S$ is not valid.
We check the reducibility of all configurations in $\mathcal{K}_{K, 2}$. All configurations in $\mathcal{K}_{K, 2}$ are reducible.

A set $\mathcal{K}_{\text{imply1}}, \mathcal{K}_{\text{imply2}}, \mathcal{K}_1, \mathcal{K}_2$ is the union of $\mathcal{K}_{K, \text{imply1}}, \mathcal{K}_{K, \text{imply2}}, \mathcal{K}_{K, 1}, \mathcal{K}_{K, 2}$ for all $K \in \mathcal{K}$ respectively.
We use several computer checks to make sure that a configuration $K$ does not result in $K_6$ after contraction and low-vertex-cut reductions in the previous section of this paper (e.g. Section \ref{sect:ccheck}, Section \ref{sect:largecont}). When doing so, we check all configurations in $\mathcal{K}, \mathcal{K}_1, \mathcal{K}_2$.
In addition, we have to check a configuration in $\mathcal{K}_{\text{imply1}}, \mathcal{K}_{\text{imply2}}$.
A set $\mathcal{K}_{\text{imply1}}^{\ast} \subseteq \mathcal{K}_{\text{imply1}}$ is the union of $\mathcal{K}_{K, \text{imply1}}$ such that the contraction size of $K$ is at least $5$.
For $L \in   \mathcal{K}_{\text{imply1}}^{\ast}$, we check the reducibility of $L$ and get contraction of $L$. The configuration $L$ is C-reducible, but we need contraction of $L$ to check whether or not the resulting graph results in $K_6$.
Otherwise, we do not have to check $L$ since the contraction size of $L$ is at most 4 and hence Lemma \ref{lem:cont4} holds for this $L$.
We do the same thing to $\mathcal{K}_{\text{imply2}}$. In other words, we generate a set $\mathcal{K}_{\text{imply2}}^{\ast}$, and check whether or not a configuration in $\mathcal{K}_{\text{imply2}}^{\ast}$ results in $K_6$, as above. By computer-check in Section \ref{sect:ccheck}, \ref{sect:largecont}, several configuration remains in $\mathcal{K}_1, \mathcal{K}_2$. However, they contain conf(1), since we do not check if a configuration in $\mathcal{K}_1, \mathcal{K}_2, \mathcal{K}_{\text{imply1}}, \mathcal{K}_{\text{imply2}}$ contains a configuration in $\mathcal{K}$ as a subgraph. The configuration conf(1) is already handled in Lemma \ref{lem:5555}, so we do not need to handle them.

Also, we have to check all configurations in $\mathcal{K}_{\text{imply1}}, \mathcal{K}_{\text{imply2}}, \mathcal{K}_1, \mathcal{K}_2$ to prove Lemma \ref{pairs6}. Almost all configurations that have diameter five in $\mathcal{K}$ have a cutvertex, but we have to check a few configurations that remain.
We handle this issue by checking $K \in \mathcal{K}$ with a stronger constraint. The length of the path used in the proof of Lemma \ref{pairs6} does not change. We have to be careful to count the number of vertices. We use a different method from the method described in Lemma \ref{pairs6}.
We calculate the number of the vertices of the ring by choosing one vertex of $K$, which is denoted as $u$, and then counting vertices in the ring that are adjacent to $u$, and subtracting 2 from the value when the counting value is 3, subtracting 1  otherwise for safety. The subtraction is needed to ensure a vertex that has a distance of five may incident with one less ring vertices. When the counting value is 3, $u$ may also be in the same situation, so the subtraction value is 2. By this computer check, one configuration remains, which is denoted as $K$. We handle this one case by generating $K_S$ for all $S \subseteq V_0$, and checking $K_S$ with an original constraint described in the proof of Lemma \ref{pairs6}.

We summarize the argument here in the table below, and we conclude that Claim \ref{clm:all-imply} holds. 

\begin{table}[htbp]
    \centering
    \caption{The number of configurations that need the contraction size $K$ in $\mathcal{K}_1, \mathcal{K}_2, \mathcal{K}_{\text{imply1}}^{\ast}, \mathcal{K}_{\text{imply2}}^{\ast}$ }
    \begin{tabular}{ccccccccccc}
        \toprule
         contraction size & 0 & 1 & 2 & 3 & 4 & 5 & 6 & 7 & $8_{\geq}$ & total   \\
        \midrule
         \# of configurations in $\mathcal{K}_1$ & 80 & 315 & 33 & 51 & 17 & 16 & 7 & 2 & 0 & 521 \\
         \# of configurations in $\mathcal{K}_2$ & 100 & 63 & 1 & 6 & 2 & 0 & 3 & 0 & 0 & 175 \\
         \# of configurations in $\mathcal{K}_{\text{imply1}}^{\ast}$ & 185 & 220 & 11 & 22 & 4 & 7 & 4 & 0 & 0 & 453 \\
         \# of configurations in $\mathcal{K}_{\text{imply2}}^{\ast}$ & 51 & 21 & 6 & 3 & 1 & 0 & 2 & 0 & 0 & 84 \\
        \bottomrule
    \end{tabular}
\end{table}

\section{Small edge-cuts in minimal counterexamples}
\label{sect:below5cuts}\showlabel{sect:below5cuts}

In this section, we show a proof for coloring a projective graph $G$ with a valid edge-cut $F \in E(G)$ which satisfies one of the following:

\begin{itemize}
    \item $2 \leq |F| \leq 4$
    \item $|F| = 5$, and neither of the two connected components of $G-F$ is a 5-cycle.
\end{itemize}

We introduce a proof for a broader theorem, speaking of islands embedded in any surface including the projective plane.

\begin{thm}
    Let $I$ be an island that is embeddable to the plane so that all vertices of degree 2 are adjacent to the outerplane.
    (We call such islands \emph{disk islands}.)
    Let $F$ be its ring.
    $I$ is reducible under a class of embeddable graphs on any surface if one of the following holds.
    \begin{itemize}
        \item $2 \leq |F| \leq 4$
        \item $|F| = 5$ and $|V(I)| > 5$
    \end{itemize}
    \label{thm-disk-cut}
\end{thm}

\begin{proof}
    Due to the smallness of $F$, there are at most only two Kempe chains that contain the edges in $F$.
    Therefore, the case where the outerplane contains one crosscap (i.e. the surface is a projective plane) is enough for the proof.

    When $|F| = 2,3$, all possible ring colorings that satisfy the parity condition are equivalent.
    So, $\mathcal{C}(I)$ must be equal to $\varnothing$ or $\mathcal{C}^{\text{valid}}$.
    If $|F| = 2$, let $G$ be a cubic graph obtained by connecting the two vertices in $I$ which are of degree two with a new edge, and if $|F| = 3$, let $G$ be a cubic graph obtained by identifying the three vertices in $I+F$ which are degree one.
    In either case, $G$ is planar and by 4CT this graph is colorable.
    Therefore, $\mathcal{C}(I) = \mathcal{C}^{\text{valid}}$, and $I$ is reducible.

    Next, when $|F| = 4$, let us name the four vertices of $I$ having degree 2 in clockwise order as $v_1, v_2, v_3, v_4$.
    We construct six cubic graphs from $I$ as follows. ($u,w$ are new vertices.)
    
    \begin{itemize}
        \item $G^A_{12} = I + v_1v_2 + v_3v_4$
        \item $G^A_{13} = I + v_1v_3 + v_2v_4$
        \item $G^A_{14} = I + v_1v_4 + v_2v_3$
        \item $G^B_{12} = I + v_1u + v_2u + v_3w + v_4w + uw$
        \item $G^B_{13} = I + v_1u + v_3u + v_2w + v_4w + uw$
        \item $G^B_{14} = I + v_1u + v_4u + v_2w + v_3w + uw$
    \end{itemize}

    $G^A_{12}, G^A_{14}, G^B_{12}, G^B_{14}$ are planar and $G^A_{13}, G^B_{13}$ are single-cross, so they are all colorable.
    The multiset expression of the restriction of a coloring of $G^A_{ij}$ to $F$ must be $\{\{i,j,k,l\},\varnothing,\varnothing\}$ or $\{\{i,j\},\{k,l\},\varnothing\}$.
    The multiset expression of the restriction of a coloring of $G^B_{ij}$ to $F$ must be $\{\{i,k\},\{j,l\},\varnothing\}$ or $\{\{i,l\},\{j,k\},\varnothing\}$.
    Note that there are four unique (in the sense that they are pairwise nonequivalent) valid $F$-colorings.
    Since the intersection of $\mathcal{C}(I)$ and the coloring pair above must not be empty, $\mathcal{C}(I)$ must contain at least three unique elements.
    The remaining unique $F$-coloring must necessarily be a ring coloring of $I$, and thus $\mathcal{C}^*(I) = \mathcal{C}^{\text{valid}}$.
    Therefore, $I$ is reducible.

    Finally, when $|F| = 5$, we construct 21 graphs as follows.

    \begin{itemize}
        \item $G^C_{ij} = I + v_iv_j + v_ku + v_lu + v_mu$ ($\{i,j,k,l,m\} = \{1,2,3,4,5\}$, $i < j$)
        \item $G^D_{ij, kl} = I + v_iu + v_ju + v_kw + v_lw + uv + wv + v_mv$ ($\{i,j,k,l,m\} = \{1,2,3,4,5\}$, $i < j$, $i < k$)
        \item $G^E = I + v_1u + v_2v + v_3w+ v_4x + v_5y + uv + vw + wx + xy + yu$
    \end{itemize}

    All of $G^C_{ij}$ and $G^D_{ij,kl}$ and $G^E$ are planar, apex, or singlecross, so every graph here is colorable.
    Note that there are ten unique valid $F$ colorings and that the corresponding multiset for them must be in the form of $\{\{i\}, \{j\}, \{k,l,m\}\}$. 
    Let us denote it by $\phi^5_{ij}$.
    The restriction of a coloring of $G^C_{ij}$ must be one of $\phi^5_{kl}, \phi^5_{km}, \phi^5_{lm}$.
    The restriction of a coloring of $G^D_{ij,kl}$ must be one of $\phi^5_{ik}, \phi^5_{il}, \phi^5_{jk}, \phi^5_{jl}$.
    The restriction of a coloring of $G^E_{ij,kl}$ must be one of $\phi_{12}, \phi_{23}, \phi_{34}, \phi_{45}, \phi_{15}$.

    Using Kempe chains, the following can be said.

    \begin{proposition}
        Let some coloring $\phi$ exist for $I$, and let $\phi^5_{ij}$ be its ring coloring.
        Then, there exists a $k \notin \{i,j\}$ such that $\phi^5_{ik}$ is a ring coloring of $I$.
    \end{proposition}

    \begin{proof}[Subproof]
        We fix a color $\kappa = \phi^5_{ij}(f_i)$ and obtain a Kempe chain in $I$.
        We Kempe change the coloring using the Kempe chain connecting $f_j$ and some other ring edge, which we call $f_k$.
        Then, the newly created coloring of $I$ has a ring coloring $\phi^5_{ik}$.
    \end{proof}

    The above three conditions, together with the proposition, imply that $\mathcal{C}(I)$ contains $\phi^5_{12}, \phi^5_{23}, \phi^5_{34}, \phi^5_{45}, \phi^5_{15}$.
    
    Let $J$ be a 5-cycle embedded in the plane. $J$ is an island and $\mathcal{C}(J)$ is equal to the above set. Since $I$ must be strictly larger than $J$, $I$ is C-reducible.
\end{proof}

\section{Reducibility of projective islands in $\Pi_y^k$}\label{sec:arp1}\showlabel{sec:arp1}

We have checked reducibility on the plane for islands in $\Pi_{y}^k$ with various $y,k$, and the results of each graph family that we use in our proof are shown in Table \ref{tab2}.
Note that the reducibility check is done with planar consistency because if an island $I \in \Pi_{y}^k$ appears in $G$, the crosscap must be contained in $I$, and the outer parts would be planar. 
In Table \ref{tab2}, $c(I) \leq 4$ for almost all C-reducible configurations, except for a few.  
\begin{center}\label{tab2}
    \begin{tabular}{ c|ccc }
        $\Pi_{y}^k$ & \#D-reducible & \#C-reducible & \#Non-reducible \\
        \hline
        $\Pi_{3}^6$ & 5 & 9 & \\
        $\Pi_{3}^7$ & 4 & 23 ($c(I) \geq 4$ for 3 configurations) & \\
        $\Pi_{3}^8$ & 6 & 44 ($c(I) \geq 5$ for 2 configurations) & \\
        $\Pi_{4}^6$ & 2 & & \\
        $\Pi_{4}^7$ & 8 & & \\
        $\Pi_{4}^8$ & 29 & 1 & \\
        $\Pi_{5}^8$ & 2 & & \\
        $\Pi_{5}^9$ & 16 & 1 & \\
        $\Pi_{5}^{10}$ & 61 & 17 & \\
        $\Pi_{5}^{11}$ & 134 & 130 & \\
        $\Pi_{5}^{12}$ & 179 & 564 & \\
        $\Pi_{5}^{13}$ & 115 & 1699 ($c(I) \geq 5$ for 158 configurations) & 6\\
    \end{tabular}
\end{center}

\section{The remaining unsafe configurations}
\label{subsec:remaining1}\showlabel{subsec:remaining1}

We now give the details for the remaining 10 configurations that require further analysis in Subsection \ref{subsec:remaining}. 

For all cases, let $G'$ be the graph in which the configuration $W$ appears, and let $S$ be the free completion of $W$.
We assign a number for each vertex of $S$, and we call each vertex $i$ as $v_i$.
The corresponding numbers are shown in the image.

\begin{figure}[H]
  \centering
  \includegraphics[width=5cm]{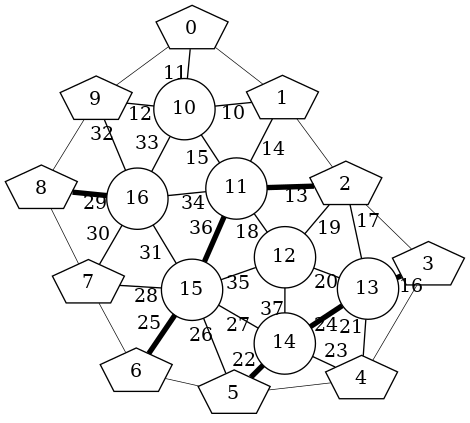}
  \label{fig:proj0103}
\end{figure}
\begin{proof}
    The two vertice sets $\{v_{10}\}$ and $\{v_5, v_{14}, v_{13}, v_3\}$ will each become a single vertex after contraction, and will not be removed by 2-vertex cuts nor 3-vertex cuts.
    They must be connected for them to be a vertex in $K_6$, and for that we need at least two of the ring vertices in $\{v_0, v_1, v_3, v_5, v_9\}$ be connected before contraction. 
    However, this will produce a 4-cycle, which we can resolve by low representativity.
\end{proof}

\begin{figure}[H]
  \centering
  \includegraphics[width=5cm]{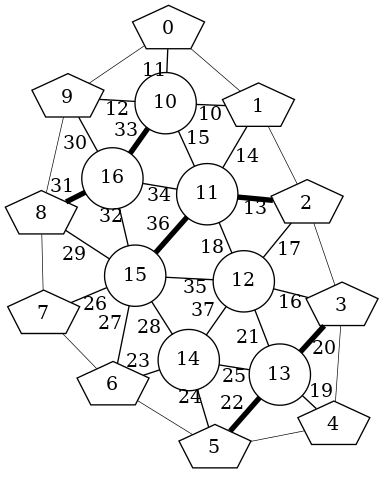}
  \label{fig:proj0131}
\end{figure}
\begin{proof}
    A similar proof as above can be constructed.
    The two vertice sets $\{v_{10}, v_{16}, v_8\}$ and $\{v_5, v_{13}, v_3\}$ will each become a single vertex after contraction.
    They must be connected, but we need $v_8$ to be connected to either $v_3$ or $v_5$, each of which can be resolved by low representativity.
\end{proof}

\begin{figure}[H]
  \centering
  \includegraphics[width=5cm]{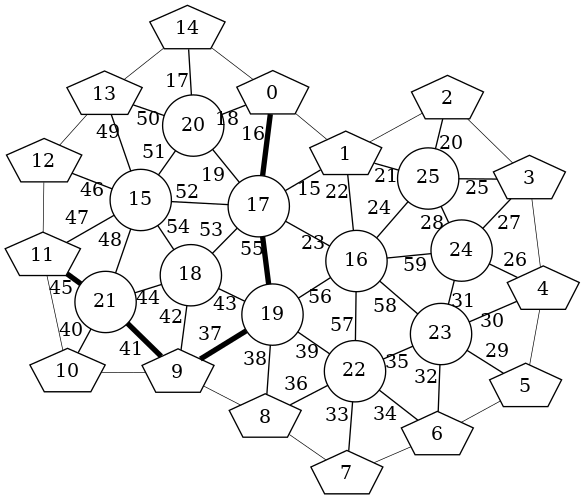}
  \label{fig:list2597}
\end{figure}
\begin{proof}
    The only dangerous outer cuts occur when $v_0$ and $v_{11}$ have a path of length 3.
    This creates a dangerous 8-cut which will become a 3-cut after contraction.
    In this case, there is a shorter 7-cut via $v_0, v_{17}, v_{18}, v_{21}, v_{11}$, and both separated components must contain at least $5$ vertices (the smaller one containing $v_{12}, v_{13}, v_{14}, v_{15}, v_{20}$).     
    Therefore, $v_{16}$ remains after reduction, and this must become degree 6.
\end{proof}

\begin{figure}[H]
  \centering
  \includegraphics[width=5cm]{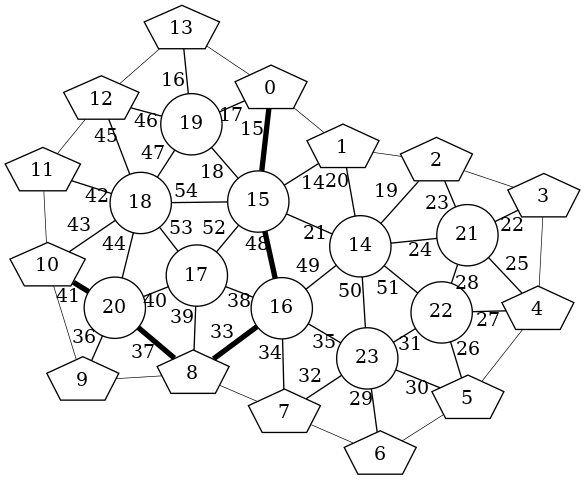}
  \label{fig:milan0432}
\end{figure}
\begin{proof}
    The only dangerous outer cuts occur when $v_0$ and $v_{10}$ are connected with a path of length 3.
    This creates a dangerous 8-cut which will become a 3-cut after contraction.
    In this case, there is a shorter 7-cut via $v_0$, $v_{15}$, $v_{17}$, $v_{20}$, $v_{10}$, and both separated components must contain at least $5$ vertices (the smaller one containing $v_{11}$, $v_{12}$, $v_{13}$, $v_{18}$, $v_{19}$).     
    Therefore, $v_{23}$ remains after reduction, and this must become degree 6.
\end{proof}

\begin{figure}[H]
  \centering
  \includegraphics[width=5cm]{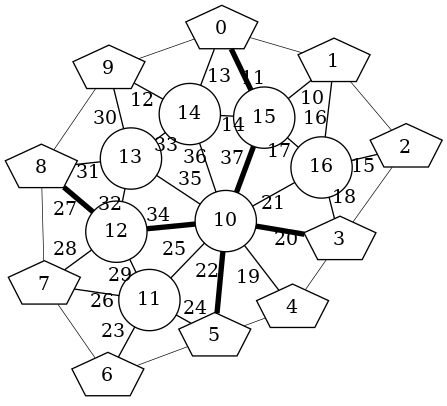}
  \label{fig:proj4560}
\end{figure}

\begin{proof}
    The two vertices $v_1$ and $v_7$ will not be removed by 2-vertex cuts nor 3-vertex cuts.
    These two must be connected in order to form $K_6$ after contraction, and this can be resolved by low representativity.
\end{proof}

%
\section{Strengthening Tutte's 4-Flow Conjecture for projective planar graphs}

An important consequence of this paper is that the Tutte 4-Flow Conjecture can be strengthened for projective planar graphs. This is a direct consequence of our main result, Theorem \ref{mainth} when we consider cubic graphs. In this section we prove that a minimum counterexample would need to be cubic, thus reducing the full conjecture to cubic graphs.

\begin{figure}[htb]
  \centering
  \includegraphics[width=5cm]{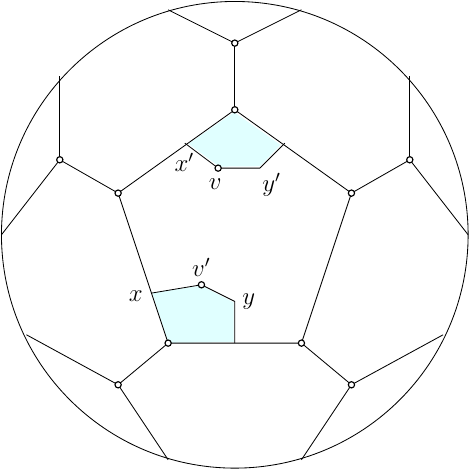}
  \caption{The graph $G_1$ after splitting off the edges $vx,vy$ could become Petersen-like.}
  \label{fig:G1 is Petersen-like}
\end{figure}

Let us extend the notion of a \emph{minimum counterexample} $G$ to be a 2-edge-connected projective planar graph $G$ with minimum $|V(G)|+|E(G)|$ that is not Petersen-like and does not admit a nowhere-zero 4-flow.

\begin{lem}
    A minimum counterexample is a 3-connected cubic graph.
\end{lem}

\begin{proof}
    Suppose that $G$ is a minimum counterexample that is not cubic. If $G$ is not 2-connected, then $G=G_1\cup G_2$ where $G_1\cap G_2$ is a cutvertex. Since $G$ has no cutedges, both $G_1$ and $G_2$ are 2-edge-connected. If they both admit 4-flows, so does $G$. So, one of them must be Petersen-like (as otherwise we would contradict that $G$ is a minimum counterexample). However, in that case, the other one must be planar, and it it easy to see that $G$ itself would be Petersen-like.

    Similarly, if $G$ has a 2-vertex-separator $\{x,y\}$. In that case, $G$ would be a 2-sum of two graphs $G_1$ and $G_2$, both containing the edge $e=xy$. Let $\phi_i$ be a nowhere-zero 4-flow in $G_i$ for $i=1,2$. If $e\notin E(G)$, then we change $\phi_2$ so that $\phi_2(xy)=-\phi_1(xy)$. Then it is clear that $\phi_1+\phi_2$ gives rise to a nowhere-zero-flow in $G$. On the other hand, if $e\in E(G)$, then we change $\phi_2$ so that $\phi_2(xy)\ne -\phi_1(xy)$. Then, again, $\phi_1+\phi_2$ gives rise to a nowhere-zero-flow in $G$. 
    
    This leaves us with the possibility that $G_1$ (say) is Petersen-like. In that case, $G_2$ must be planar. Since $G$ itself is not Petersen-like, the edge $xy\in E(G_1)$ is one of the 15 edges confirming that $G_1$ is Petersen-like. In that case, we let $\phi_1$ be a nowhere-zero 4-flow of $G_1-xy$ (which is 2-edge-connected since otherwise it can be shown that $G$ would be Petersen-like). Now, taking a nowhere-zero 4-flow $\phi_2$ of $G_2$ (or of $G_2-xy$ if $xy\notin E(G)$), the sum $\phi_1+\phi_2$ is a nowhere-zero 4-flow of $G$, a contradiction.
    
    This confirms that $G$ is 3-connected. If it has a vertex $v$ of degree more than 3, consider splitting off a pair of consecutive edges $vx,vy$ in the clockwise rotation around $v$. The resulting graph $H$ is 2-connected since $G$ is 3-connected. A flow in $H$ gives a flow in $G$. As $H$ is smaller (it has fewer edges than $G$), it must be Petersen-like. Let $v'$ be the vertex of degree 2 with neighbors $x,y$ obtained from $v$ after splitting. Since the edges $vx,vy$ are consecutive around $v$, the vertices $v$ and $v'$ are on the same face in the projective plane embedding of $H$, and since $G$ is not Petersen-like, they belong to distinct graphs when building $H$ from the Petersen graph. See Figure \ref{fig:G1 is Petersen-like}. Suppose that the consecutive neighbors around $v$ are $x',x,y,y',\dots$. Now, it is easy to see that we could have split off the edges $vx',vx$ or the edges $vy,vy'$ and then $G_1$ would not be Petersen-like, thus yielding a contradiction.

    The above case shows that the minimum counterexample is cubic. This completes the proof.    
\end{proof}

\end{document}